\documentclass[11pt]{amsart}
\usepackage{hyperref,pb-diagram,amssymb,epic,eepic,verbatim,graphicx,graphics,epsfig,psfrag}
\hypersetup{colorlinks}
\usepackage{color}
\definecolor{darkred}{rgb}{0.5,0,0}
\definecolor{darkgreen}{rgb}{0,0.5,0}
\definecolor{darkblue}{rgb}{0,0,0.5}
\hypersetup{ colorlinks,
linkcolor=darkblue,
filecolor=darkgreen,
urlcolor=darkred,
citecolor=darkblue }
\usepackage{amssymb,verbatim,color,graphicx,pb-diagram}
\usepackage{graphics,epsfig,psfrag,mathrsfs} 

\newtheorem{theorem}{Theorem}[section]
\newtheorem{corollary}[theorem]{Corollary}

\newtheorem{proposition}[theorem]{Proposition}
\newtheorem{lemma}[theorem]{Lemma}
\newtheorem{lem}[theorem]{}
\theoremstyle{definition}
\newtheorem{definition}[theorem]{Definition}
\theoremstyle{remark}
\newtheorem{remark}[theorem]{Remark}
\newtheorem{example}[theorem]{Example}
\newcommand{\blem}{\begin{lem} \rm}
\newcommand{\elem}{\end{lem}}
\newcommand\B{\mathcal{B}}
\newcommand\E{\mathcal{E}}
\newcommand\M{\mathcal{M}}
\newcommand\D{\mathcal{D}}
\newcommand\cH{\mathcal{H}}

\newcommand\cG{\mathcal{G}}

\renewcommand\M{\mathcal{M}}

\newcommand\cS{\mathscr{S}}
\newcommand\cB{\mathscr{B}}
\newcommand{\K}{\mathcal{K}}

\newcommand{\T}{\mathcal{T}}
\newcommand{\J}{\mathcal{J}}

\newcommand{\U}{\mathcal{U}}

\newcommand{\F}{\mathcal{F}}
\newcommand{\N}{\mathbb{N}}
\newcommand{\R}{\mathbb{R}}
\renewcommand{\H}{\mathbb{H}}
\newcommand{\bH}{\mathbb{H}}
\newcommand{\RR}{\mathcal{R}}
\renewcommand{\SS}{\mathcal{S}}
\newcommand{\C}{\mathbb{C}}
\newcommand{\cC}{\mathcal{C}}

\newcommand{\Z}{\mathbb{Z}}

\newcommand{\ddt}{\frac{d}{dt}}

\renewcommand{\P}{\mathbb{P}}
\newcommand{\PP}{\mathcal{P}}

\newcommand{\on}{\operatorname}
\newcommand{\ainfty}{{$A_\infty$\ }}

\newcommand{\olp}{\ol{\partial}}

\newcommand\pre{{\on{pre}}}

\newcommand{\dual}{\vee}

\newcommand{\Ham}{\on{Ham}}

\newcommand{\GG}{\mathcal{G}}

\renewcommand{\top}{{\on{top}}}
\newcommand{\Fun}{\on{Func}}
\newcommand{\Obj}{\on{Obj}}
\newcommand{\Edge}{\on{Edge}}

\newcommand{\Symp}{\on{Symp}}

\newcommand{\Lag}{\on{Lag}}

\newcommand{\Ver}{\on{Vert}}

\newcommand{\radius}{\on{radius}}

\renewcommand{\det}{\on{det}}

\newcommand{\Aut}{ \on{Aut} } 
 
\newcommand{\aut}{ \on{aut} }

\newcommand{\Hom}{ \on{Hom}}
\newcommand{\Ind}{ \on{Ind}}

\newcommand{\coker}{ \on{coker}}

\newcommand{\ind}{ \on{ind}}

\newcommand{\UU}{ \mathcal{U}}

\newcommand{\codim}{\on{codim}}

\newcommand\dirac{/\kern-1.2ex\partial} 
\newcommand\qu{/\kern-.7ex/} 
\newcommand\lqu{\backslash \kern-.7ex \backslash} 

\newcommand\dr{r_+ \kern-.7ex - \kern-.7ex r_-}
 
\def\pd{\partial}

\renewcommand{\comment}[1]   {{}}
\newcommand{\labell}\label

\newcommand{\lra}{\longrightarrow}

\renewcommand{\d}{{\mbox{d}}}
\newcommand{\ol}{\overline}

\newcommand\eps{\epsilon}

\newcommand{\del}{\delta}
\newcommand{\f}{\frac}
\newcommand{\lan}{\langle}
\newcommand{\ran}{\rangle}
\newcommand{\hh}{{\f{1}{2}}}

\newcommand{\ti}{\tilde}

\newcommand\pt{\on{pt}}
\newcommand\cE{\mathcal{E}}

\newcommand\cF{\mathcal{F}}

\newcommand\mE{\mathcal{E}}

\newcommand\Map{\on{Map}}

\newcommand\Vect{\on{Vect}}
\newcommand\ul{\underline}

\renewcommand\H{\mathcal{H}}
\renewcommand\Im{\on{Im}}

\newcommand\reg{{\on{reg}}}
\newcommand\bra[1]{ \lan {#1} \ran} 
\newcommand\bdefn{\begin{definition}}
\newcommand\edefn{\end{definition}}
\newcommand\bea{\begin{eqnarray*}}
\newcommand\eea{\end{eqnarray*}}
\newcommand\bcv{\left[ \begin{array}{r} }
\newcommand\ecv{\end{array} \right] }

\newcommand\bma{\left[ \begin{array} }
\newcommand\ema{\end{array} \right]}
\newcommand\ben{\begin{enumerate}}
\newcommand\een{\end{enumerate}}
\newcommand\beq{\begin{equation}}
\newcommand\eeq{\end{equation}}
\newcommand\bex{\begin{example}}
\newcommand\bsj{\left\{ \begin{array}{rrr} }
\newcommand\esj{\end{array} \right\}}

\newcommand\Cone{\on{Cone}}
\newcommand\Id{\on{Id}}
\newcommand\cI{\mathcal{I}}

\newcommand\eex{\end{example}}

\newcommand\sx{*\kern-.5ex_X}

\def\mathunderaccent#1{\let\theaccent#1\mathpalette\putaccentunder}
\def\putaccentunder#1#2{\oalign{$#1#2$\crcr\hidewidth \vbox
to.2ex{\hbox{$#1\theaccent{}$}\vss}\hidewidth}}

\renewcommand\sharp{\includegraphics[height=.07in]{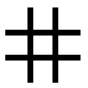}}

\newcommand{\CC}{\mathcal{C}}
\newcommand{\DD}{\mathcal{D}}
\newcommand{\HH}{\mathcal{H}}

\newcommand{\TT}{\mathcal{T}}

\newcommand\GFuk{{\on{Fuk}^{\sharp}}}
\newcommand\Fuk{{\on{Fuk}}}

\begin{document}
%
%
%
%
%

\title[\ainfty functors for Lagrangian correspondences] 
{\ainfty functors for Lagrangian correspondences} 

\author{S. Ma'u}

\author{K. Wehrheim}

\author{C. Woodward}

\begin{abstract}
We construct \ainfty functors between Fukaya categories associated to
monotone Lagrangian correspondences between compact symplectic
manifolds.  We then show that the composition of \ainfty functors for
correspondences is homotopic to the functor for the composition, in
the case that the composition is smooth and embedded.
\end{abstract}

\maketitle

\setcounter{tocdepth}{1}
\tableofcontents

\section{Introduction}  

Recall that to any compact symplectic manifold $(M,\omega)$ satisfying
suitable monotonicity conditions is a {\em Fukaya
  category} \label{cathere}
$\Fuk(M)$
whose objects are Lagrangian submanifolds, morphism spaces are
Lagrangian Floer cochain groups, and composition maps count
pseudoholomorphic polygons with boundary in a given sequence of
Lagrangians \cite{fuk:garc}.  A construction of Kontsevich
\cite{kon:hom} constructs a triangulated {\em derived Fukaya category}
which is related via the homological mirror symmetry conjecture to the
derived category of bounded complexes of coherent sheaves.  The latter
admits natural {\em Mukai functors} associated to correspondences
which play an important role in, for example, the McKay correspondence
\cite{re:mckay}, the work of Nakajima \cite{nak:quiv}, etc.

The main result of this paper constructs \ainfty functors associated
to monotone Lagrangian correspondences which are meant to be mirror
analogs to the Mukai functors.  We learned the idea of constructing
functors associated to Lagrangian correspondences from Fukaya, who
suggested an approach using duality.  In his construction the functor
maps a Fukaya category of the domain of the correspondence to the dual
of the codomain..  This makes composition \label{comphere} of functors problematic; the
approach here avoids that problem by enlarging the Fukaya category.
The results of this paper are chain-level versions of an earlier paper
\cite{we:co} in which the second two authors constructed
cohomology-level functors between categories for Lagrangian
correspondences.  We also showed in \cite{ww:isom} that composition of
these functors agrees with the geometric composition in the case that
the Lagrangian correspondences have embedded composition.
Applications of the calculus of \ainfty functors developed in this
paper can be found in Abouzaid and Smith \cite{ab:hms4torus} and Smith
\cite{smith:pq}, as well as in Wehrheim-Woodward \cite{fielda},
\cite{fieldb}, \cite{wo:ex}.

\subsection{Summary of results}

Given \ainfty categories $\CC_0,\CC_1,$ let $\Fun(\CC_0,\CC_1)$ denote
the \ainfty category of functors from $\CC_0$ to $\CC_1$ (see
Definition \ref{etc} for our conventions on \ainfty categories and
functors).  We construct for any pair of monotone symplectic manifolds
$M_0,M_1$ a Fukaya category of admissible correspondences
$\GFuk(M_0,M_1)$.  The objects of $\GFuk(M_0,M_1)$ are {\em sequences}
of compact Lagrangian correspondences with a brane structure, which we
call {\em generalized Lagrangian correspondences}.  The {\em brane
  structure} consists of an orientation, grading, and relative spin
structure.  The correspondences are also required to be {\em
  admissible} in the sense that the minimal Maslov numbers are at
least three, or vanishing disk invariant, and the fundamental groups
are torsion for any choice of base point.  \label{admhere} Denote by
$\GFuk(M) := \GFuk( \on{pt}, M)$ the natural enlargement of the Fukaya
category $\Fuk(M)$ whose objects are admissible generalized Lagrangian
correspondences with brane structures from points to a compact
monotone symplectic manifold $M$.  Our first main result is:

\begin{theorem}   \label{mainfunc}
{\rm (Functors for Lagrangian correspondences)} Suppose that $M_0,M_1$
are compact monotone symplectic manifolds with the same monotonicity
constant.  There exists an \ainfty functor
$$
\GFuk(M_0,M_1)  \to \Fun(\GFuk(M_0),\GFuk(M_1)) 
$$
inducing the functor of cohomology categories in \cite[Definition 5.1]{we:co}.
\end{theorem} 

In particular, for each admissible Lagrangian correspondence $L_{01}
\subset M_0^-\times M_1$ equipped with a brane structure we construct
an \ainfty functor
$$\Phi(L_{01}) : \GFuk(M_0) \to
\GFuk(M_1)$$ 
acting in the expected way on Floer cohomology: for Lagrangian branes
$L_0 \subset M_0, L_1 \subset M_1$ there is an isomorphism with
$\Z_2$-coefficients
$$ H \Hom( \Phi(L_{01})L_0, L_1) \cong HF(L_0^- \times L_1, L_{01}) $$
where the right-hand-side is the Floer cohomology of the pair $(L_0^-
\times L_1,L_{01})$.  For a pair of Lagrangian correspondences
$L_{01}, L_{01}' \in M_0^- \times M_1$ and a Floer cocycle $\alpha \in
CF(L_{01},L_{01}')$ we construct a natural transformation
$$\TT_\alpha: \Phi(L_{01}) \to \Phi(L_{01}')$$ 
of the corresponding \ainfty functors. 

The behavior of the \ainfty functors for Lagrangian correspondences
under embedded geometric composition as defined in \cite{ww:isom} is
our second main result.  To state it, we recall that the geometric
composition of Lagrangian correspondences 
$$L_{01} \subset M_0^- \times
M_1, \quad L_{12} \subset M_1^- \times M_2$$ 
is
\begin{equation} \label{geomcomp}
 L_{01} \circ L_{12} := \pi_{02} (L_{01} \times_{M_1} L_{12}) \end{equation} 
where $\pi_{02}: M_0 \times M_1^2 \times M_2 \to M_0 \times M_2$ is
the projection onto the product of the first and last factors.  If the fiber product
is transverse and embedded by $\pi_{02}$ then $L_{01} \circ L_{12}$ is
a smooth Lagrangian correspondence.

\begin{theorem} \label{maincompose} {\rm (Geometric composition
    theorem)} Suppose that $M_0,M_1,M_2$ are monotone symplectic
  manifolds with the same monotonicity constant.  Let
  $L_{01} \subset M_0^- \times M_1, L_{12} \subset M_1^- \times M_2$
  be admissible Lagrangian correspondences with spin structures and
  gradings such that $L_{01} \circ L_{12}$ is smooth, embedded by
  $\pi_{02}$ in $M_0^- \times M_2$, and admissible.  Then there exists
  a homotopy of \ainfty functors
$$ 
\Phi(L_{12}) \circ \Phi(L_{01}) \simeq \Phi(L_{01} \circ L_{12}) .
$$
\end{theorem}  

\noindent 
There is a slightly more complicated statement in the case that the
correspondences are only relatively spin, which involves a shift in
the background class.  In particular, the theorem implies that the
associated derived functors
$$D (\Phi(L_{12})) \circ D(\Phi(L_{01})) \cong D(\Phi(L_{01} \circ
L_{12})) : D\GFuk(M_0) \to D\GFuk(M_2) $$
are canonically isomorphic.  The result extends to generalized
Lagrangian correspondences, in particular the empty correspondence.
In the last case the result shows that the Fukaya categories
constructed using two different systems of perturbation data are
homotopy equivalent.  

A complete chain-level version of the earlier work is still missing.
Namely, one would like to construct a {\em Weinstein-Fukaya} \ainfty
$2$-category whose objects are symplectic manifolds and morphism
categories are the extended Fukaya categories of correspondences.
Furthermore one would like an \ainfty categorification functor given
by the extended Fukaya categories on objects and the functor of
Theorem~\ref{mainfunc} on morphisms. This theory would be the chain
level version of the Weinstein-Floer $2$-category and categorification
functor constructed in \cite{we:co}.  Some steps in this direction
have been taken by Bottman \cite{bottman}, \cite{bottman2}.  Batanin
has pointed out to us a possibly-relevant construction of homotopy
higher categories in
\label{batglob} \cite[Definition 8.7]{bat:glob}.

The \ainfty structures, functors, and natural transformations are
defined using a general theory of family quilt invariants that count
pseudoholomorphic quilts with varying domain.  This theory includes
families of quilts associated to the associahedron, multiplihedron,
and other polytopes underlying the various \ainfty structures.
Unfortunately these families of quilted surfaces come with the rather
inconvenient (for analysis) property that degeneration is not given by
``neck stretching'' but rather by ``nodal degeneration''.  Our first
step is to replace these families by ones that are more analytically
convenient, see Section \ref{family} for the precise definitions.  We
say that a stratified space is {\em labelled by quilt data} if for
each stratum there is given a combinatorial type of quilted surface,
and each pair of strata there is given a subset of gluing parameters
for the strip like ends as in Definition \ref{quiltdata} below.  For
technical reasons (contractibility of various choices) it is helpful
to restrict to the case that each patch of each quilt is homeomorphic
to a disk with at least one marking, and so has homotopically trivial
automorphism group.  Such quilt data are called {\em irrotatable}; the
general case could be handled with more complicated data associated to
the stratified space.

\begin{theorem}  
{\rm (Existence of families of quilts with strip-like ends)} \label{A}
Given a stratified space $\RR$ equipped with irrotatable quilt data,
then there exists a family of quilted surfaces $\SS = (\ul{S}_r)_{r
  \in \RR}$ with strip-like ends over $\RR$ with the given data in
which degeneration is given by neck-stretching.  \end{theorem}

The next step is to define pseudoholomorphic quilt invariants associated to
these families.  Let $\SS = (\ul{S}_r, r \in \RR)$ be a family of
quilted surfaces with strip-like ends over a stratified space $\RR$, $
\ul{M}$ be a collection of admissible monotone symplectic manifolds
associated to the patches, and $\ul{L}$ a collection of admissible
monotone Lagrangian correspondences associated to the seams and
boundary components.  Given a family $\ul{J}$ of compatible almost
complex structures on the collection $\ul{M}$ and a Hamiltonian
perturbation $\ul{K}$, a {\em holomorphic quilt} from a fiber of $\SS$
to $\ul{M}$ is pair 
$$ (r \in \RR, \ul{u}:  \ul{S}_r \to \ul{M} )$$
consisting of a point $r \in \RR$ together with a
$(\ul{J},\ul{K})$-holomorphic map $\ul{u}: \ul{S}_r \to \ul{M}$ taking
values in $\ul{L}$ on the seams and boundary, see Definition
\ref{holquilt} for the precise equation.  The necessary regularity
\label{provedin} statement is the following, proved in Theorem
\ref{familytransversality} in Section \ref{family}.

\begin{theorem}   {\rm (Transversality for families of holomorphic quilts)}
 \label{B} 
 Suppose that $\SS \to \RR$ is a family of quilted surfaces with
 strip-like ends equipped with compact monotone symplectic manifolds
 $\ul{M}$ for the patches and admissible Lagrangian correspondences
 $\ul{L}$ for the seams/boundaries.  Suppose over the boundary of
 $\RR$ a collection of perturbation data $(\ul{J},\ul{K})$ is given
 making all pseudoholomorphic quilts of formal dimension at most one
 regular.  Then for a generic extension of $(\ul{J},\ul{K})$ agreeing
 with the extensions given by gluing near the boundary $\SS
 |_{\partial \RR}$, every pseudoholomorphic quilt $u: \ul{S}_r \to \ul{M}$
 of formal dimension at most one with strip-like ends is parametrized
 regular.
\end{theorem} 

Using Theorem \ref{B} we construct moduli spaces of pseudoholomorphic quilts
and, using these, chain level {\em family quilt invariants} given as
counts of isolated elements in the moduli space.  As in the standard
topological field theory philosophy, these invariants map the tensor
product of cochain groups for the incoming ends $\E_-(\SS)$ to that
for the outgoing ends $\E_+(\SS)$:
$$ \Phi_{\SS} :\bigotimes_{\ul{e} \in \E_-(\SS)} CF(\ul{L}_{\ul{e}})
\to \bigotimes_{\ul{e} \in \E_+(\SS)} CF(\ul{L}_{\ul{e}}).
$$
These chain-level family invariants satisfy a {\em master equation}
arising from the study of one-dimension components of the moduli
spaces of pairs above:

\begin{theorem} 
{\rm (Master equation for family quilt invariants)} \label{C} Suppose
that, in the setting of Theorem \ref{B}, $\SS \to \RR$ is a family of
quilted surfaces with strip-like ends over an oriented stratified
space $\RR = \cup_{\Gamma} \RR_\Gamma$ (here the strata are indexed by
$\Gamma$) with boundary multiplicities $m_{\Gamma} \in \Z,
\codim(\RR_\Gamma) = 1$.  Then the chain level invariant $\Phi_{\SS}$
and the coboundary operators $\partial$ on the tensor products of
Floer cochain complexes satisfy the relation
$$ 
\partial \circ \Phi_{\SS} - \Phi_{\SS}
\circ \partial = \sum_{\Gamma, \codim(\RR_\Gamma) = 1} m_{\Gamma}
\Phi_{\SS_\Gamma} 
 .$$
\end{theorem} 

\noindent In other words, if $\partial \SS$ denotes the contribution
from boundary components of $\RR$ counted with multiplicity and
$\partial \Phi_{\SS} = [\partial, \Phi_{\SS}]$ denotes the boundary of
$\Phi_\SS$ considered as a morphism of chain complexes then
\begin{equation} \label{master}
 \partial \Phi_{\SS} = \Phi_{\partial \SS} .\end{equation} 
The master equation \eqref{master} specializes to the \ainfty
associativity, functor, natural transformation, and homotopy axioms
for the various families of quilts we consider.

The paper is divided into two parts.  The first part covers the
general theory of parametrized pseudoholomorphic quilts and the
construction of family quilt invariants.  The second part covers the
application of this general theory to specific families of quilts.
These applications include the construction of the generalized Fukaya
category, \ainfty functors between generalized Fukaya categories, as
well as natural transformations and homotopies of \ainfty functors.
The reader is encouraged to look at the constructions of Section
\ref{parttwo} while reading Sections \ref{partone1} and \ref{partone2}, in
order to have concrete examples of families of quilts in mind.

The present paper is an updated and more detailed version of a paper
the authors have circulated since 2007. The second and third authors
have unreconciled differences over the exposition in the paper, and
explain their points of view at
\href{https://math.berkeley.edu/~katrin/wwpapers/}{math.berkeley.edu/$\sim$katrin/wwpapers/}
resp.
\href{http://christwoodwardmath.blogspot.com/}{christwoodwardmath.blogspot.com/}. The
publication in the current form is the result of a mediation.

\section{Families of quilted surfaces with strip-like ends}

\label{partone1}
\label{family}

In this and the following section we construct invariants of families
of pseudoholomorphic quilts over stratified spaces, mapping tensor products
of the Floer cochain groups for the incoming ends to those for the
outgoing ends.  We also show that Theorems \ref{A}, \ref{B} and
\ref{C} from the introduction hold.  

First we define a surface with strip-like ends.  The definition below
is essentially the same as the definition given in Seidel's book
\cite{se:bo}, except that each strip-like end comes with an extra
parameter prescribing its width.

\begin{definition}  \label{surfstrip}
{\rm (Surfaces with strip-like ends)} A {\em surface with
  strip-like ends} consists of the following data:
\begin{enumerate}
\item 
A compact oriented surface $\ol{S}$ with boundary $\pd\ol{S}$ the
disjoint union of circles 
$\pd \ol{S} = C_1 \sqcup \ldots \sqcup C_m $
and $d_n \ge 0$ distinct points
$\ul{z}_n = (z_{n,1},\ldots,z_{n,d_n}) \subset C_n$
in cyclic order on each boundary circle $C_n\cong S^1$ for each
$ n = 1,\ldots, m$.  We use the indices on $C_n$ modulo $d_n$, and
index all marked points by
\begin{equation} \label{eseq}  \E=\E(S)=\bigl\{ e=(n,l) \,\big|\, n\in\{1,\ldots,m\},
l\in\{1,\ldots,d_n\} \bigr\} .
\end{equation} 
Here we use the notation $e\pm 1:=(n,l\pm 1)$ for the cyclically
adjacent \label{indiceshere} indices to $e=(n,l)$.  Denote by
$I_e :=I_{n,l}\subset C_n$ the component of $\partial S$ between
$z_e:=z_{n,l}$ and $z_{e+1}:=z_{n,l+1}$.  However, the boundary
$\partial S$ may also have compact components $I=C_n\cong S^1$;
\item 
A complex structure $j_S$ on $S:=\ol{S}\setminus\{z_e \,|\, e\in \E
\}$;
\item
A set of {\em strip-like ends} for $S$, that is a set of embeddings
with disjoint images
$$ \eps_e : \R^\pm \times [0,\delta_e] \to S $$
for all $e\in\cE$ such that the following hold:
$$ \begin{array}{l} \eps_{e}(\R^\pm\times\{0,\delta_e\})\subset\partial S \\ \lim_{s \to
  \pm \infty}(\eps_{e}(s,t)) = z_e, \ \ \ \forall t \in [0,\delta_e] \\ \eps_{e}^*j_S=j_0 \end{array} $$
where in the first item $\R^\pm = (0,\pm \infty)$ and in the third
item $j_0$ is the canonical complex structure on the half-strip
$\R^\pm \times[0,\delta_e]$ of width $\delta_e>0$.  Denote the set
of incoming ends $\eps_{e}: \R^- \times [0,\delta_e] \to S$ by
$\E_-=\E_-(S)$ and the set of outgoing ends $\eps_{e}: \R^+ \times
[0,\delta_e] \to S$ by $\E_+=\E_+(S)$;
\item \label{order} An ordering of the set of (compact) boundary components of
  $\ol{S}$ and orderings 
$$ \E_-=(e^-_1,\ldots, e^-_{N_-}), \quad \E_+=(e^+_1,\ldots,
  e^+_{N_+})$$ 
of the sets of incoming and outgoing ends; Here
$e^\pm_i=(n^\pm_i,l^\pm_i)$ denotes the incoming or outgoing end at
$z_{e^\pm_i}$.
\end{enumerate}
A {\em nodal surface with strip-like ends} consists of a surface with
strip-like ends $S$, together with a set of pairs of ends (the {\em
  nodes} of the nodal surface)
$$ \ul{w} = \{ \{ w_1^+, w_1^- \}, \ldots, \{ w_m^+,w_m^- \} \},
w^\pm_i \in \mE = \sqcup_{k \in K} \mE_k $$
such that for each $w_j^+,w_j^-$, the widths satisfy
$\delta_j^+ = \delta_j^-$ (the widths of the strips are the same).  A
nodal surface $(S,\ul{w})$ give rise to a topological space obtained
by from the union
$S \cup \{ w_j^\pm, j = 1,\ldots, m \} \subset \ol{S}$ by identifying
$w_j^+$ with $w_j^-$ for each $j = 1,\ldots, m$.  The resulting
surface is still denoted $S$.
\label{endshere}
\end{definition}

The structure maps of the Fukaya category, according to the definition
in Seidel \cite{se:bo}, are defined by counting points in a
parametrized moduli space in family of surfaces with strip-like ends.
These are defined as follows:

\begin{definition}
  \label{smoothfam} {\rm (Families of nodal surfaces with strip-like
    ends)} A {\em smooth family} of nodal surfaces with strip-like
  ends over a smooth base $\RR$ consists of
\begin{enumerate} 
\item a smooth manifold with boundary $\SS$, 
\item a fiber bundle $\pi: \SS \to \RR$ and 
\item a structure of a nodal surface with strip-like ends on each
  fiber $\SS_r := \pi^{-1}(r)$, whose diffeomorphism type is
  independent of $r$;
\end{enumerate}
such that $\SS_r$ varies smoothly with $r$ (that is, the complex
structures $j_{S_r}$ fit together to smooth maps $T^{\on{vert}} \SS
\to T^{\on{vert}} \SS$) and each $r \in \RR$ contains a neighborhood
$U$ in which the seam maps extend to smooth maps $\varphi_\sigma:
I_\sigma \times U \overset{\sim}{\to} I_{\sigma}' $.  
\end{definition} 

\begin{example} \label{gluingstrip} {\rm (Gluing strip-like ends)}  
A typical example of a family of surfaces with strip-like ends is
obtained by gluing strip-like ends by a neck of varying length.  Given
a nodal surface $S$ with strip-like ends and $m$ nodes and a pair of
ends with the same width $\delta_e$, define a family of surfaces with
strip-like ends over $\RR = \R_{\ge 0}^m$ by the following {\em gluing
  construction}: For any $\gamma = (\gamma_1,\ldots,\gamma_m) \in
\R_{\ge 0}^m$ define
\begin{equation} \label{glues}
 G_{\gamma}(S)= \left( S - \cup_{k=1}^m 
    \eps^{-1}_{k,\pm}( \pm (1/\gamma_k,\infty)\times [0,1])
  \right)/\sim \end{equation}
by identifying the ends $w_k^\pm, k =1,\ldots, m$ of $S$ by the gluing
in a neck of length $1/\gamma_k$, if $\gamma_k \neq 0$.  That is, if
both ends are outgoing then one removes the ends $w_k^\pm$ with
coordinate $s > 1/\gamma_k$ and identifies
$$ \eps_{w_k^+}(s,t)  \sim  \eps_{w_k^-}(s - 1/\gamma_k ,t) $$
for $s \in (0,1/\gamma_k)$ and $t \in [0,\delta_e]$.  If $\gamma_k = 0$ then the gluing
construction leaves the node in place.  This construction gives a
family of surfaces with the same number of strip-like ends and one
less node than $S$ over $\RR$ called the {\em glued surface}.  More
generally, given a family $\SS = (\SS_r, r \in \RR)$ of nodal surfaces
with strip-like ends with $m$ nodes over a base $\RR$, we obtain via
the gluing construction a family 
$$G(\SS) = \bigcup_{(r,\gamma)}  G_{\gamma}(S_r) $$
over the base $\RR \times (\R_{\ge 0})^m$ whose fiber at $(r,\gamma)$
is the glued surface $G_{\gamma}(S_r)$.
\end{example} 

In our earlier papers \cite{ww:quiltfloer}, \cite{ww:quilts} we
associated invariants to Lagrangian correspondences by counting maps
from {\em quilted surfaces}.  The notion of family and gluing
construction generalize naturally to the quilted setting.  Recall the
definition of quilted surface from \cite{ww:quilts}. 

\begin{definition}  \label{qsurf} {\rm (Quilted surfaces with strip-like ends)}  
A {\em quilted surface} $\ul{S}$ with strip-like ends consists of the
following data:
\ben 
\item {\rm (Patches)} A collection $\ul{S} = (S_k)_{k=1,\ldots,m}$ of
  {\em patches}, that is surfaces with strip-like ends as in
  Definition~\ref{surfstrip}~(a)-(c).  In particular, each $S_k$
  carries a complex structure $j_k$ and has strip-like ends
  $(\eps_{k,e})_{e\in \E(S_k)}$ of widths $\delta_{k,e}>0$ near marked
  points:
$$\lim_{s \to \pm \infty} \eps_{k,e}(s,t) = z_{k,e} \in
  \partial\overline{S}_k, \quad \forall t \in [0,\delta_e] .$$  
Denote by $I_{k,e}\subset\partial S_k$ the noncompact boundary
component between $z_{k,e-1}$ and $z_{k,e}$.
\item {\rm (Seams)} A collection of {\em seams}, pairwise-disjoint
  pairs
$$ \cS = \bigl( \{(k_\sigma,I_\sigma),(k_\sigma',I_\sigma')\}
  \bigr)_{\sigma\in\cS}, \quad \sigma \subset \bigcup_{k=1}^m \{ k \}
  \times \pi_0(\partial S_k) , $$
and for each $\sigma \in \cS$, a diffeomorphism of boundary components
$$\varphi_\sigma: \partial S_{k_\sigma} \supset I_\sigma \overset{\sim}{\to}
I_{\sigma}'\subset \partial S_{k_\sigma'}  $$
that satisfy the conditions:
\begin{enumerate}
\item {\rm (Real analytic)} Every $z\in I_\sigma$ has an open
  neighborhood $\U\subset S_{k_\sigma}$ such that
  $\varphi_\sigma|_{\U\cap I_\sigma}$ extends to an embedding
$$\psi_z:\U\to S_{k'_\sigma}, \quad \psi_z^* j_{k'_\sigma} = -
  j_{k_\sigma} .$$  
In particular, this forces $\varphi_\sigma$ to reverse the orientation
on the boundary components.  One might be able to drop the real
analytic condition, but we have not developed the necessary technical
results.
\item {\rm (Compatible with strip-like ends)} Suppose that
  $I_{\sigma}$ (and hence $I_{\sigma}'$) is noncompact, i.e.\ lie
  between marked points, $I_{\sigma}=I_{k_\sigma,e_\sigma}$ and
  $I_{\sigma}'=I_{k_\sigma',e_\sigma'}$.  In this case we require that
  $\varphi_\sigma$ matches up the end $e_\sigma$ with $e_\sigma'-1$
  and the end $e_\sigma-1$ with $e_\sigma'$.  That is
  $\eps_{k_\sigma',e_\sigma'}^{-1}\circ\varphi_\sigma\circ\eps_{k_\sigma,e_\sigma-1}$
  maps $(s,\delta_{k_\sigma,e_\sigma-1}) \mapsto (s,0)$ if both ends
  are incoming, or it maps $(s,0)\mapsto
  (s,\delta_{k_\sigma',e_\sigma'})$ if both ends are outgoing.  We
  disallow matching of an incoming with an outgoing end.  The
  condition on the other pair of ends is analogous.
\end{enumerate}
\item {\rm (Orderings of the ends)} There are orderings
$$\E_-(\ul{S})=(\ul{e}^-_1,\ldots,\ul{e}^-_{N_-(\ul{S})}), \quad 
\E_+(\ul{S})=(\ul{e}^+_1,\ldots,\ul{e}^+_{N_+(\ul{S})})$$ 
of the quilted ends.
\end{enumerate} 
As a consequence of (a) and (b) we obtain 
\begin{enumerate} 
\item {\rm (True boundary components)} a set of remaining boundary
  components $I_b\subset\partial S_{k_b}$ that are not identified with
  another boundary component of $\ul{S}$.  These {\em true boundary
    components} of $\ul{S}$ are indexed by
\begin{equation} \label{trueB} \cB  = \bigl(  (k_b,I_b) \bigr)_{b\in\cB}
 := \bigcup_{k=1}^m \{ k \} \times \pi_0(\partial S_k) \; \setminus \;
\bigcup_{\sigma\in\cS} \sigma .
\end{equation}
\item {\rm (Quilted Ends)} The {\em quilted ends}
  $\ul{e}\in \E(\ul{S})=\E_-(\ul{S})\sqcup \E_+(\ul{S})$ consist of a
  maximal sequence $\ul{e}=(k_i,e_i)_{i=1,\ldots,n_{\ul{e}}}$ of ends
  of patches with boundaries
  $\eps_{k_i,e_i}(\cdot,\delta_{k_i,e_i}) \cong
  \eps_{k_{i+1},e_{i+1}}(\cdot,0)$
  identified via some seam $\phi_{\sigma_i}$.  This end sequence could
  be cyclic, i.e.\ with an additional identification
  $\eps_{k_n,e_n}(\cdot,\delta_{k_n,e_n}) \cong
  \eps_{k_{1},e_{1}}(\cdot,0)$
  via some seam $\phi_{\sigma_n}$.  Otherwise the end sequence is
  noncyclic, i.e.\ $\eps_{k_1,e_1}(\cdot,0)\in I_{b_0}$ and
  $\eps_{k_n,e_n}(\cdot,\delta_{k_n,e_n})\in I_{b_n}$ take values in
  some true boundary components $I_{b_0}, I_{b_n}$.  In both cases,
  the ends $\eps_{k_i,e_i}$ of patches in one quilted end $\ul{e}$ are
  either all incoming, $e_i\in \E_-(S_{k_i})$, in which case we call
  the quilted end incoming, $\ul{e}\in\cE_-(\ul{S})$, or they are all
  outgoing, $e_i\in \E_+(S_{k_i})$, in which case we call the quilted
  end incoming, $\ul{e}\in\cE_+(\ul{S})$.
\end{enumerate} 
As part of the definition we fix an ordering
$\ul{e}=\bigl((k_1,e_1),\ldots (k_{n_{\ul{e}}},e_{n_{\ul{e}}})\bigr)$
of strip-like ends for each quilted end $\ul{e}$. For noncyclic ends,
this ordering is determined by the order of patches in \ref{qsurf} and
ends as in \ref{surfstrip} \eqref{order}. For cyclic ends, we choose a
first strip-like end $(k_1,e_1)$ to fix this ordering.
\end{definition} 

\begin{figure}[ht]
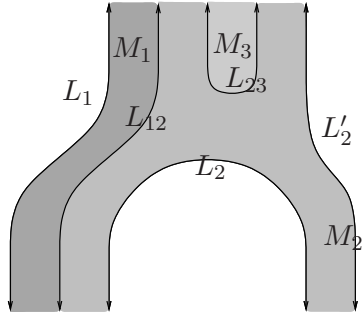

\begin{picture}(0,0)%
\includegraphics{quilted.pstex}%
\end{picture}%
\begin{picture}(0,0)%
\includegraphics{quilted.pstex}%
\end{picture}%
\setlength{\unitlength}{3947sp}%
\begingroup\makeatletter\ifx\SetFigFont\undefined%
\gdef\SetFigFont#1#2#3#4#5{%
  \reset@font\fontsize{#1}{#2pt}%
  \fontfamily{#3}\fontseries{#4}\fontshape{#5}%
  \selectfont}%
\fi\endgroup%
\begin{picture}(2215,1973)(2961,-1623)
\put(3311,-309){\makebox(0,0)[lb]{{{{$L_1$}%
}}}}
\put(3703,-469){\makebox(0,0)[lb]{{{{$L_{12}$}%
}}}}
\put(4141,-778){\makebox(0,0)[lb]{{{{$L_2$}%
}}}}
\put(4951,-1222){\makebox(0,0)[lb]{{{{$M_2$}%
}}}}
\put(3625,-40){\makebox(0,0)[lb]{{{{$M_1$}%
}}}}
\put(4254,-14){\makebox(0,0)[lb]{{{{$M_3$}%
}}}}
\put(4332,-228){\makebox(0,0)[lb]{{{{$L_{23}$}%
}}}}
\put(4935,-563){\makebox(0,0)[lb]{{{{$L_2'$}%
}}}}
\end{picture}%

\caption{Lagrangian boundary conditions for a quilt}
\label{bc quilt}
\end{figure}

Later in the construction of invariants arising from families of
quilted surfaces, we will need the following auxiliary results
concerning convexity of seams and tubular neighborhoods of them.
Let $\ul{S}$ be a quilted surface with strip-like ends and $S$ the
unquilted surface obtained by gluing together the seams.

\begin{definition} \label{tubnbds}  {\rm (Tubular neighborhoods of seams)} A {\em
    tubular neighborhood} of a seam $I$ is an embedding
  $I \times (-\eps,\eps) \to S$ such that $I \times \{ 0 \}$ is an
  orientation-preserving diffeomorphism onto the image of $I$ in $S$.
  Two tubular neighborhoods $I \times (-\eps_j,\eps_j) \to S, j = 1,2$
  are {\em equivalent} if they agree on $I \times (-\eps,\eps)$ for
  some $\eps \in (0, \min(\eps_1,\eps_2))$.  A {\em germ} of a tubular
  neighborhood is its equivalence class.  
\end{definition}  

\begin{lemma} \label{product} \label{contractible} {\em (Contractibility of tubular
    neighborhoods of seams)} Let $I$ be a seam in a quilted surface
  $\ul{S}$.  The set of germs of tubular neighborhoods of $I$ is in
  bijection with the set of germs of vector fields on $S$ normal to
  $I$, which is contractible in the $C^k$ topology for any $k \ge 0$.
\end{lemma}

\begin{proof} 
  Tubular neighborhoods can be constructed using normal flows as
  follows.  Let $S$ be the surface obtained from $\ul{S}$ by gluing
  together the seams.  Let $I \subset S$ be a seam equipped with a
  metric $g_I\in \on{Sym}^{\otimes 2}(T^\dual I)$, base point
  $z \in S$ and a vector field $v \in \Vect(S)$ transverse to the
  seam.  The flow of the vector field $v$ gives a map
$$\varphi:I \times (-\eps,\eps) \to S, \quad \ddt \varphi(s,t) =
v(\varphi(s,t)), \quad \varphi(I \times \{ 0 \}) = I .$$
Since $v$ is transverse to the seam, the flow $\varphi$ is a
diffeomorphism for $\eps$ sufficiently small by the inverse function
theorem.  This construction produces a bijection between germs of
tubular neighborhoods $\varphi$ and germs of vector fields
$v \in \Vect(S)$ transverse to the seam and agreeing with the given
orientation.  Let $\Lambda^{\top}_+(TS) \cong S \times \R_{> 0}$ be
the space of orientation forms on $S$.  The space of such vector
fields is
\begin{equation}\label{VI}
\{ v \in \Vect(S) \ | \ v |_I \wedge w \subset \Lambda^{\top}_+(TS)
|_I \}\end{equation}
where $w \in \Vect(I)$ is the positive unit vector field on the seam,
and as such is convex.  It follows that the space of germs of such
vector fields, hence also the space of germs of tubular neighborhoods,
is contractible.
\end{proof}

\begin{lemma} \label{homtriv}
{\rm (Contractibility of the space of metrics of product form near a
  seam)} Let $\ul{S}$ be a quilted surface with strip-like ends.  The
space of metrics on $\ul{S}$ that are locally of product form near the
seams is non-empty and homotopically trivial.
\end{lemma} 

\begin{proof}  
  Let $S$ denote the unquilted surface obtained by gluing along the
  seams of $\ul{S}$.  For each seam $I$ and a tubular neighborhood
  $I \times (-\eps,\eps)\to S $ one obtains from the standard metric
  on the domain a metric $g_I: T^{\otimes 2}U_I \to \R$ on a
  neighborhood $U_I$ of $I$ in $S$.  After shrinking the tubular
  neighborhood, there exists an extension
  $ g : T^{\otimes 2}S \to \R $ such that $g | U_I = g_I$ for all
  seams $I$.  Indeed, the fact that the space of metrics compatible
  with the given complex structure is contractible.  Contractibility
  follows from contractibility of the space of metrics and of germs of
  tubular neighborhoods, as in Lemma \ref{product}.
\end{proof}

Next we introduce a definition of nodal quilted surfaces suitable for
the purpose of defining family quilt invariants.

\begin{definition} \label{nqs} {\rm (Nodal quilted surfaces)}  
  A {\em nodal quilted surface} consists of a quilted surface $\ul{S}$
  with a set of pairs of ends (the {\em nodes} of the quilted surface)
$$ \{ \{ \ul{w}_1^+, \ul{w}_1^- \}, \ldots, \{ \ul{w}_m^+,\ul{w}_m^-
\} \}, \quad \ul{w}^\pm_i \in \ul{\mE}, i = 1,\ldots, m $$
such that each is distinct and for each pair $\ul{w}_j^+,\ul{w}_j^-$,
the data of the ends (number of seams and widths of strips) is the
same. 
\end{definition} 

\begin{definition} \label{qglue}
\begin{enumerate} 
\item {\rm (Gluing quilted surfaces with strip-like ends)} Given a
  nodal quilted surface $\ul{S}$ with $m$ nodes, a non-zero gluing
  parameter $\gamma$, and a node represented by ends
  $\ul{e}_\pm = (k_{i,\pm}, e_{i,\pm})_{i=1}^{m_\pm}$ with the same
  widths, we obtain a {\em glued quilted surface}
  \begin{equation} \label{qglues} G_\gamma(\ul{S}) = \ul{S} - \cup_{i
      = 1,\ldots m_\pm} \eps^{-1}_{k_i,e_i,\pm}( \pm
    (1/\gamma,\infty)\times [0,1]) / \sim \end{equation}
  by identifying the ends of $\ul{S}$ by gluing in a neck of length
  $1/\gamma$, that is,
  \begin{equation} \label{identify} \eps_{k_{-,i},e_{-,i}}(s,t) \sim
    \eps_{k_{+,i},e_{+,i}}(s - 1/\gamma,t) \end{equation}
  for $s \in (0,1/\gamma)$.  More generally, a similar definition
  constructs a glued surface given a {\em collection} of gluing
  parameters $\ul{\gamma} = (\gamma_1,\ldots, \gamma_m)$ associated to
  the nodes.  We extend the definition to allow gluing parameters in
  $[0,\infty)^m$ with the convention that if a gluing parameter is
  zero, then we leave the node as is (that is, do not perform the
  gluing).
\item {\rm (Isomorphisms of nodal quilted surfaces)} An {\em
  isomorphism} between nodal quilted surfaces is a diffeomorphism
  between the disjoint union of the components, that preserves the
  matching of the ends, the ordering of the seams and boundary
  components.  
\item {\rm (Smooth families of nodal quilted surfaces)} A {\em smooth
  family} $\SS = (\ul{S}_r)_{r \in \RR}$ of quilted surfaces over a
  manifold $\RR$ {\em of fixed type} is a collection of families of
  surfaces with strip-like ends $(S_j \to \RR)_{j=1}^m$ each of fixed
  type together with seam identifications that vary smoothly in $r \in
  \RR$ in the local trivializations.  Each fiber $\ul{S}_r$ is a
  quilted surface with strip-like ends.
\end{enumerate}  
\end{definition}  

\begin{remark} \label{insertrem} {\rm (Inserting strips construction)}  
Another way of producing families of quilted surfaces is by the
following {\em inserting strips} construction produces from a family
of surfaces with strip-like ends and no compact boundary components a
family of quilted surfaces with strip-like ends.  Given a collection
$(n_1,\ldots, n_d)$ of positive integers and for each $j = 1,\ldots,
d$ a sequence $\delta_j = (\delta_j^1,\ldots, \delta_j^{n_j})$ of
positive real numbers let
$$\ul{S}(\ul{\delta}) = \ul{S} \sqcup \coprod_{i,j} ([0,\delta_j^i]
\times \R) $$
denote the quilted surface with strip-like ends obtained by gluing on
strips of width $\delta_j^i$ to the boundary, using the given local
coordinates near the seams.  If $\SS \to \RR$ is a family of quilted
surfaces with strip-like ends, this construction gives a bundle
$\SS(\ul{\delta})$ over $\RR$ whose fiber is a quilted surface, with
$n_j$ strips corresponding to the $j$-th component of the boundary of
the underlying quilted disk.
\end{remark} 

Later we will need that certain families of quilted surfaces are
automatically trivializable, after forgetting the complex structures:

\begin{lemma} {\rm (Trivializability of families of quilted surfaces)}  
\label{trivfam}   Suppose that $\SS = (\ul{S}_r)_{r \in \RR}$ is a smooth family of
  quilted surfaces over a base $\RR$ such that each patch $S_j$ is
  homeomorphic to the disk and has at least one marking.  Then $\SS$
  is smoothly globally trivializable in the sense that there exists a
  diffeomorphism $\SS \to \RR \times \ul{S}_0$ mapping $\SS_r$ to
  $\{ r \} \times \ul{S}_0$ for any $r \in \RR$ for a fixed quilted
  surface with strip-like ends $\ul{S}_0$, not necessarily preserving
  the complex structures or strip-like ends.
\end{lemma}  

\begin{proof}   Suppose that $S_{k,r}$ are holomorphic disks with $n_k$ markings
$z_1,\ldots, z_{n_k} \in \partial S_{k,r}$ on the boundary.  If $n_k
  \ge 3$ then $S_{k,r}$ admits a canonical isomorphism to the unit
  disk
$$ S_{k,r} \to D := \{z \in \C \ | \ | z | \leq 1 \}$$ 
in the complex plane with first three markings $z_1,z_2,z_3$ mapping
to $1,i,-1$, and the remainder to the lower half of the unit circle
$\partial D \cap \{ \on{Im}(z) < 0 \}$.  Consider over the set of
equivalence classes of such tuples the universal marked disk bundle
$$ \UU^{n_k} = \{ (D, z_1,\ldots, z_{n_k} \in \partial D) \} /\Aut(D)
.$$
Choose a connection on this bundle preserving the markings.
Identifying $\RR^{n_k} \cong \R^{n_k - 3}$, such a connection is given
by lifts
$$ \ti{\partial_i} \in \Vect(\UU^{n_k}), \quad \ti{\partial_i}(
z_i(r)) \in \on{Im}( \D z_i ( T_r \RR^{n_k})) $$
of the coordinate vector fields $\partial_i \in \Vect(\RR^{n_k - 3})$
tangent to the images of the sections
$$z_i: \RR^{n_k} \to \UU^{n_k}, \quad i =1,\ldots, n_k $$
given by the markings.  (The lifts may be defined first locally, using
the fact that the sections are disjoint, and then patched together
using contractibility of the space of lifts.)  Using the connection,
any homotopy of the identity
$\varphi_t: \RR^{n_k} \to \RR^{n_k}, \varphi_0 = \on{Id}$ lifts to a
homotopy $\ti{\varphi}_t: \UU^{n_k} \to \UU^{n_k}$ by requiring
$\ddt \ti{\varphi}_t (u)$ to be horizontal and project to
$\ddt \varphi_t (r)$ for $u \in \UU^{n_k}_r$.  In particular, a
contraction of $\RR^{n_k}$ to a point $r_0$ lifts to a trivialization
$\UU^{n_k} \to \UU^{n_k}_{r_0} \times \RR^{n_k}$.  Similarly if
$n_k = 2$ or $1$ then $S_{k,r}$ admits such an isomorphism canonical
up to translation (resp. translation and dilation).  Since the groups
of such are contractible, the bundle $S_k \to \RR$ is again trivial.
The nodal sections $w^\pm_k: \RR \to \SS$ are trivial with respect to
these trivialization of the disk bundles, by construction.  Finally
the space of seam identifications is convex, hence in any family
contractible to a fixed choice.
\end{proof} 

\label{varying}

Next we discuss families of varying combinatorial type.  The natural
category of base spaces for these are {\em stratified spaces} in the
sense of Mather, whose definition we now review, c.f. \cite{go:st}. We
begin with the definition of {\em decomposed spaces}.

\begin{definition} \label{decomposed} {\rm (Decomposed spaces)}  
 Let $\GG$ be a partially ordered set with partial order $\leq$.  Let
 $\RR$ be a Hausdorff paracompact space.  A {\em $\GG$-decomposition}
 of $\RR$ is a locally finite collection of disjoint locally closed
 subspaces $\RR_\Gamma, \Gamma \in \GG$ each equipped with a smooth
 manifold structure of constant dimension $\dim(\RR_\Gamma)$, such
 that
$$ \RR = \bigcup_{\Gamma \in \GG} \RR_\Gamma $$
and 
$$ (\RR_\Gamma \cap \ol{\RR_{\Gamma'}} \neq \emptyset) \iff ( \RR_{\Gamma}
\subset \ol{\RR_{\Gamma'}}) \iff (\Gamma \leq \Gamma') .$$
The {\em dimension} of a $\GG$-decomposed space $\RR$ is
$$ \dim \RR = \sup_{\Gamma \in \GG} \dim( \RR_\Gamma) .$$
The {\em stratified boundary} $\partial_s \RR$ resp. {\em stratified
  interior} $\on{int}_s \RR$ of a $\GG$-decomposed space $\RR$ is the
union of pieces $\RR_\Gamma$ with $\dim(\RR_\Gamma) < \dim(\RR)$,
resp. $\dim(\RR_\Gamma) = \dim(\RR)$.  An {\em isomorphism} of
$\GG$-decomposed spaces $\RR_0,\RR_1$ is a homeomorphism $\RR_0 \to
\RR_1$ that restricts to a diffeomorphism on each piece.
\end{definition}  

\begin{example} \label{cone}
\begin{enumerate} 
\item {\rm (Cone construction)} Let $\RR$ be a $\GG$-decomposed space,
  and let 
  \begin{equation} \label{CGeq} C\GG := \left( \GG \times \{ \{ 0 \},
      (0,\infty) \} \right) \sqcup \{ \infty \} \end{equation}
  with the partial order determined by
$$\{ \infty \} \preceq (\Gamma, (0,\infty)), \quad (\Gamma,\{ 0 \} ) \preceq (\Gamma,
(0,\infty)), \quad \forall \Gamma \in \GG .$$
  The {\em cone} on $\RR$
$$ \Cone(\RR) := \left( \RR \times [0,\infty] \right) / \left( (r,\infty) \sim (r',\infty), r,r' \in \RR 
\right) $$
has a natural $C\GG$-decomposition with
$$ ( \Cone(\RR))_{ (\Gamma, (0,\infty))}
\cong \RR_\Gamma \times (0,\infty), \quad \dim(\Cone(\RR)) = \dim(\RR)
+ 1 .$$
  More generally, if $\RR$ is a $\GG$-decomposed space equipped with a
  locally trivial map $\pi$ to a manifold $B$, the {\em cone bundle}
  on $\RR$ is the union of cones on the fibers, that is,
$$ \Cone_B (\RR) := \left( \RR \times [0,\infty] \right) / \left( (r,\infty) \sim (r',\infty), \pi(r) = \pi(r')
  \in \RR \right),
 $$
 is again a $C\GG$-decomposed space with dimension $\dim(\RR) + 1$.
\item {\rm (Convex polyhedra)} Let $V$ be a vector space over $\R$
  with dual $V^\dual$.  A {\em half-space} in $V$ is a subset of the
  form $\lambda^{-1}[c,\infty)$ for some $\lambda \in V^\dual$ and
  $c \in \R$.  A {\em convex polyhedron} in $V$ is the the
  intersection of finitely many half-spaces in $V$.  A half-space
  $H \subset V$ is {\em supporting} for a polytope $P$ if and only if
  $P \subseteq H$.  A {\em closed face} of $P$ is the intersection
  $F = P \cap \partial H$ of $P$ with the boundary of a supporting
  half-space $H$.  Let $\cF(P)$ denote the set of closed faces.  The
  {\em open face} $F^\circ$ corresponding to a face $F$ is the closed
  face minus the union of proper subfaces
  $F^\circ := F - \cup_{F' \subsetneq F} F' .$ The decomposition of
  $P$ into open faces
$$P = \bigcup_{F \subset \cF(P)} F^\circ$$ 
gives $P$ the structure of a decomposed space.  
\item {\rm (Locally polyhedral spaces)}  
 \label{locpoly}  A decomposed space is {\em locally polyhedral} if it is locally
  isomorphic to a polyhedral space, that is, any point has an open
  neighborhood that, as a decomposed space, is isomorphic to a
  neighborhood of a point in a convex polyhedron.
\end{enumerate}
\end{example} 

\begin{definition} {\rm (Stratified spaces)}  
A decomposition $\RR = \cup_{\Gamma \in \GG} \RR_\Gamma$ of a space
$\RR$ is a {\em stratification} if the pieces $\RR_\Gamma$ fit
together in a nice way: Given a point $r$ in a piece $\RR_\Gamma$
there exists an open neighborhood $U$ of $r$ in $\RR$, an open ball
$B$ around $r$ in $\RR_\Gamma$, a stratified space $L$ (the {\em link}
of the stratum) and an isomorphism of decomposed spaces $B \times CL
\to U $ that preserves the decompositions in the sense that it
restricts to a diffeomorphism from each piece of $B \times CL$ to a
piece $U \cap \RR$.  A {\em stratified space} is a space equipped with
a stratification.
\end{definition} 

\begin{remark}  {\rm (Recursion on depth versus recursion on dimension)}  
The definition of stratification is recursive in the sense that it
requires that stratified spaces of lower dimension have already been
defined; in general one can allow strata with varying dimension and
the recursion is on the {\em depth} of the piece, see
e.g. \cite{sj:st}.
\end{remark} 

The master equation for our family quilt invariants involves the following 
notion of boundary of a stratified space. 

\begin{definition} {\rm (Boundary with multiplicity)}  
\begin{enumerate} 
\item An {\em orientation} on a stratified space $\RR = \cup_{\Gamma
  \in \GG} \RR_\Gamma$ is an orientation on the top-dimensional
  pieces.  If $\RR_{\Gamma_1} \subset \ol{\RR_{\Gamma_2}}$ is the
  inclusion of a codimension one piece in a codimension zero piece,
  then the finite fibers of the link bundle $L_{\Gamma_1}$ inherit an
  orientation from the top-dimensional pieces and the positive
  orientation of $\RR$.
\item Summing the signs over the points in the fibers of the link
  bundle defines a locally constant {\em multiplicity function}
$$m_{\Gamma}: \RR_{\Gamma} \to \Z$$ 
on the codimension one pieces $\RR_\Gamma$.  
\item The {\em boundary with multiplicity} $\partial_m \RR$ of $\RR$
  is the union of codimension one pieces $\RR_{\Gamma_1}$ equipped
  with the given multiplicity function $m_{\Gamma_1}$.
\item 
Let $\RR = \cup_{\Gamma \in \GG} \RR_\Gamma $ be a stratified space.
A {\em family of quilted surfaces with strip like ends} over $\RR$ is
a stratified space 
$$\SS = \cup_{\Gamma \in \GG} \SS_\Gamma $$ 
equipped with a stratification-preserving map to $\RR$ such that each
$\SS_\Gamma \to \RR_\Gamma$ is a smooth family of quilted surfaces
with fixed type.  Furthermore local neighborhoods of $\SS_\Gamma$ in
$\SS$ are given by the gluing construction: there exists a
neighborhood $U_\Gamma$ of $\SS_\Gamma$, a projection $\pi_\Gamma:
U_\Gamma \to \RR_\Gamma$, and a map $\gamma_\Gamma: U_\Gamma \to
(\R_{\ge 0})^m$ such that if $r \in \RR_\Gamma$ then
$$\ul{S}_r = G_{\gamma_\Gamma(r) } \ul{S}_{\pi_\Gamma(r)} .$$
\end{enumerate} 
\end{definition} 

In other words, for a family of quilted surfaces with strip-like ends,
degeneration as one moves to a boundary stratum is given by
neck-stretching.  Often we will be given a family of quilted surfaces
without strip-like ends in which degeneration is difficult to deal
with analytically, and we wish to produce a family with strip-like
ends where degeneration is given by neck-stretching.  The following
theorem allows us to replace our original family with a nicer one.

\begin{definition} \label{quiltdata} {\rm (Quilt data for a stratified space)}  
A stratified space $\RR$ is {\em equipped with quilt data} if the
index set $\GG$ is a subset of the set of combinatorial types of
quilts and for each piece $\RR_\Gamma$ such that $\Gamma$ has $m$
nodes there exists a stratified subspace 
$Z_\Gamma \subset \R_{\ge   0}^{m}$
\footnote{That is, the stratification of $Z_\Gamma$ is induced from
  the stratification of $\R_{\ge 0}^{m}$ as a manifold with corners
  indexed by subsets of $\{ 1,\ldots, m \}$, defining the strata to be
  submanifolds where those coordinates are zero.}  and collar
neighborhoods\footnote{That is, open embeddings mapping $\RR_\Gamma
  \times \{ 0 \}$ diffeomorphically onto $\RR_\Gamma$ .}
$$ \varphi_\Gamma : \RR_\Gamma \times Z_\Gamma 
\to \RR $$
such that the following compatibility condition holds: for any two
strata $\Gamma_1, \Gamma_2$ such that $\Gamma_1 \leq \Gamma_2$ the
diagram
$$ 
\begin{diagram} 
\node{\RR_{\Gamma_1} \times Z_{\Gamma_1}} 
\arrow{s} \arrow{e} \node{\RR}
\node{\RR_{\Gamma_2} \times Z_{\Gamma_2}} 
\arrow{w}  \arrow{s} \\ 
\node{\R^{m_1}} \node{}
\node{\R^{m_2}} \arrow[2]{w} 
\end{diagram} 
$$
commutes where defined, that is, on the overlap of the images of the
open embeddings in $\RR$.
\end{definition} 

The following result builds up families of quilted surfaces by
induction on the dimension of the stratum in the base of the family.

\begin{theorem}  {\rm (Extension of quilt data over the interior)}  
\label{extension} 
Let $\RR$ be a stratified space labelled by quilt data.  Given a
family $\SS$ of quilted surfaces with strip-like ends on the boundary
$\partial \RR$ such that the family of quilts in the neighborhood of
the boundary obtained by gluing is smoothly trivializable, there
exists an extension of $\SS$ to a family of quilted surfaces with
strip-like ends over the interior of $\RR$.
\end{theorem}  

\begin{proof}  
  The existence of an extension is a combination of the gluing
  construction and the contractibility of the space of metrics and
  seam structures.  Namely, via the gluing construction \ref{qglue}
  one obtains in an open neighborhood $U$ of $\partial \RR$ a family
  of quilted surfaces $S_r, r \in U$ with strip-like ends
  $\eps_{r,i}$, compatible metrics $g_r$ and seam maps
  $\varphi_{\sigma,r}$ so that the metrics are of product form near
  the seams.  By assumption, this family is smoothly trivial and so by
  Lemma \ref{homtriv} the family $g_r, \eps_{r,i}, \varphi_{\sigma,r}$
  extends over the interior, possibly after shrinking the neighborhood
  of the boundary.  Indeed, since the spaces of metrics $g_r$ and seam
  maps $\varphi_{\sigma,r}$ are contractible, the metrics and seam map
  extend over the interior using cutoff functions and patching;
  similarly the space of strip-like ends is convex, as it is
  isomorphic to the space of local coordinates at a point on the
  boundary of complex half-space.  Finally, choose collar
  neighborhoods of the seams $\varphi_{\sigma,r}(I_r)$.  The
  corresponding complex structures $j_r : TS_r \to TS_r$ have the
  property that the seams are automatically real analytic.
\end{proof} 

Theorem \ref{A} of the Introduction now follows by recursively
applying Theorem \ref{extension} to the strata.

\section{Moduli spaces of pseudoholomorphic quilts in families} 
\label{partone2}

In this section we construct the moduli spaces of pseudoholomorphic quilts
for families of quilts.  Let $ \SS = (\ul{S}_r, r \in \RR)$ be a
family of quilted surfaces with strip-like ends over a stratified
space $\RR$.  Let $\SS_\Gamma \to \RR_\Gamma$ denote the pieces of
$\SS$.

\begin{definition} \label{datum}
\begin{enumerate}
\item {\rm (Symplectic datum for a family of quilted surfaces)} $\SS =
  (\ul{S}_r)_{r \in \RR}$ is {\em labelled by symplectic data}
  $(\ul{M},\ul{L})$ if each patch $S_k$ is labelled by a component
  $M_k$ of $\ul{M}$ (we assume the same indexing for simplicity) that
  is a symplectic background, each seam $I_{\sigma}$ is labelled by a
  Lagrangian correspondence $L_\sigma \subset M_p^- \times M_{p'}$ for
  the product of symplectic manifolds for the adjacent patches
  $S_p,S_{p'}$, with admissible brane structure.
\item {\rm (Almost complex structures and Hamiltonian perturbations
  for the ends)} For each end $e \in \mE(S)$ with widths $\delta_j$
  and symplectic labels $M_j$ for $j=0,\ldots,r$ we assume that we
  have chosen almost complex structures
$$\ul{J}_e = (J_j) \in \oplus_{j=0}^r C^\infty([0,\delta_j]
  ,\J(M_j,\omega_j)) $$
and Hamiltonian perturbations
$$\ul{H}_e = (H_j) \in \oplus_{j=0}^r C^\infty([0,\delta_j] \times
M_j) ,$$
with Hamiltonian vector fields $Y_j,j = 0,\ldots,r$ as in
\cite[Theorem 5.2.1]{ww:quiltfloer} so that the set of perturbed
intersection points
$$ \cI(\ul{L}_e) := \left\{ \ul{x}=\bigl(x_j: [0,\delta_j] \to
M_j\bigr)_{j=0,\ldots,r} \, \left|
\begin{aligned}
\dot x_j(t) = Y_j(t,x_j(t)), \\ (x_{j}(\delta_j),x_{j+1}(0)) \in
L_{j(j+1)}
\end{aligned} \quad \forall j \right.\right\} .
$$
is cut out transversally.  Denote by 
$$ \ul{K}_e = ( K_j \in \Omega^1([0,\delta_j], C^\infty( M_j)), K_j :=
H_j \d t )_{j = 0,\ldots, r} $$
the corresponding family of function-valued one-forms.
\item {\rm (Perturbation datum for a family of quilted surfaces with
  symplectic data)} Let $\ul{\J}_{\ul{\omega}}$ denote the space of
  almost complex structures on the symplectic manifolds $\ol{M}$
  compatible with (or, it would suffice, tamed by) the symplectic
  forms $\ol{\omega}$.  An {\em almost complex structure} for a family
  $\SS \to \RR$ of quilted surfaces with strip-like ends equipped with
  a symplectic labelling is a collection of maps
$$\ul{J}_\Gamma \in C^\infty( \SS_\Gamma, \ul{\J}_{\ul{\omega}}) $$
agreeing with the given almost complex structures on the strip-like
ends $\ul{J}_e$ and agreeing on the ends corresponding to any node,
with the additional property that if $\Gamma_1 < \Gamma_2$ then
$\ul{J}_{\Gamma_2}$ is obtained from the gluing construction
\ref{qglue} from $\ul{J}_{\Gamma_1}$ under the identifications
\eqref{identify}.  A {\em Hamiltonian perturbation} for $\SS$ is a
family
$$\ul{K}_\Gamma \in \Omega^1_{\SS_\Gamma/\RR_\Gamma}(\cH(\ul{M})) $$
agreeing with the given Hamiltonian perturbations $\ul{K}_e$ on the
ends with the additional property that if $\Gamma_1 < \Gamma_2$ then
$\ul{K}_{\Gamma_2}$ is given on a neighborhood of $\ul{K}_{\Gamma_1}$
by the gluing construction \eqref{qglue}.
\end{enumerate} 
To clarify the notation $C^\infty( \SS_\Gamma, \ul{\J}_{\ul{\omega}})
$ each quilted surface $ \ul{S}_r, r \in \RR$ splits into a union of
patches, $\cup_i (S_i)_r$, with all patches $S_i$ labeled by a target
symplectic manifold $ (M_i, \omega_i)$.  For each $(r, z), z\in
(S_i)_r$, $J_i(r,z)$ is an $\omega_i$-compatible almost complex
structure on $M_i$.  The notation $
\Omega^1_{\SS_\Gamma/\RR_\Gamma}(\cH(\ul{M}))$ represents the space of
1-forms on each quilted surface $\SS_r$ that on each patch $S_i$ of
the quilt take values in the space of Hamiltonians on $M_i$.
\end{definition} 


The domains of the pseudoholomorphic quilts associated to a family of
quilted surface are pseudoholomorphic maps from {\em destabilizations} of
elements of the family in the following sense:

\begin{definition}  \label{holquilt} 
\begin{enumerate} 
\item {\rm (Destabilizations)} Let $\ul{S}$ be a quilted surface with
  strip-like ends.  A {\em destabilization} of $\ul{S}$ is a quilted
  surface with strip-like ends $\ul{S}^{\on{ds}}$ obtained from $\ul{S}$ by
  inserting a finite collection of quilted strips (twice marked disks)
  at the nodes and ends.  See Figure \ref{destab}.
\begin{figure}
\includegraphics[height=3in]{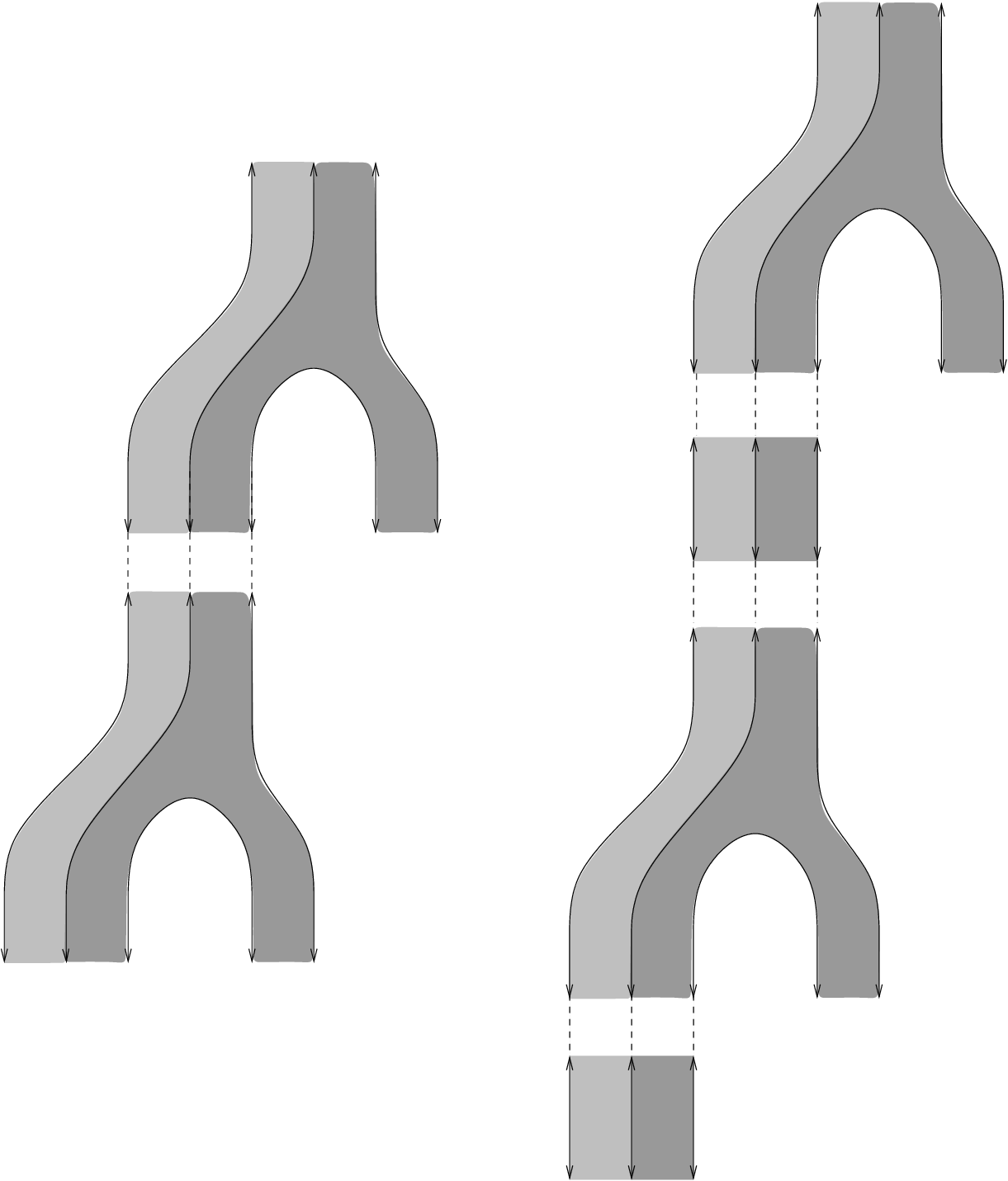}
\caption{A quilted surface with strip-like ends and a destabilization
  of it}
\label{destab}
\end{figure}
\item {\rm (Pseudoholomorphic quilts with varying domain)} Let $\SS \to \RR$
  be a smooth family of quilted surfaces with strip-like ends.  A {\em
    pseudoholomorphic quilt} for $\SS$ is a datum $(r,\ul{S}_r^{\on{ds}},\ul{u})$
  where $r \in \RR$, $\ul{S}_r^{\on{ds}}$ is a destabilization of $\ul{S}_r$,
  $\ul{u} \in C^\infty(\ul{S}_r^{\on{ds}}, \ul{M})$ and $\ul{u}$ satisfies the
  inhomogeneous pseudoholomorphic map equation on each patch
\begin{multline}
d u_p(z) + J_p(r, z,u_p(z)) \circ \d u_p(z) \circ j_p(r, z) =
Y_p(r,z,u_p(z)) + \\ J_p(r,z,u_p(z)) \circ Y_p(r,z,u_p(z)) \circ
j_p(r,z), \quad \forall z \in S_p, \forall p = 1,\ldots, k
\end{multline}  
where $j_p(r,z)$ is the complex structure on the quilt $S_{r,p}$ at $z
\in S_{r,p}$, and $Y_p(r,z) \in \Vect(M_p)$ is the Hamiltonian vector
field associated to the Hamiltonian perturbation $K_p(r,z)$.
\item {\rm (Isomorphism)} Two pseudoholomorphic quilts
  $(r,\ul{S}_r^{\on{ds}},\ul{u}), (r',\ul{S}_r^{\on{ds},'},\ul{u}')$ are
  {\em isomorphic} if $r = r'$ and there exists an isomorphism of
  destabilizations $\phi: \ul{S}_r^{\on{ds}} \to \ul{S}_r^{\on{ds},'}$
  inducing the identity on $\ul{S}_r$ such that $u' \circ \phi = u$.
\item {\rm (Regular pseudoholomorphic quilts)} Associated to any
  pseudoholomorphic quilt $\ul{u}: \SS^{\on{ds}} \to \ul{M}$ with
  destabilization $\SS$ is a Fredholm {\em linearized operator} for
  any integer $p > 2$
\begin{multline} \label{linop}
 D_{\SS, r,\ul{u}} : T_r \RR \times 
\Omega^0( \ul{S}^{\on{ds}}_r, \ul{u}^* T
 \ul{M}, \partial \ol{u}^* T \ul{L})_{1,p} 
\to \Omega^{0,1}(
 \ul{S}_r^{\on{ds}}, \ul{u}^* T \ul{M})_{0,p} \\ ( \delta r, \ul{\xi}) \mapsto
 \hh (\ul{J}_{\ul{u}} \circ \d \ul{u} \circ
 D\ul{j}_{\ul{S}_r^{\on{ds}}}(\delta) ) + D_{\ul{S}_r^{\on{ds}},\ul{u}}
 (\ul{\xi}) .\end{multline}
%
Here $D_{\ul{S}_r^{\on{ds}},\ul{u}} (\ul{\xi}) $ is the usual linearized
Cauchy-Riemann operator of e.g. \cite[Chapter 3]{ms:jh}, acting on the
space $\Omega^0( \ul{S}^{\on{ds}}_r, \ul{u}^* T \ul{M}, \partial \ol{u}^* T
\ul{L})_{1,p} $ of sections of $\ol{u}^* T\ul{M}$ of Sobolev class
$W^{1,p}$ with Lagrangian boundary and seam conditions, and
$D\ul{j}_{\ul{S}_r^{\on{ds}}}(\delta r)$ is the infinitesimal variation of the
complex structure on $\ul{S}_r^{\on{ds}}$ determined by $\delta r$.  A
pseudoholomorphic quilt is {\em regular} if the associated linearized
operator is surjective.
\end{enumerate}
\end{definition} 

We introduce the following notation for moduli spaces.  Denote the
moduli space of isomorphism class of pseudoholomorphic quilts with
varying domain
$$\M = \{ (r, \ul{u}: \ul{S}_r^{\on{ds}} \to \ul{M} ) \}/{\rm isomorphism} .$$
Let $\M(\ul{y}_{\ul{e}}, \ul{e} \in \mE)$ be the subspace of pseudoholomorphic
quilts with limits $\ul{y}_{\ul{e}}$ along the ends $\ul{e} \in \mE$, and
$\M(\ul{y}_{\ul{e}}, \ul{e} \in \mE)_d$ the component of formal dimension
$$d = \Ind(D_{\SS, r,\ul{u}}) - \dim(\aut(\ul{S}_r^{\on{ds}})) $$ 
where the last term arises from strip components.

The Gromov compactness theorem has a straight-forward generalization to families
of pseudoholomorphic quilts, as follows.  

\begin{theorem} \label{gromov} {\rm (Gromov compactness for families of quilts)} 
Suppose that $\RR$ is a compact stratified space equipped with a
family of quilts $\SS \to \RR$ with patches labelled by symplectic
backgrounds $\ul{M}$ and boundary/seams labelled by Lagrangians
$\ul{L}$, and $\M$ is the moduli space of pseudoholomorphic quilts
with this data.
\begin{enumerate} 
\item {\rm (Gromov convergence for bounded energy)} Any sequence
  $[\ul{u}_\nu: \ul{S}_{r_\nu} \to \ul{M}]$ in $\M$ with bounded
  energy has a Gromov convergent subsequence, that is, a sequence of
  representatives such that $r_\nu$ converges to some $r \in \RR$,
  there exists a destabilization $\ul{S}^{\on{ds}}_r$ of $\ul{S}_r$, a
  pseudoholomorphic quilt $\ul{u}: \ul{S}^{\on{ds}}_r \to \ul{M}$ and a finite
  bubbling set $Z \subset \ul{S}^{\on{ds}}_r$ such that $\ul{u}_\nu$ converges
  to $\ul{u}$ uniformly in all derivatives on compact subsets of the
  complement of $Z$:
$$ (G_{\gamma(r_\nu)})^* \lim \ul{u}_\nu | C = \ul{u} | C , \quad C
\subset \ul{S} - Z \ \text{ compact} $$
where $G_{\gamma(r_\nu)}: \ul{S}_r^{\on{ds}} \to \ol{S}_{r_\nu}$ is the
identification of domains given by the gluing parameters in \eqref{qglues}.
\item {\rm (Convergence in the admissible, low dimension case)} If in
  addition the formal dimension satisfies $d \leq 1$ and all moduli
  spaces are regular then (sphere and disk) bubbling is ruled out by the
  monotonicity conditions and the bubbling set is empty (although
  there still may be bubbling off trajectories on the strip-like
  ends:)
$$ (d \leq 1) \implies (Z = \emptyset ) .$$
\end{enumerate} 
\end{theorem} 

\begin{proof} The proof of the first statement is a combination of
  standard arguments (exponential decay on strip-like ends, energy
  quantization for sphere and disk bubbles as well as Floer
  trajectories) and left to the reader. In particular, uniform
  exponential decay results are also proved in \cite[Lemma
  3.2.3]{ww:isom} for one varying width; the case of several varying
  widths is similar. For the second statement, suppose a sphere or
  disk bubble develops for some sequence
  $\ul{u}_\nu: \ul{S}_\nu \to \ul{M}$.  Any such sphere or disk bubble
  captures Maslov index at least two, by monotonicity, and so the
  index of the limiting configuration $\ul{u}:\ul{S} \to \ul{M}$
  without the bubbles (obtained by removal of singularities) is at
  most two less than the indices of the maps $\ul{u}_\nu$ in the
  sequence.  But this implies that $\ul{u}$ lies in a component of the
  moduli space with negative expected dimension, contradicting the
  regularity assumption.
\end{proof}  

The Theorem does not quite show that the moduli spaces are compact;
for this one needs to show in addition that convergence in the
topology whose closed sets are closed under Gromov convergence is the
same as Gromov convergence, see \cite[5.6.5]{ms:jh}. 

Next we turn to transversality.  The following is a more precise
version of Theorem \ref{B} of the introduction, that for a
sufficiently generic choice of perturbation data the moduli space is a
smooth manifold of expected dimension.

\begin{theorem} \label{familytransversality} {\rm (Existence of a
    regular extension of perturbation data over the interior)} Let
  $\SS \to \RR$ be a family of quilts with patch labels $\ul{M}$ and
  boundary/seam conditions $\ul{L}$ as in Definition \ref{datum} so
  that $\ul{M},\ul{L}$ are in particular monotone.  Suppose that a
  collection of perturbation data $(\ul{J},\ul{K})$ on the restriction
  of $\SS \to \RR$ to the stratified boundary $\partial_s \RR$ is
  given such that holomorphic quilted surfaces with strip-like ends of
  formal dimension at most one are regular.  Let $U$ be a sufficiently
  small open neighborhood of the stratified boundary $\partial_s \SS$
  with compact complement
\footnote{Hence including some subset of the strip-like
    ends over $\RR$ as well as the entire family over a neighborhood of
    $\partial_s R$ } 
and let $(\ul{J}_0,\ul{K}_0)$ be a pair over $\SS$ agreeing with the
complex structures obtained by gluing on $U$.  There exists a comeager
subset $\PP^{\on{\reg}}(\SS)$ of the set $\PP(\SS)$ of perturbations
$(\ul{J},\ul{K})$ agreeing with $(\ul{J}_0,\ul{K}_0)$ on slightly
smaller open neighborhood $V \subset U$ of $\partial_s \SS$ such that
every holomorphic quilt with strip-like ends with formal dimension at
most one is parametrized regular.
\end{theorem} 

\begin{proof}   
  First we note regularity of any perturbation system for quilts near
  the boundary of the moduli space.  Indeed for sufficiently small
  $U \subset \RR$ such that $\RR \setminus U$ is compact, every
  pseudoholomorphic quilt $(r,u)$ with $r \in U$ of formal dimension
  at most one is regular.  Indeed, otherwise we would obtain by Gromov
  convergence a sequence $(r_\nu,u_\nu)$ with $r_\nu$ converging to a
  point in the boundary of $\RR$.  By Gromov compactness, the maps
  $u_\nu$ converge to a map $u: C \to X$ where the domain $C$ is the
  quilted surface corresponding to $r$ with a collection of disks,,
  spheres, and Floer trajectories added.  After removing these
  additional disk, spheres one obtains a configuration with lower
  energy.  By the energy-index relation \cite[(4)]{ww:quilts}, the
  index of the resulting configuration would be at most $-1$, and so
  does not exist by the regularity assumption on the boundary.  Hence
  disk and sphere does not occur and the domain of $u$ is the quilted
  surface corresponding to $r$, up to the possible addition of strips
  at the strip-like ends.  Since $(r,u)$ is regular by assumption,
  $(r_\nu,u_\nu)$ is also regular by standard arguments involving
  linearized operators.

A regular extension of the perturbation system over the interior of
the family is given by the Sard-Smale theorem.  In order to apply it
we introduce suitable Banach manifolds of almost complex structures.
Let $\J_U(\ul{M},\ul{\omega})$ be the set of all smooth
$\ul{\omega}$-compatible almost complex structures parametrized by
$\SS$, that agree with the original choice on $\SS\setminus
\pi^{-1}(V)$ and with the given choices on the images of the
strip-like ends.  For a sufficiently large integer $l$ let $\J^l$
denote the completion of $\J_U(\ul{M},\ul{\omega})$ with respect to
the $C^l$ topology.  The tangent space to $\J^l$ is the linear space
\label{isthelinearspace}
$$ \T_J^l := T_{\ul{J}} \J^l = \left\{ \delta \ul{J} \in
C^l(\SS\times T\ul{M}, T\ul{M}) \left| 
\begin{array}{l}  \delta \ul{J} \circ
\ul{J} + \ul{J}\circ \delta \ul{J} = 0 \\  
\ul{\omega}(\delta \ul{J}(\ul{v}), \ul{w}) + \ul{\omega}(v, \delta
\ul{J} (w)) = 0 \\   \delta\ul{J} \lvert_{\SS\setminus
  \pi^{-1}(U)} \equiv 0 \end{array} \right. \right\} .$$
For small $\delta \ul{J}$ there is a smooth exponentiation map to
$\J^l$, given explicitly by
$\delta \ul{J} \mapsto \ul{J} \circ \exp(-\ul{J}\circ \delta \ul{J})$.
Similarly we introduce suitable Banach spaces of Hamiltonian
perturbations.  Let $\mathcal{K}^l$ denote the completion in the $C^l$
norm of the subset $\Omega^1_{\SS/\RR}(\cH(\ul{M}); \ul{K})$ of
$\Omega^1_{\SS/\RR}(\cH(\ul{M}))$ consisting of 1-forms that equal
$\ul{K}$ on the complement of the inverse image of $V$.  The tangent
space $\T_{K}^l := T_{\ul{K}} \mathcal{K}^l$ consists of
$\delta \ul{K} \in \Omega^1_{\SS/\RR}(\cH(\ul{M}))$ such that
$\delta \ul{K} \lvert_{\SS\setminus V} \equiv 0$.  Such elements can
be exponentiated to elements of $\mathcal{K}^l$ via the map
$\ul{K} + \delta \ul{K}$.

Construct a smooth universal space of pseudoholomorphic quilts as follows.
Let
$$\B = \{ (r, \ul{u}) | r \in \RR, \ul{u}: (\ul{S}_r,\ul{I}_r) \to
(\ul{M}, \ul{L}) \} $$
denote the space of pairs $r$ in the parameter space $\RR$ and maps
$\ul{u}$ from $\ul{S}_r$ to $\ul{M}$ of Sobolev class $W^{1,p}$ with
seams/boundaries in $\ul{L}$ and
$$\cE = \bigcup_{(r,\ul{u}) \in \B} \Omega^{0,1}_{\SS/\RR}(\ul{u}^* T
\ul{M})_{0,p} .$$
Here it suffices to consider the case that the domain of $\ul{u}$ is a
stable surface, since the strip components are already assumed
regular.  Since the bundle $\SS \to \RR$ is
trivializable, \label{trivable} $\cE$ is a Banach vector bundle of
class $C^{l-1}$.  The universal moduli space
$\M^{\on{univ}}(\ul{y}_{\ul{e}}, \ul{e} \in \mE) \subset \B \times
\J^l\times \mathcal{K}^l$
is a Banach submanifold of class $C^{l-1}$, by a discussion parallel
to \cite[Lemma 3.2.1]{ms:jh}.  Indeed, the universal moduli space
$\M^{\on{univ}}(\ul{y}_{\ul{e}}, \ul{e} \in \mE)$ is the intersection
of the section
$$ \ol{\partial}: \B \times \J^l\times \mathcal{K}^l \to \cE, \quad
(r, \ul{u}, \ul{J}, \ul{K}) \mapsto \ol{\partial}_{\ul{J},
  \ul{K}}(r, \ul{u}) $$
with the zero-section of the bundle $\cE \to \B$:
$$\M^{\on{univ}}(\ul{y}_{\ul{e}},
\ul{e} \in \mE) := \ol{\partial}^{-1}(0) .$$
To show that the universal moduli space is a Banach manifold, it
suffices to show that the linearized operator
\begin{multline}\label{linop2}
 D_{\SS, r, \ul{u},J, K} : T_{(r,\ul{u})} \B \times \T_J^l \times
 \T_K^l \lra \cE_{\SS,r,\ul{u}} \\ ( \delta r, \ul{\xi}, \delta
 \ul{J}, \del \ul{K}) \mapsto (\delta \ul{Y})^{0,1} +
 \ul{j}_{\ul{S}_r} \circ \frac{1}{2}(d\ul{u} - \ul{Y})\circ \delta
 \ul{J} + D_{\SS, r,\ul{u}}(\delta r, \ul{\xi})
\end{multline}
is surjective at all $(r,\ul{u},\ul{J},\ul{K})$ for which
$\ol{\partial}_{(\ul{J},\ul{K})}(r,\ul{u}) = 0$.  Since the last
operator in \eqref{linop2}
is Fredholm, the image of $D_{\SS, r, \ul{u}, \ul{J},
  \ul{K}}$ is closed.  

We prove that the linearized operator cutting out the universal moduli
space is surjective.  Suppose that the cokernel is not zero.  By the
Hahn-Banach theorem, there exists a linear functional
$$\eta \in (\cE_{\SS,r,\ul{u}})^{*} = L^q(\ul{S}_r,
(\Omega_{\ul{S}_r}^{0,1}\otimes \ul{u}^*T\ul{M})),\quad 1/p + 1/q = 1 $$
that is non-zero and that vanishes on the image of the linearized
operator.  In particular $\eta$ vanishes on the image of $D_{\ul{S}_r,
  \ul{u}}$, i.e.
\bea \int_{\ul{S}_r} \langle D_{\ul{S}_r, \ul{u}} (\ul{\xi}), \eta\rangle = 0 =
\int_{\ul{S}_r} \langle \ul{\xi}, D_{\ul{S}_r, \ul{u}}^* (\eta)\rangle \eea 
for all $\ul{\xi} \in W^{1,p}(\ul{S}_r, \ul{u}^*T\ul{M})$.  This
argument implies $D_{\ul{S}_r, \ul{u}}^* (\eta) = 0$, and elliptic
regularity ensures that $\eta$ is of class $C^l$ at least in the
interiors of the patches of the quilt $\ul{S}_r$.  To prove that
$\eta$ is zero, it suffices to show that it vanishes on an open subset
of each patch of the quilt $\ul{S}_r$, since by unique continuation
for solutions of $D_{\ul{S}_r,\ul{u}}^*\eta = 0$ it follows that it
vanishes on all of $\ul{S}_r$.  We may assume that the complement of
the images of the strip-like ends contains such an open subset.
Considering the image of $(0,0, 0, \delta \ul{K})$ shows that
\bea \int_{\ul{S}_r} \langle (\delta
\ul{Y})^{0, 1}, \eta \rangle = 0 \eea 
for all $\delta \ul{Y}$, where $\delta \ul{Y}$ is the Hamiltonian
vector field associated to $\delta \ul{K}$.  But now, for each $z$ in
the complement of the inverse image of $U$ and the complement of the
images of the strip-like ends, there exists a sequence of functions
$\delta \ul{K}_n$ in $\T_{\ul{K}}^l$ that are supported on
successively smaller neighborhoods of $z$ and such that $ (\delta
\ul{Y}_n)^{0,1}$ converges to the delta function $\delta_z \otimes
\eta(z)$.  It follows that there exists a limit
\bea |\eta (z)|^2 = \lim_{n\to\infty}\int_{\ul{S}_r} \langle (\delta
\ul{Y}_n)^{0,1}, \eta \rangle = 0.  \eea 
So $\eta(z) = 0$ on an open subset of each patch of the quilt.  By
unique continuation it must vanish everywhere. Thus, $\eta = 0$, which
is a contradiction. Hence the linearized operator is surjective, so by
the implicit function theorem for $C^{l-1}$ maps of Banach spaces,
$\M^{\on{univ}}(\ul{y}_{\ul{e}}, \ul{e} \in \mE)$ is also a Banach
manifold of class $C^{l-1}$.  One can now consider the projection
\bea \Pi: \M^{\on{univ}}(\ul{y}_{\ul{e}}, e \in \mE)_d \to \J^l \times
\K^l, \ \ (r,\ul{u}, \ul{J}, \ul{K}) \mapsto (\ul{J}, \ul{K}) \eea
on the subset $\M^{\on{univ}}(\ul{y}_{\ul{e}}, \ul{e} \in \mE)_d $ of
parametrized index $d$, which is a Fredholm map between Banach
manifolds of index $d$.  By the Sard-Smale theorem, for $l >
\max(d,0)$ the subset of regular values 
$$\PP_{\reg}^l = \{ (\ul{J},\ul{K}) \in  \J^l
\times \K^l \ | \  \coker(D_{r,\ul{u},\ul{J},\ul{K}} \Pi) = 0, \ \ 
\forall (r,\ul{u}, \ul{J}, \ul{K}) \in \M^{\on{univ}}(\ul{y}_{\ul{e}} ) \} $$
is comeager, hence dense. Now the regular values of the projection
correspond precisely to regular perturbation data for the moduli
spaces $\M(\ul{y}_{\ul{e}}, \ul{e} \in \mE)$, thus the subset of
regular $C^l$-smooth perturbation data comeager in $\J^l \times \K^l$.
 
The final step is to pass from $C^l$-smooth regular perturbation data,
to $C^\infty$ regular perturbation data. This is a standard argument
due to Taubes, see Floer-Hofer-Salamon \cite{fhs:tr} and
McDuff-Salamon \cite{ms:jh} for its use in pseudoholomorphic curves,
which we explain in the unquilted case for simplicity.  Let us write
$$\J := \bigcap_{l\geq 0} \J^l, \quad \K := \bigcap_{l\geq 0} \K^l,
\quad \PP := \J\times \K$$
with the $C^\infty$ topology on each factor.  Let $C > 0$ be a
constant such that any pseudoholomorphic quilt has exponential decay
satisfying $| \d u (s,t) | < \exp( -C s)$ for all $(s,t)$ coordinates
on each end $\ul{e} \in \mE$, see \cite[Theorem 5.2.4]{ww:quiltfloer}.
Let $\psi: S \to \R$ be a positive function given by $\psi(s,t) = s$
on each end, for each surface $S_r, r \in \RR$.  For $k > 0$, let
$\PP_{\reg, k} \subset \PP$ consist of the perturbation data for which
the associated linearized operators $D_{\SS, \ul{u}}$ are surjective
for all $\ul{u} \in \M(\ul{y}_{\ul{e}}, \ul{e} \in \mE)$ satisfying
\begin{eqnarray}
  \ind D_{\SS, \ul{u}} \in \{ 0, 1 \},\ \ \|\d \ul{u} \exp( C
  \psi)\|_{L^\infty} \leq k.
\end{eqnarray}
The index restrictions above suffice since we only consider moduli
spaces of expected dimension 0 and 1.  We will show that the set
$$ \PP_{\reg} := \bigcap_{k > 0 } \PP_{\reg, k}$$
is comeager in $\PP$, by showing that each of the sets $\PP_{\reg, k}$
is open and dense in $\PP$ with respect to the $C^\infty$ topology.

To show that $\PP_{\reg,k}$ is open, consider a sequence
$\{(\ul{J}_\nu, \ul{K}_\nu)\}_{\nu=1}^{\infty}$ in the complement of
$\PP_{\reg, k}$, converging in the $C^\infty$ topology to a pair
$(\ul{J}_\infty, \ul{K}_\infty)$.  We claim that $(\ul{J}_\infty,
\ul{K}_\infty) \notin \PP_{\reg,k}$. By assumption, there exists a
sequence $\ul{u}_{\nu}$ such that
$$\ol{\partial}_{\ul{J}_\nu, \ul{K}_\nu} \ul{u}_\nu = 0 \quad \ind
D_{\SS, \ul{u}} \in \{ 0 , 1 \}, \quad \|\ul{u}_\nu\|_{L^\infty} \leq
k .$$
We may take the elements of the sequence to have the same index.  The
uniform bound on the derivative $\|\d \ul{u}_\nu \exp(C \psi)
\|_{L^\infty} \leq k$ implies that $\ul{u}_\nu$ converges to a
pseudoholomorphic quilt.  Since surjectivity of Fredholm operators is an
open condition, $D_{\SS, \ul{u}_\nu}$ must be surjective for
sufficiently large $\nu$. This argument proves that $\PP_{\reg, k}$ is
open in $\PP$ for each $k>0$.

To show that $\PP_{\reg, k}$ is dense, note that we can write
$\PP_{\reg,k} = \PP^l_{\reg,k} \cap \PP$, where the definition of
$\PP^l_{\reg,k}$ is the same as the definition for $\PP_{\reg,k}$, but
as a subset of $\PP^l$.  The argument given above to prove that
$\PP_{\reg,k}$ is open in $\PP$ with respect to the $C^\infty$
topology can be repeated to show that for all sufficiently large $l$
the subset $\PP_{\reg,k}^l$ is open in $\PP^l$ with respect to the
$C^l$ topology.  The set $\PP_{\reg}^l$ is dense $\PP^l$,
and since $\PP_{\reg, k}^l \supset \PP_{\reg}^l$, this implies that
$\PP_{\reg,k}^l$ is dense in $\PP^l$.  So fix $(\ul{J}, \ul{K}) \in
\PP$.  We find a sequence $(\ul{J}_\nu, \ul{K}_\nu) \in \PP_{\reg,k}$
that converges to $(\ul{J}, \ul{K})$ in the $C^\infty$ topology.  
Consider a sequence
$$(\ul{J}_l, \ul{K}_l) \in \PP_{\reg,k}^l, \quad\|
\ul{J}_l - \ul{J}\|_{C^l} + \|\ul{K}_l - \ul{K}\|_{C^l} \leq 2^{-l} .$$
Such a sequence exists because $\PP_{\reg,k}^l$ is dense in $\PP^l$
for each $l$, and $(\ul{J}, \ul{K}) \in \PP \subset \PP^l$.  Now, by
assumption $\PP_{\reg,k}^l$ is open in $\PP^l$, and so for each
$(\ul{J}_l,\ul{K}_l) \in \PP^l$ there exists an $\eps_l >0$ such that
$$
\|\ul{J} - \ul{J}_l\|_{C^l} + \|\ul{K} - \ul{K}_l\|_{C^l} < \eps_l \implies (\ul{J}, \ul{K}) \in \PP_{\reg,k}^l $$
for all $(\ul{J},\ul{K}) \in \PP^l$. Finally, $\PP$ is dense in $\PP^l$
for each $l$ (i.e. $C^\infty$ functions are dense in the space of
$C^l$ functions). Therefore, for each $l$ we may find an element
$(\tilde{\ul{J}}_l, \tilde{\ul{K}}_l) \in \PP$ such that
$$
\| \tilde{\ul{J}}_l - \ul{J}_l \|_{C^l} + \|\tilde{\ul{K}}_l -  \ul{K}_l\|_{C^l} < \min \{\eps_l, 2^{-l} \}. $$
Thus, every term in the sequence $(\tilde{\ul{J}}_l,
\tilde{\ul{K}}_l)$ is in $\PP \cap \PP_{\reg,k}^l = \PP_{\reg,k}$, and
it converges in all $C^l$ norms, hence in the $C^\infty$ topology, to
the pair $(\ul{J},\ul{K})$.  Thus, $\PP_{\reg} = \underset{k>0}{\cap}
\PP_{\reg,k}$ is a countable intersection of open, dense sets in
$\PP$ as claimed.
\end{proof}

\begin{remark}
\begin{enumerate} 
\item 
 {\rm (Zero and one-dimensional components of the moduli spaces)} For
 $d = 0$, the moduli space $\M_\SS(\ul{y}_{\ul{e}}, \ul{e} \in \mE)_d$
 lies entirely over the highest dimensional strata of $\RR$.  On the
 other hand for $d = 1$ the intersection with the highest dimensional
 strata is one-dimensional, while the intersection with the
 codimension one strata is a discrete set of points.
\item {\rm (Comparison with Seidel)} Seidel's book \cite{se:bo} uses
  perturbations that are supported arbitrarily close to the boundary.
  The advantage of these is that one can make the higher-dimensional
  moduli spaces regular as well.  However, only the zero and
  one-dimensional moduli spaces are needed here.
\end{enumerate} 
\end{remark} 

\begin{remark} {\rm (Orientations for families of pseudoholomorphic quilts)}  
To define family quilt invariants over the integers we require that
the moduli spaces are oriented.  Orientations on the moduli spaces may
be constructed as follows \cite{orient}.  At any element $(r,\ul{u})
\in \M_\SS(\ul{y}_{\ul{e}}, \ul{e} \in \mE)$ the tangent space to the
moduli space of pseudoholomorphic quilts is the kernel of the linearized
operator \eqref{linop}.  The operator $D_{\ul{u}}$ is canonically
homotopic to the operator $0 \oplus D_{\ul{u},r}$ (the latter is the
operator for the trivial family $ \{ r\}$, that is, the unparametrized
linearized operator) via a path of Fredholm operators.  This induces
an isomorphism
\begin{equation} \label{split1}
 \det(T_{(r,\ul{u})} \M(\ul{y}_{\ul{e}}, \ul{e} \in \mE)) \to \det(T_r
 \RR) \otimes \det(D_{\ul{u}}) .\end{equation}
 First one deforms the seam conditions $ \ul{u}^* T \ul{L}$ to
 condition of split type, that is for each seam $I$ adjacent to
 patches $S_{p_-}, S_{p_+}$ deform the map $I \to \Lag(M_{p_-} \times
 M_{p_+})$ defined by $u,\ul{L}$ to a map $I \to \Lag(M_{p_-}) \times
 \Lag( M_{p_+})$.  This deformation identifies the corresponding
 determinant lines and reduces the claim to the case of an unquilted
 pseudoholomorphic map $u: \ul{S}_r \to M$ with boundary condition $L$.  The
 determinant line $\det(D_{u})$ is oriented by ``bubbling off
 one-pointed disks'', see \cite[Theorem 44.1]{fooo} or \cite[Equation (36)]{orient}.  The
 orientation at $u$ is determined by an isomorphism \label{familytransversality2}
 \begin{equation} \label{split2} \det(D_{{u}}) \cong \DD^+_{{x}_0}
   \DD^-_{x_1} \ldots \DD^-_{x_d}
 \end{equation}
where $\DD^+_{x_j}$ are determinant lines associated with one-marked
disks with marking ${x}_j$, $\DD^-_{x_j}$ is the tensor product of the
determinant line for the once-marked disk with $\det(T_{x_j}L)$ and
the orientations on $\DD^\pm_{x_j}$ are chosen so that there is a
canonical isomorphism $\DD^-_{x_j} \otimes \DD^+_{x_j} \to \R .$ The
isomorphism \eqref{split2} is determined by degenerating surface with
strip-like ends to a nodal surface with each end replaced by a disk
with one end attached to the rest of the surface by a node.  The
boundary condition on these disks is given by a chosen path of
Lagrangian subspaces in the tangent space at the end.  Furthermore,
the Lagrangian boundary condition is deformed to a constant boundary
condition using the relative spin structure.  These choices are
analogous to choice of orientations on the tangent spaces to the
stable manifolds in Morse theory, on which the orientations of the
moduli spaces of Morse trajectories depend.
\end{remark} 

The master equation for family quilt invariants is a consequence of
the following description of the boundary of the one-dimensional
moduli spaces of quilts:

\begin{theorem} \label{masterthm} \label{familymod} {\rm (Description of the boundary of
    one-dimensional moduli spaces of pseudoholomorphic quilts)}
  Suppose that $\SS \to \RR$ is a family of quilted surfaces over a
  compact stratified space $\RR$ with a single open stratum denoted
  $\SS_0 \to \RR_0$ labelled with monotone symplectic data
  $\ul{M},\ul{L}$, and $\ul{J},\ul{K}$ are a regular set of
  perturbation data.  Then for any limits
  $\ul{y}_{\ul{e}}, \ul{e} \in \mE$
\begin{enumerate} 
\item {\rm (Zero-dimensional component)} the zero-dimensional
  component $\M_\SS(\ul{y}_{\ul{e}}, \ul{e} \in \mE)_0$ of the moduli space
  of pseudoholomorphic quilts for $\SS$ is a finite set of points and
\item {\rm (One-dimensional component)} the one-dimensional component
  $\M_{\SS_0}(\ul{y}_{\ul{e}}, \ul{e} \in \mE)_1$ has a compactification as
  a one-manifold with boundary
\begin{multline} 
 \partial \M_\SS(\ul{y}_{\ul{e}}, \ul{e} \in \mE)_1 = \M_{\partial
  \SS}(\ul{y}_{\ul{e}}, \ul{e} \in \mE)_1 \\ \cup \bigcup_{{\ul{f}} \in \mE} \M(y_{\ul{f}},y_{\ul{f}}')_0
\times 
\M_{\SS}(\ul{y}_{\ul{e}}, \ul{e} \in \mE; \ul{y}_{\ul{f}} \mapsto 
\ul{y}_{\ul{f}}')_0
\end{multline}
with sign of inclusion given by $+1$ for the first factor and $\pm 1$
for the second factor, depending on whether $\ul{y}_{\ul{f}}$ is an
incoming or outgoing end.  Here
$\M_{\SS}(\ul{y}_{\ul{e}}, \ul{e} \in \mE; \ul{y}_{\ul{f}} \mapsto
\ul{y}_{\ul{f}}')_1 $
denotes $\M_{\SS}(\ul{y}_{\ul{e}}, \ul{e} \in \mE)$ with the end label
$\ul{y}_{\ul{f}}$ replaced by $\ul{y}_{\ul{f}}')_1 $ while
$ \M(y_{\ul{f}},y_{\ul{f}}')_0$ denotes the space of Floer
trajectories from $y_{\ul{f}}$ to $y_{\ul{f}}'$ of formal dimension
$0$.
\end{enumerate} 
\end{theorem} 

\begin{proof}  
  The gluing theorem is proved in the same way as for Ma'u
  \cite{mau:gluing}, who considered the gluing along strip-like ends
  that arise in the definition of the generalized Fukaya category.
Compactness is Theorem \ref{gromov}.
The claim on orientations is proved in
  \cite{orient}.
\end{proof}

Finally we use the moduli spaces of quilts to construct chain-level
invariants.  Let $\RR$ be a stratified space labelled by quilt data
$\ul{M},\ul{L}$ as in Theorem \ref{B}, and $\SS \to \RR$ a family of
quilted surfaces with strip-like ends constructed in Section
\ref{varying}.

\begin{definition} {\rm (Family quilt invariants)}  
Given a regular pair $(\ul{J},\ul{K})$ as in Theorem \ref{B} we define
a (cochain level) {\em family quilt invariant}
$$ \Phi_{\SS} :\bigotimes_{\ul{e} \in \E_-(\SS)} CF(\ul{L}_{\ul{e}}) 
\to  \bigotimes_{\ul{e} \in \E_+(\SS)} CF(\ul{L}_{\ul{e}})
$$
by 
\begin{equation*} \label{defrel}
 \Phi_{\SS} \biggl( \bigotimes_{\ul{e} \in \E_-} \bra{\ul{x}^-_{\ul{e}}}
 \biggr) := \sum_{\ul{u} \in
   \M_{\SS}(\bra{\ul{x}_{\ul{e}}}_{\ul{e} \in \E_-},\bra{\ul{x}_{\ul{e}}}_{\ul{e} \in \E_+})_0} \sigma(u)
    \bigotimes_{\ul{e} \in \E_+} \bra{\ul{x}^+_{\ul{e}}}
   ,\end{equation*}
where
$$ \sigma: \M_{\SS}(x_{\ul{e}}, \ul{e} \in \mE)_0 \to \{-1,+1\} $$
is defined by comparing the orientation to the canonical orientation
of a point.  
\end{definition} 

Theorem \ref{C} follows from Theorem \ref{familymod} and the following
discussion of orientations.  In particular, if $\RR$ has no
codimension one strata, then $\Phi_{\SS}$ is a cochain map.  The case
that $\RR$ is a point was considered in \cite{ww:quilts}.

\section{The Fukaya category of generalized Lagrangian branes}
\label{parttwo}

In the remainder of the paper we apply the results of the first two
sections to construct \ainfty categories, \ainfty functors, \ainfty
pre-natural transformations and \ainfty homotopies and prove Theorems
\ref{mainfunc} and \ref{maincompose} from the introduction.  The
Fukaya category of a symplectic manifold, when it exists, is an
\ainfty category whose objects are Lagrangian submanifolds with
certain additional data, and morphism spaces are Floer cochain spaces.
In \cite{we:co} we explained that in order to obtain good
functoriality properties one should allow certain more general
objects, which we termed {\em generalized Lagrangian branes},
comprised of {\em sequences of Lagrangian correspondences}.  The
necessary analysis for defining Fukaya categories with these
generalized objects for compact monotone symplectic manifolds was
\label{parttwo2} developed by the first author in \cite{mau:gluing}, and extends the
constructions of Fukaya \cite{fuk:garc} and Seidel \cite{se:bo} to
include generalized Lagrangian branes as introduced in
\cite{we:co}.

\subsection{Quilted Floer cochain groups}
\label{concat} 

In this section we review the construction of Floer cochain groups for
certain symplectic manifolds with additional structure.  The cochain
groups are the morphism spaces in the version of the Fukaya category
on which our functors are defined.  We begin by stating the technical
hypotheses under which our Floer cochain complexes are well-defined.

\begin{definition} \label{backgrounds} 
{\rm (Symplectic backgrounds)}  
Fix a monotonicity constant $\tau\geq 0$ and an even integer $N > 0$.
A {\em symplectic background} is a tuple $(M,\omega,b,\Lag^N(M))$ as
follows.
\begin{enumerate}
\item {\rm (Bounded geometry)} $M$ is a smooth manifold, which is
  compact if $\tau>0$.
\item {\rm (Monotonicity)} $\omega$ is a symplectic form on $M$ which
  is monotone, i.e.\ $[\omega] = \tau c_1(TM) $ and if $\tau=0$
  then $M$ satisfies ``bounded geometry'' assumptions as in e.g.
  \cite{se:bo}.
\item {\rm (Background class)} $b \in H^2(M,\Z_2)$ is a {\em
  background class}, which will be used for the construction of
  orientations.
\item {\rm (Maslov cover)} $\Lag^N(M) \to \Lag(M)$ is an $N$-fold
  Maslov cover in the sense of \cite{se:gr}, \cite{ww:quiltfloer} such
  that the induced $2$-fold Maslov covering $\Lag^2(M)$ is the
  oriented double cover.
\end{enumerate}
\end{definition}

We often refer to a symplectic background
$(M,\omega,b,\Lag^N(M))$ as $M$.

\begin{example}  {\rm (Point background)}  
The point $M=\{\pt\}$ can be viewed as a canonical $\tau$-monotone,
$N$-graded symplectic background $(\{\pt\},\omega=0,b=0,\Lag^N(\pt))$,
which we denote by $\pt$.
\end{example} 

Next introduce Lagrangian branes, which will be the objects of the
Fukaya categories we consider.  Let $M$ be a symplectic background.

\begin{definition} \label{admissible}  {\rm (Admissible Lagrangians)} 
\ben
\item 
A Lagrangian submanifold $L \subset
M$ is {\em admissible} if
\ben
\item $L$ is compact and oriented;
\item $L$ is monotone, that is, for $u : (D,\partial D) \to (M,L)$ the
  symplectic action $A(u)$ and index $I(u)$ are related by
$$ 2 A(u) = \tau I(u) \quad \forall u : (D,\partial D) \to (M,L) , $$
where $\tau$ is the monotonicity constant for $M$;
\item $L$ has minimal Maslov number at least $3$, or minimal Maslov
  number $2$ and disk invariant $\Phi_L = 0$ in the sense of
  \cite{oh:fl1} (that is, the signed count of Maslov index $2$ disks
  with boundary on $L$); and
\item the image of $\pi_1(L)$ in $\pi_1(M)$ is torsion, for any
  choice of base point.  \een
\item 
An {\em admissible grading} of an oriented Lagrangian submanifold $L
\subset M$ is a lift 
$$\sigma_L^N: L \to \Lag^N(M)$$ 
of the canonical section $L \to \Lag(M)$ such that the induced lift
$\sigma_L^2$ equals to the lift induced by the orientation.  See
\cite{ww:quiltfloer} for details.
\item A {\em relative spin structure} on an admissible Lagrangian
  submanifold $L\subset M$ with respect to the background class 
$$b \in H^2(M,\Z_2)$$ 
is \cite{fooo},\cite{orient} a lift of the class of $TL$ defined in
the first relative \v{C}ech cohomology group for the inclusion $i: L
\to M$ with values in $\on{SO}(\dim(L))$ to first relative \v{C}ech
cohomology with values in $\on{Spin}(\dim(L)))$, with associated class
$b$.  \een
\end{definition}  

Recall that a {\em Lagrangian correspondence} is a Lagrangian
submanifold of a product of symplectic manifold with the symplectic
form on the first factor reversed.  Given symplectic manifolds $M_0$
and $M_1$ and a Lagrangian correspondence
$L_{01} \subset M_0^- \times M_1$, the {\em transpose} of $L_{01}$ is
the generalized Lagrangian correspondence $L_{01}^T$ from $M_1$ to
$M_0$ obtained by applying the anti-symplectomorphism
$M_0^- \times M_1 \to M_1^- \times M_0, \ (m_0,m_1) \mapsto (m_1,m_0)$
to $L_{01}$.

\begin{definition}    \label{branes} 
{\rm (Generalized Lagrangian branes)} Let ${M_s} :=
  (M_s,\omega_s,b_s,\Lag^N(M_s))$ and ${M_t} :=
  (M_t,\omega_t,b_t,\Lag^N(M_t))$ be two symplectic backgrounds.  A
  {\em generalized Lagrangian brane from ${M_s}$ to ${M_t}$} is a tuple
  $\ul{L}=(L_{(j-1)j})_{j=1,\ldots,r}$ of length $r \ge 0$ of
  Lagrangian correspondences equipped with gradings, relative spin
  structures, and widths as follows.
\begin{enumerate}
\item {\rm (Sequence of backgrounds)} 
$(N_i,\omega_i,b_i,\Lag^N(N_i))_{i=0,\ldots,r}$ is a sequence of symplectic backgrounds such that $N_0 = M_s$ and
  $N_r = M_t$ as symplectic backgrounds;
\item {\rm (Sequence of correspondences)} $L_{(j-1)j} \subset
  N_{j-1}^-\times N_{j}$ is an admissible Lagrangian submanifold for
  each $j=1,\ldots,r$ with respect to $-\pi_{j-1}^* \omega_{j-1} +
  \pi_{j}^* \omega_{j}$, where $\pi_{j-1},\pi_j$ are the projections
  to the factors of $N_{j-1}^-\times N_{j}$;
\item {\rm (Gradings)} a grading on $\ul{L}$, by which we mean a
  collection of gradings
$$\sigma_{L_{(j-1)j}}^N: L_{(j-1)j} \to \Lag^N(N_{j-1}^-\times N_j)$$
for $j=1,\ldots,r$ with respect to the Maslov cover induced by the
product of covers of $N_{j-1}$ and $N_j$.  
\item {\rm (Relative spin structures)} a relative spin structure on
  $\ul{L}$ is a collection of relative spin structures on $L_{(j-1)j}$
  for $j=1,\ldots,r$ with background classes $-\pi_{j-1}^* b_{j-1} +
  \pi_{j}^* b_{j}$;
\item {\rm (Widths)} a collection of {\em widths}
  $\ul{\delta}=(\delta_j>0)_{j=1,\ldots,r-1}$.
\end{enumerate}
Let ${M} := (M,\omega,b,\Lag^N(M))$ be a symplectic background.  Then
a {\em generalized Lagrangian brane in ${M}$} is a generalized
Lagrangian brane from $\pt$ to $M$.
\end{definition} 

Next we define brane structures on Lagrangian correspondences.  Given
symplectic backgrounds $M_s, M_t,M_u$ with the same monotonicity
constant, admissible generalized Lagrangian correspondence branes
$\ul{L}^+,\ul{L}^-$ from $M_s$ to $M_t$ resp. $M_t $ to $M_u$ with the
same background class in $M_t$ and width $\eps > 0$ we can concatenate
them to obtain a generalized Lagrangian correspondence
$\ul{L}^+\sharp_\eps \ul{L}^-$ from $M_s$ to $M_u$.  More precisely,
we define $\ul{L}:=\ul{L}^+ \sharp_\eps (\ul{L}^-)^T$ to be the
generalized Lagrangian correspondence with gradings and relative spin
structures, given as follows:
\begin{enumerate}
\item
symplectic backgrounds with the same monotonicity constant
indexed up to $r:=r^++r^-$
\begin{align*}
(N_0,\ldots,N_r) :=
\bigl(M_s=N^+_0,\ldots,N^+_{r^+}=M_t=N^-_{r^-},\ldots,N^-_0=M_u \bigr);
\end{align*}
\item the admissible Lagrangian submanifolds
\begin{equation} \label{algcomp}
(L_{01},\ldots,L_{(r^++r^--1)(r^++r^-)}):=
  \bigl(L^+_{01},\ldots,L^+_{(r^+-1)r^+},L^-_{01}, \ldots ,L^-_{(r^-
    - 1)r^-} \bigr);
\end{equation}
\item the relative spin structures on $L^+_{(j-1)j}$ for
  $j=1,\ldots,r^+$ and the relative spin structures on $L_{(j-1)j}^-$
  induced from those on $L_{(j-1)j}^-$ for $j=1,\ldots,r^-$;
\item the widths are those of $\ul{L}^+,\ul{L}^-$ together with
  $\eps$.
\end{enumerate}

\noindent In particular given symplectic backgrounds $M_s, M_t$ with the same
monotonicity constant, admissible generalized Lagrangian
correspondence branes $\ul{L}^+,\ul{L}^-$ from $M_s$ to $M_t$, and
width $\eps > 0$ we can transpose one and then concatenate them to
obtain a cyclic Lagrangian correspondence $\ul{L}^+\sharp_\eps
(\ul{L}^-)^T$.  Here the gradings $\sigma_{L^-_{(j-1)j}}^N$ of
$L^-_{(j-1)j}$ for $j=1,\ldots,r^+$ inducing gradings of
$(L_{(j-1)j}^-)^T$ for $j=1,\ldots,r^-$ and similarly for the relative
spin structures.  The resulting sequence can be visualized as
$$
\begin{diagram} \node{M_s=N^+_0} \arrow{s,r}{=} 
\arrow{e,t}{L^+_{01}} 
\node{\ldots} \arrow{e,t}{L^+_{(r^+-1)r^+}} \node{N^+_{r^+} = M_t} \arrow{s,r}{=} 
\\
\node{M_s = N_0^-} 
\node{\ldots}
\arrow{w,t}{(L_{01}^-)^T} 
\node{ N^-_{r^-} = M_t}
\arrow{w,t}{(L_{(r^--1)r^-}^-)^T} 
\end{diagram} 
$$

The Floer cohomology of a cyclic generalized Lagrangian correspondence
is defined as follows.  Choose regular Hamiltonian perturbations
$$\ul{H}\in \oplus_{j=0}^r C^\infty([0,\delta_j] \times N_j) ,$$ 
and almost complex structures
$$\ul{J}\in \oplus_{j=0}^r C^\infty([0,\delta_j] ,\J(N_j,\omega_j)) $$
as in \cite{ww:corr}.  The generators of the quilted Floer cochain
complex of a generalized Lagrangian brane $\ul{L}$ are the perturbed
intersection points
$$ \cI(\ul{L}) := \left\{ \ul{x}=\bigl(x_j: [0,\delta_j] \to
N_j\bigr)_{j=0,\ldots,r} \, \left|
\begin{aligned}
\dot x_j(t) = Y_j(t,x_j(t)), \\ (x_{j}(\delta_j),x_{j+1}(0)) \in
L_{j(j+1)}
\end{aligned} \quad \forall j \right.\right\} .
$$
Here $Y_j$ is the Hamiltonian vector field corresponding to $H_j$.
The gradings on $\ul{L}$ induce a grading $|\ul{x}|\in \Z_N$ for
$\ul{x}\in\cI(\ul{L})$, and hence induce a $\Z_N$-grading on the
space of quilted Floer cochains
$$ CF(\ul{L}) := \bigoplus_{\ul{x}\in\cI(\ul{L})} \Z \bra{\ul{x}} =
\bigoplus_{k \in \Z_{N}} CF^k(\ul{L}), \qquad CF^k(\ul{L}) :=
\bigoplus_{|\ul{x}|=k} \Z \bra{\ul{x}} . $$
The Floer coboundary operator is defined by counts of the moduli
spaces of quilted pseudoholomorphic strips,
$$\partial : \ CF^\bullet(\ul{L}) \to CF^{\bullet+1}(\ul{L}), \quad \bra{\ul{x}_-} 
\mapsto \sum_{\ul{x}_+\in\cI(\ul{L})} \Bigl(
\sum_{\ul{u} \in\M(\ul{x}_-,\ul{x}_+)_0} \eps(\ul{u})\Bigr)
\bra{\ul{x}_+} ,
$$
where the signs 
\begin{equation} \label{epseq}
\eps: \M(\ul{x}_-,\ul{x}_+)_0 \to \{ \pm 1 \} \end{equation} 
are given by the orientation of the moduli space
$$ \M(\ul{x}_-,\ul{x}_+)_0 := \bigl\{ \ul{u}=\bigl( u_j : \R \times
[0,\delta_j] \to N_j \bigr)_{j=0,\ldots,r} \,\big|\, \eqref{Jjhol} - \eqref{ulim2}, \Ind(D_{\ul{u}})= 1 \bigr\} / \R
$$
of tuples of pseudoholomorphic maps
\beq \label{Jjhol}
\overline\partial_{J_j,H_j} u_j = \partial_s u_j + J_j \bigl(\partial_t u_j - Y_j(u_j)  \bigr) = 0 
\qquad\forall j = 0,\ldots r ,
\eeq
satisfying the seam conditions
\beq\label{ubc}
(u_{j}({s},\delta_j),u_{j+1}({s},0)) \in L_{j(j+1)}\qquad\forall j = 0,\ldots r ,\ s\in\R ,
\eeq
with finite energy 
\beq \label{ulim1}
\sum_{j=0}^r \int_{\R \times [0,\delta_j]} u_j^*\omega_j + {\rm d}(H_j(u_j) {\rm d}t) <\infty,
\eeq
and prescribed limits
\beq \label{ulim2} \qquad\lim_{s\to\pm\infty} u_j(s,\cdot) = x^\pm_j \quad\forall j = 0,\ldots,r . 
\eeq

The Floer coboundary operator is the first in a sequence of operators
associated to pseudoholomorphic quilts with varying domain.  In
\cite{ww:quiltfloer} we showed that $\partial^2 = 0$, and hence the
quilted Floer cohomology
$$HF^\bullet(\ul{L}) = H^\bullet(CF(\ul{L},\partial))$$ 
is well defined.  Here we work on chain level, and in case $M_s=\pt$
interpret $\partial=:\mu^1$ as 
the 
first of the \ainfty composition maps on $\GFuk(M)$,
$$ \mu^1 : CF^\bullet(\ul{L}^+,\ul{L}^-) \to
CF^{\bullet+1}(\ul{L}^+,\ul{L}^-) .
$$
The objects in the extended Fukaya category are generalized Lagrangian
branes.  The morphism spaces in the extended Fukaya category are the
quilted Floer chain complex associated to the cyclic generalized
Lagrangian correspondence $\ul{L}$ of length $r = r^+ + r^-$ shifted
in degree
$$ \Hom(\ul{L}^+,\ul{L}^-) := CF(\ul{L}^+,\ul{L}^-)[d], \quad d = \hh
\Bigl( \sum_{k^+} \dim(N_{k^+}) + \sum_{k^-} \dim(N^-_{k^-}) \Bigr) $$
where 
$$ CF(\ul{L}^+,\ul{L}^-) := CF(\ul{L}^+\sharp_{\eps = 1} (\ul{L}^-)^T) =
CF(\ul{L}).
$$

\subsection{The associahedra}
\label{assocsec}

The higher composition maps in Fukaya categories are defined by
counting pseudoholomorphic polygons with Lagrangian boundary
condition. 

\label{domainofeach} The domain of each polygon corresponds to a point in a Stasheff
associahedra as follows.  Let $d \ge 2$ be an integer.  The $d$-th
associahedron $\K^d$ is a cell complex of dimension $d-2$ defined
recursively as the cone over a union of lower-dimensional
associahedra, whose vertices correspond to the possible ways of
parenthesizing $d$ variables $a_1,\ldots,a_d$, see Stasheff
\cite{st:ho}.  More precisely, any such expression corresponds to a
tree $\Gamma$ describing the parenthisization, which is {\em stable}
in the sense that the valence $|v|$ of any vertex $\Ver(\Gamma)$ is at
least $3$.

The recursive construction starts from the case that the space is a
point, and builds up from lower dimensional associahedra.  In the base
case, $\K^2$ is by definition a point.  Let $d \ge 3$ and suppose that
the associahedra $\K^n$ for $n < d$ have already been constructed.
For any tree $\Gamma$ with $n$ semi-infinite edges at least two
vertices and vertices $v \in \Ver(\Gamma)$ define
$$ \K^\Gamma = \prod_{v \in \Ver(\Gamma)} \K^{|v|} .$$
By assumption, the spaces $\K^n$ are equipped with injective maps
$\iota_\Gamma : \K^{\Gamma} \to \K^n$ for any stable tree with $n$
leaves.  For any morphism of stable trees $\Gamma' \to \Gamma$, we
have a natural injective map $\K^{\Gamma'} \to \K^\Gamma$ defined by
the product of the maps $\iota_{\pi^{-1}(v)}$ where $v$ ranges over
vertices of $\K^\Gamma$.  Define
\begin{equation} \label{boundary}
  \partial \K^d := \left( \bigcup_\Gamma \K^\Gamma \right) / 
  \sim \end{equation}
where the equivalence relation $\sim$ is the one induced by the
various maps $\iota_\Gamma^{\Gamma'}$. Let $\K^d$ be the cone on
$\partial \K^d$
$$ \K^d = \Cone(\partial \K^d) .$$
For example, the associahedron $\K^4$ is the pentagon shown in Figure
\ref{K4}.
\begin{figure}[h]
\begin{picture}(0,0)%
\includegraphics{k4sym.pstex}%
\end{picture}%
\setlength{\unitlength}{4144sp}%
\begingroup\makeatletter\ifx\SetFigFont\undefined%
\gdef\SetFigFont#1#2#3#4#5{%
  \reset@font\fontsize{#1}{#2pt}%
  \fontfamily{#3}\fontseries{#4}\fontshape{#5}%
  \selectfont}%
\fi\endgroup%
\begin{picture}(2501,2041)(775,-2963)
\put(1230,-1013){\makebox(0,0)[lb]{{{{$(a_1(a_2a_3))a_4$}%
}}}}
\put(2722,-996){\makebox(0,0)[lb]{{{{$a_1((a_2a_3)a_4)$}%
}}}}
\put(3261,-2156){\makebox(0,0)[lb]{{{{$a_1(a_2(a_3a_4))$}%
}}}}
\put(1955,-3027){\makebox(0,0)[lb]{{{{$(a_1a_2)(a_3a_4)$}%
}}}}
\put(390,-2303){\makebox(0,0)[lb]{{{{$((a_1a_2)a_3)a_4$}%
}}}}
\end{picture}%
\caption{$\K^4$}
\label{K4}
\end{figure}

The associahedra admit homeomorphism to convex polytopes, so that the
interiors of the cones are the faces of the polytope \cite{assoc}.
That is, for any $\Gamma$ let $\on{int}(\K^\Gamma) = \K^\Gamma
- \partial \K^\Gamma$.  Then 
$ \K^n = \sqcup_\Gamma \on{int}( \K^\Gamma ) $
is the decomposition into open faces.

The associahedra can be realized as the moduli space of
stable \label{stablepage} marked disks, which provides a connection
with Deligne-Mumford moduli spaces of stable spheres.

\begin{definition} 
\begin{enumerate} 
\item {\rm (Nodal disks)} A {\em nodal disk} $D$ is a contractible
  space obtained from a union of disks $D_i,i =1,\ldots,l$ (called the
  components of $D$) by identifying pairs of points $w_j^+,w_j^-, j
  =1,\ldots, k$ on the boundary (the {\em nodes} in the resulting
  space)
$$ D = \sqcup_{i=1}^l D_i / (w_j^+ \sim w_j^-, j = 1,\ldots,k ) $$
so that each node $w_j\in D$ belongs to exactly two disk components
$D_{i_-(j)}, D_{i_+(j)}$.
\item {\rm (Marked nodal disks)} A {\em set of markings} is a set $\{
  z_0,\ldots,z_d \}$ of the boundary $\partial D$ in counterclockwise
  order, distinct from the singularities.  A {\em marked} nodal disk
  is a nodal disk with markings.  A {\em morphism of marked nodal
    disks} from $(D,\ul{z})$ to $(D',\ul{z}')$ is a homeomorphism
  $\varphi:D \to D'$ restricting to a holomorphic isomorphism
  $\varphi|_{D_i}$ on each component $i=1,\ldots,l$ and mapping the
  marking $z_j$ to $z_j'$.
\item {\rm (Stable disks)} A marked nodal disk $(D,\ul{z})$ is {\em
  stable} if it has no automorphisms or equivalently if each disk
  component $D_i \subset D$ contains at least three nodes or markings.
\item {\rm (Combinatorial types)} The {\em combinatorial type} of a
  nodal disk with markings is the tree 
$$\Gamma = (\Ver(\Gamma),\Edge(\Gamma)), \quad \Edge(\Gamma) = \Edge_{<\infty}(\Gamma)
\sqcup \Edge_\infty(\Gamma) $$ 
obtained by replacing each disk with a vertex $v \in \Ver(\Gamma)$,
each node with a finite edge $e \in \Edge_{<\infty}(\Gamma)$, and each
marking with a semi-infinite edge $e \in \Edge_{\infty}(\Gamma)$.  The
semi-infinite edges $\Edge_\infty(\Gamma)$ are labelled by
$0,\ldots,d$ corresponding to which marking they represent, and the
tree has a planar structure given by the ordering of the leaves.  \label{planarst}
 \end{enumerate}
\end{definition}

A suitable notion of convergence, similar to that for stable marked
genus zero curves in \cite[Appendix]{ms:jh} defines a topology on
\label{RRpage} $\ol{\RR}^d$ that we will not detail \label{topologypage} here.  In
fact, $\ol{\RR}^d$ embeds as a subset of the real locus in the moduli
space of stable genus zero curves; see also the discussion of the
topology on $\ol{\RR}^{d,e}$ in Section \ref{biassoc}.  For each
combinatorial type $\Gamma$ let $\RR_\Gamma^d$ denote the space of
isomorphism classes of stable nodal $d+1$-marked disks of
combinatorial type $\Gamma$, and
$$ \ol{\RR}^d = \bigcup_{\Gamma} \RR^d_{\Gamma} .$$
Write $\Gamma \leq \Gamma'$ if and only if \label{porder} there is a
surjective morphism of trees (composition of morphisms collapsing an
edge) from $\Gamma$ to $\Gamma'$ in which case $\RR_\Gamma$ is
contained in the closure of $\RR_{\Gamma'}$.  In case $d = 3$, there
is a canonical homeomorphism $\ol{\RR}^{3} \to [0,1]$ given by the
cross-ratio.  The moduli space $\ol{\RR}^{4}$ is shown in Figure
\ref{M4}.
\begin{figure}[h]
\includegraphics[width=2.5in]{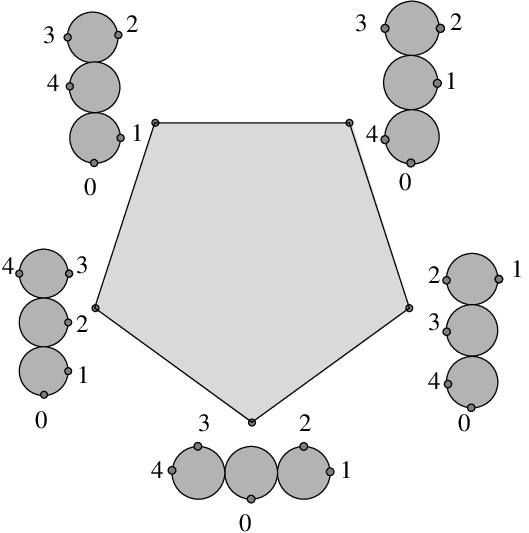}
\caption{$\ol{\RR}^{4}$}
\label{M4}
\end{figure}

The moduli space of stable marked disks comes with a {\em universal
  curve} whose fibers over any isomorphism class of curves is the
curve itself.  That is, the universal curve $\ol{\UU}^d$ is the space
of isomorphism classes of tuples $[D,z_1,\ldots, z_n, z]$ where
$(D,z_1,\ldots, z_n)$ is a stable $n$-marked disk and $z \in D$ is a
possibly nodal or marked point.  The topology on the universal curve
is defined in a similar way to the moduli space of stable disks, and
the map forgetting the additional point
$$ \ol{\UU}^d \to \ol{\RR}^d, \quad 
[D,z_1,\ldots,z_n,z] \mapsto [D,z_1,\ldots, z_n] $$
is continuous with respect to these topologies.

Each moduli space of disks has the structure of a manifold with
corners.  Coordinates near each stratum are given by the gluing
construction, as follows.

\begin{theorem}  
\label{collard} 
{\rm (Compatible tubular neighborhoods for associahedra)} For integers
$d \ge 2, m \ge 0$, each stratum $\RR_\Gamma^d$ with $m$ nodes has an
open neighborhood homeomorphic to $\RR_{\Gamma} \times [0,\infty)^m$.
  The normal coordinates can be chosen compatibly in the sense that if
  $\Gamma < \Gamma'$ and $\Gamma'$ has $m'$ nodes then the diagram
$$ 
\begin{diagram}
\node{\RR_\Gamma^d \times [0,\eps)^m} \arrow[2]{e}\arrow{se} \node{}
  \node{\ol{\RR}_{\Gamma'}^d \times [0,\eps)^{m'} } \arrow{sw}
    \\ \node{} \node{\ol{\RR}^d} \node{}
\end{diagram} 
$$
commutes.
\end{theorem} 

\begin{proof}[Sketch of proof]
Given a stable $n$-marked nodal disk $D$ with components $D_0,\ldots,
D_r$ with $m$ nodes $w_1^\pm,\ldots, w_m^\pm$, let $\delta_1,\ldots,
\delta_m \in \R_{\ge 0}$ be a set of {\em gluing parameters}.  Suppose
that each node is equipped with {\em local coordinates}: holomorphic
embeddings 
$$\phi_j^\pm: (B^+_\eps(0),0) \to (D,w_j^\pm)$$ 
where $B^+_\eps(0) \subset B_\eps(0)$ is the part of the $\eps$-ball
around $0$ in the complex plane with non-negative imaginary part.  The
glued disk $G_\delta(D)$ is obtained by removing small balls around
the $j$-th node and identifying points by the map
$z \mapsto \delta_j/z$ for every gluing parameter that is non-zero.
Suppose that one has for every point $r \in \RR_\Gamma^d$ a set of
such local coordinates varying smoothly in $r$.  Then one obtains from
the gluing construction a collar neighborhood as in the statement of
the theorem.

To check the compatibility relation, suppose that
$I \subset \{ 1,\ldots, m \}$ is a subset of the nodes and
$\delta_I \in \R_{\ge 0}^{|I|}$ the corresponding gluing parameters.
Starting with the disk \label{diskpage} above, glue together open balls around the
nodes $w_i^\pm, i \in I$ to obtain a disk $G_{\delta_I}(D)$, equipped
with local coordinates near the unresolved nodes given by the local
coordinates near the nodes of $D$, of combinatorial type $\Gamma'$
with $m' = m - |I|$ nodes.  Suppose that a family of local coordinates
near the nodes of $\RR_{\Gamma'}$ is given such that in a neighborhood
of $\RR_\Gamma$ the local coordinates are induced by those on
$\RR_\Gamma$ by gluing.  In this case the collar neighborhoods for
$\Gamma$ and $\Gamma'$ are compatible in the sense that the diagram in
the theorem commutes.

One may always choose the local coordinates to be given by gluing near
the boundary, since the space of germs of local coordinates is convex.
Indeed a map $\phi_j^\pm: (B^+_\eps(0),0) \to (D,w_j^\pm)$ defines a
local coordinate in some neighborhood of $0$ iff $D\phi_j^\pm(0) > 0$,
which is a convex condition.  So we may assume that on each stratum
$\RR_\Gamma^d$ there is a family of local coordinates such that near
any stratum $\RR_{\Gamma'}^d$ contained in the closure the local
coordinates are those induced by $\RR_{\Gamma'}^d$ from gluing.  This
completes the proof.
\end{proof} 

It follows that the stratified space $\ol{\RR}^d$ is {\em equipped
  with quilt data} in the sense of Definition \ref{quiltdata}: each
stratum comes with a collar neighborhood described by gluing
parameters compatible with the lower dimensional strata.

The moduli spaces of marked disks admit forgetful morphisms to moduli 
spaces with fewer numbers of markings.   For each $i = 1,\ldots, d$, 
we have a forgetful morphism 
\begin{equation} \label{fi} f_i: \ol{\RR}^d \to \ol{\RR}^{d-1} ,\quad 
  [D,z_1,\ldots,z_n] \mapsto [D,z_1,\dots, z_{i-1}, z_{i+1}, \ldots,
  z_n]^{\on{st}} \end{equation}
where the superscript $\on{st}$ indicates the disk obtained from $D$
by collapsing unstable components.  There is a well-known description
of the forgetful map (in the context of Deligne-Mumford spaces) as the
projection from the universal curve, developed in the case of closed
curves by Knudsen \cite[Sections 1,2]{kn:proj2}.  In the case of
disks, the universal curve $\ol{\UU}^d$ has a fiber-wise boundary
$\partial \UU^d$ which splits as a union of intervals
$$ \partial \UU^d = (\partial \UU^d)_0 \cup \ldots \cup (\partial
\UU^d)_d $$
where $(\partial \UU^d)_i$ is the part of the boundary between the
$i-1$-st and $i$-th marking, where $i$ is taken mod $d+1$.  See Figure
\ref{bpart}.  The map $f_i$ naturally lifts to a continuous map
\begin{equation} \label{univcurve} \ti{f}_i : \ol{\RR}^d \to (\partial
  \UU^{d-1})_i, \quad [D,z_1,\dots, z_n] \mapsto
  [D^{\on{st}},z_1,\dots, z_{i-1},z_{i+1}, \ldots, z_n, z_i^{\on{st}}
  ] \end{equation}
where $z_i^{\on{st}} \in D^{\on{st}}$ is the image of $z_i$ under the
stabilization map $D \to D^{\on{st}}$.  A continuous inverse to
$\ti{f}_i$ is defined by inserting $z_i^{\on{st}}$ in between
$z_{i-1}$ and $z_{i+1}$, and creating an additional component if
$z_i^{\on{st}}$ is a nodal point of $D^{\on{st}}$, showing that
$\ti{f}_i$ is a homeomorphism.

\begin{figure}[ht]
\begin{picture}(0,0)%
\includegraphics{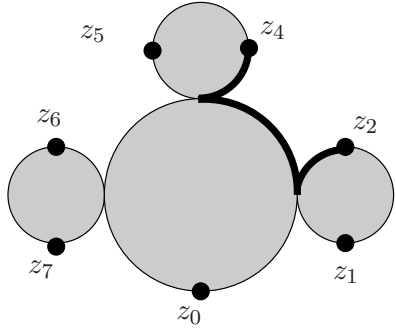}%
\end{picture}%
\setlength{\unitlength}{3947sp}%
\begingroup\makeatletter\ifx\SetFigFont\undefined%
\gdef\SetFigFont#1#2#3#4#5{%
  \reset@font\fontsize{#1}{#2pt}%
  \fontfamily{#3}\fontseries{#4}\fontshape{#5}%
  \selectfont}%
\fi\endgroup%
\begin{picture}(2438,2070)(2993,-2423)
\put(4059,-2381){\makebox(0,0)[lb]{{{{$z_0$}%
}}}}
\put(5043,-2124){\makebox(0,0)[lb]{{{{$z_1$}%
}}}}
\put(3453,-611){\makebox(0,0)[lb]{{{{$z_5$}%
}}}}
\put(3181,-1140){\makebox(0,0)[lb]{{{{$z_6$}%
}}}}
\put(3128,-2093){\makebox(0,0)[lb]{{{{$z_7$}%
}}}}
\put(5164,-1182){\makebox(0,0)[lb]{{{{$z_2$}%
}}}}
\put(4582,-589){\makebox(0,0)[lb]{{{{$z_4$}%
}}}}
\end{picture}%
\caption{The part of the boundary between two markings} 
\label{bpart}
\end{figure}

The forgetful maps induce the structure of fiber bundles on the moduli
spaces with contractible fibers.  Indeed the gluing construction
identifies all nearby fibers $G_\delta D$ with the interval obtained
by removing small disks around the nodes and identifying the
endpoints:
\begin{equation} \label{removeballs} (\partial G_\delta D )_i = \left(
    (\partial D)_i - \cup_{k=1}^m B_{\delta_k^{1/2}} (w_k)
  \right)/\sim .\end{equation}
On the other hand the intervals on the right-hand side of 
\eqref{removeballs} are homeomorphic to $(\partial D)_i$ itself, 
by extending a homeomorphism of the complement 
$$ (\partial D)_i - \cup_{k=1}^m B_{\delta_k^{1/2}} (w_k)
\to (\partial D)_i - \cup_{k=1}^m \{ w_k\} $$
to the boundary.  It follows that $f_i$ induces on $\ol{\RR}^d$ the
structure of a fiber bundle over $\ol{\RR}^{d-1}$ with interval
fibers.  The discussion above shows by induction that the moduli
space $\ol{\RR}^d$ is a topological disk.  (And it follows from the
isomorphism with the associahedron that they are isomorphic, as
decomposed spaces, to convex polytopes.)

\subsection{Higher compositions}

In this section we construct the higher composition maps on the Fukaya
category of generalized Lagrangian branes.  These are defined by
family quilt invariants for families of surfaces with strip-like ends
over the associahedra constructed in the following proposition:

\begin{proposition}  \label{Ad}
{\rm (Existence of families of strip-like ends over the associahedra)}
For each $d \ge 2$ there exists a collection of families of quilted
surfaces $\ol{\SS}^d$ with strip-like ends over $\ol{\RR}^d$ with the
property that the restriction $\SS^d_\Gamma$ of the family to a
stratum $\RR_\Gamma^d$ that is isomorphic to a product of $\RR^{i_j},
j = 1,\ldots, k$ is a product of the corresponding families
$\SS^{i_j}$, and collar neighborhoods of $\SS^d_\Gamma$ are given by
gluing along strip-like ends.
\end{proposition} 

\begin{proof}   The claim follows 
by induction using Theorem \ref{A} applied to the stratified space
$\ol{\RR}^d$ constructed in Theorem \ref{collard}, starting from the
case of three-marked disk where we choose a genus zero surface with
strip-like ends. \end{proof}

The compactness and regularity properties of the family moduli spaces
in the previous section combine to the following statement, in the
case of the constructed families over the associahedra:

\begin{proposition}   \label{holassoc} {\rm (Existence of compact families of holomorphic quilts 
over the associahedra)} Let $M$ be a symplectic background, and for $d
  \ge 2$ let $\ul{L}^0,\ldots,\ul{L}^d$ be admissible generalized
  Lagrangian branes in $M$.  For generic choices of inductively-chosen
  perturbation data $\ul{J},\ul{K}$ and $d \ge 2$
\begin{enumerate} 
\item the moduli space of pseudoholomorphic quilts $\M^d_0$ of dimension
  zero in $M$ with boundary in $\ul{L}^j, j = 0,\ldots, d$ is compact,
  and 
\item the one-dimensional component $\ol{\M}^d_1$ has a
  compactification as a one-manifold with boundary the union
$$ \partial \ol{\M}^d_1 = \bigcup_{\Gamma} {\M}^{d}_{\Gamma,1}$$
of strata $\M^d_{\Gamma,1}$ of $\ol{\M}^d_{1}$ corresponding to trees
with two vertices (where either (1) $\Gamma$ is stable with two
vertices, or (2) $\Gamma$ is unstable and corresponds to bubbling of a
Floer trajectory).
\end{enumerate} 
A similar statement holds for $d = 1$, using moduli spaces of {\em
  unparametrized} Floer trajectories.
\end{proposition} 

\begin{proof}  
  For $d \ge 2$, regular perturbation data exist by the recursive
  application of Theorem \ref{B} to the family of quilted surfaces
  $\ol{\SS}^d \to \ol{\RR}^d$ constructed in Proposition \ref{Ad}; the
  perturbation over the boundary of $\ol{\SS}^d$ is that of product
  form for the lower-dimensional associahedra.  The description of the
  boundary follows from Theorem \ref{masterthm}.
\end{proof}  

 Theorem \ref{C} gives chain-level family quilt invariants associated
 to the family over the associahedron defined in Theorem
 \ref{holassoc}.  The \ainfty composition maps are related to these by
 additional signs:

\begin{definition}  {\rm (Higher composition maps for the extended Fukaya category)}  
For $d \ge 2$ let $\ol{\SS}^d \to \ol{\RR}^d$ be the family of
surfaces with strip-like ends over the associahedron $\ol{\RR}^d$
constructed in Theorem \ref{Ad} and $\Phi_{\ol{\SS}^d}$ the associated
family quilt invariants.  Given objects $\ul{L}^0,\ldots,\ul{L}^d$
define
$$ \mu^d:  \Hom(\ul{L}^0,\ul{L}^1) \times \ldots \times \Hom(\ul{L}^{d-1},\ul{L}^d) 
\to \Hom(\ul{L}^0,\ul{L}^d) $$
by 
\begin{equation} \label{highercomp} \mu^d(\bra{\ul{x}_1},\ldots,\bra{\ul{x}_d}) = 
(-1)^\heartsuit \Phi_{\ol{\SS}^d} ( \bra{\ul{x}_1},\ldots,\bra{\ul{x}_d}) 
\end{equation}
where 
\begin{equation} \label{heartsuit}
 \heartsuit = {\sum_{i=1}^d i|\ul{x}_i|} .\end{equation} 
\end{definition} 

\begin{theorem} \label{ainftycat}  Let $M$ be a symplectic background, and $\mu^d$ 
for $d \ge 1$ the maps defined in \eqref{highercomp} for some choice
of family of surfaces of strip-like ends and perturbation data.  Then
the maps $\mu^d, d \ge 1$ define an \ainfty category $\GFuk(M)$.
\end{theorem} 

\begin{proof}  
  Without signs the theorem holds for $d \ge 2$ by the description of the ends of the one-dimensional
moduli spaces in Theorem \ref{holassoc}.  To check the signs in the
\ainfty associativity relation, suppose for simplicity that all
generalized Lagrangian branes are length one.  Let $x_j \in
\cI(\ul{L}^j,\ul{L}^{j+1})$ for $j = 0,\ldots, d$ indexed mod $d+1$
and $\ol{\M}^d (x_0,\ldots,x_d)$ the moduli space of quilts
with limits $x_0,\ldots, x_d$ along the strip-like ends.
Consider the gluing map
constructed in Ma'u \cite[Theorem 1]{mau:gluing}
\begin{equation} \label{gmap} \M^m({y},{x}_{n+1},\ldots, {x}_{n+m})_0 \times \M^{d-m
  +1}({x}_0,{x}_1,\ldots,{y},\ldots,{x}_d)_0 \to
  \M^d({x}_0,\ldots,{x}_d)_1 .\end{equation}
For any intersection point $x_j$ let $\DD^+_{x_j}$ denote the
determinant line associated to $x_j$ in \cite[Equation (40)]{orient}
of the Fredholm operator on the once-punctured disk associated to a
choice of path from $T_{x_j} \ul{L}^{j-1}$ and $T_{x_j} \ul{L}^{j}$ to
a collection of Lagrangians of split form in $T_{x_j} \ul{M}_j$ (where
here $\ul{M}_j$ denotes the collection of patches meeting the $j$-th
end).  The determinant lines $\DD^-_{x_j}$ is defined similarly, but
with an added determinant $\det(T_{x_j} \ul{L}^j)$ of the Lagrangians
meeting the end.  By assumption, the orientations on $\DD^\pm_{x_j}$
are defined so that the tensor product
$$ \DD^-_{x_j} \otimes \DD^+_{x_j} \cong \R $$
is canonically trivial.  By deforming the parametrized linear operator
to the linearized operator plus a trivial operator, and bubbling off
marked disk on each strip like end we may identify via \cite[Equation
(40)]{orient}
\begin{equation} \label{convent} \det( T {\M}^d (x_0,\ldots,x_d)) \to
 \det(T\RR^d) \DD^+_{x_0}   
  \DD^-_{x_1} 
 \ldots 
\DD^-_{x_d}
   .\end{equation}
After this identification the gluing map \eqref{gmap} takes the form
(omitting tensor products from the notation to save space)
\begin{multline} 
\det(\R)
 \det(T\RR^m) 
\DD^+_y 
 \DD^-_{x_{n+1}}  
\ldots \DD^-_{x_{n+m}} 
\det(T\RR^{d-m+1})
\DD^+_{x_0}   
  \DD^-_{x_1} \ldots \DD^-_y \ldots \DD^-_{x_d} 
  .\end{multline}
To determine the sign of this map, first note that the gluing map
$$(0,\eps) \times \RR^m \times \RR^{d-m+1} \to \RR^d$$ 
on the associahedra is given in coordinates (using the automorphisms
to fix the location of the first and second point in $\RR^m$ to equal
$0$ resp. $1$ and $\RR^{d- m + 1}$) by
\begin{multline} \label{signs} (\delta,(z_3,\ldots,z_m)
  ,(w_3,\ldots,w_{d-m+1}) ) \\ \to (w_3,\ldots, w_{n+1}, w_{n+1} +
  \delta, w_{n+1} + \delta z_3, \ldots, w_{n+1} + \delta z_m,
  w_{n+2},\ldots, w_{d-m}) .\end{multline}
This map acts on orientations by a sign of $-1$ to the power
\begin{equation} \label{signa} (m-1)(n-1). \end{equation}
These signs combine with the contributions
\begin{equation} \label{signb} 
\sum_{k=1}^n k |x_k| +
(n+1)|y| + 
 \sum_{k=n+m+1}^d 
  (k-m+1) |x_k| 
+ 
\sum_{k=n+1}^m (k-n) |x_k|
 \end{equation} 
in the definition of the structure maps, and a contribution
\begin{equation} \label{signc}
(d-m+1)m + m \left(  |y| + \sum_{i \ge n } |x_i| \right) 
 \end{equation} 
 from permuting the determinant lines
 $\DD_{x_j}^-, j =n+1,\dots,n+m, \DD_y^+$ with $\det(T \RR^{d-m+1})$
 and permuting these determinant lines with the
 $\DD_{x_i}^-, i \leq n, \DD_y^-$.  On the other hand, the sign in the
 \ainfty axiom contributes
\begin{equation} \label{signd} \sum_{k=1}^n (|x_k| + 1) \end{equation}
Combining the signs \eqref{signa}, \eqref{signb}, \eqref{signc},
\eqref{signd} one obtains mod $2$ \label{mod2}
\begin{multline} 
\left(   mn + n + m \right) +
\left( \sum_{k=1}^n k |x_k| + (n+1) |y| + \sum_{k=n+m+1}^d (k-m+1)
  |x_k| +  \sum_{k=n+1}^{n+m} 
(k-n) |x_k|  \right) \\ 
 + 
(d-m+1)m + m \left( |y| +  \sum_{i \leq n} |x_i| \right)
+ \sum_{k=1}^n (|x_k| + 1) \\
\equiv 
( mn + m + n) + 
\sum_{k=1}^d k |x_k| + (n+1) |y| + 
 \sum_{k=n+m+2}^d (m-1) |x_k| + \sum_{k=n+1}^{n+m}  n |x_k| 
\\  + 
(d-m+1)m + m \left( d + \sum_{i \ge n + m + 1} |x_i| \right)
+ \sum_{k=1}^n |x_k| + n  \end{multline} 
\begin{multline}
  \equiv mn + m + \sum_{k=1}^d k |x_k| + |y| + \sum_{k=n+m+2}^d |x_k|
  + nm +
  (d-m+1)m + md + \sum_{k=1}^n |x_k|  + n  \\
  \equiv m + \sum_{k=1}^d k |x_k| + \sum_{k=n+1}^{n+m} |x_k| + m +
  \sum_{k=n+m+2}^d |x_k| + \sum_{k=1}^n |x_k|
\end{multline}
which is congruent mod $2$ to 
\begin{equation} \label{assocsign} \sum_{k=1}^d (k+1)
  |x_k|. \end{equation}
Since \eqref{assocsign} is independent of $n,m$, the
\ainfty-associativity relation \eqref{ainftyassoc} follows.
\end{proof} 

\begin{remark} {\rm (Units)} 
  Cohomological units are constructed in \cite{we:co}.  The unit for
  $\ul{L}$ is defined by counting perturbed pseudoholomorphic
  once-punctured disks with boundary in $\ul{L}$.  That is, the unit
  is the relative invariant associated to the quilt obtained from the
  once-punctured disk by attaching sequence of strips so that the
  boundaries lie in the Lagrangian correspondences in $\ul{L}$.
\end{remark} 

\subsection{The Maslov index two case}
\label{maslovtwo} 
The definitions of the previous sections extend to the case of Maslov
index two Lagrangians, once one fixes a {\em total disk invariant} as
in Oh \cite[Addendum]{oh:fl1}.  Given a Lagrangian $L \subset M$ and a
point $\ell \in L$, we denote by $\M_1^2(L,J,\ell)$ the moduli space
of Maslov index two holomorphic maps $u: (D,\partial D) \to (X,L)$
mapping $1 \in D$ to $\ell \in L$.

\begin{proposition} \label{prop:disk number} {\rm (Disk invariant of a
    Lagrangian)} For any $\ell \in L$ there exists a comeager subset
  $\J^{\rm reg}(\ell) \subset \J(M,\omega)$ such that
  $\M_1^2(L,J,\ell)$ is cut out transversally.  Any relative spin
  structure on $L$ induces an orientation on $\M_1^2(L,J,\ell )$.
  Letting $\eps: \M_1^2(L,J, \ell ) \to \{ \pm 1 \} $ denote the map
  comparing the given orientation to the canonical orientation of a
  point, the disk number of $L$,
$$ w(L) := \sum_{[u] \in \M_1^2(L,J,\ell)} \eps([u]) ,$$
is independent of $J \in J^{\reg}(\ell)$ and $\ell \in L$.
\end{proposition}  

See Oh \cite[Addendum]{oh:fl1} for the proof. The quilted Floer
operator in the minimal Maslov index two case satisfies the following
relation involving the disk invariant above.  Let
$\ul{L}=(L_{j(j+1)})_{j=0,\ldots,r}$ be a cyclic generalized
Lagrangian brane between symplectic backgrounds
$M_j, j = 0,\ldots, r$.  Let
$$ \J_t(\ul{L}) := \prod_{j=0}^r
C^\infty([0,\delta_j],\J(M_j,\omega_j)) $$ 
denote the space of time-dependent almost complex structures on strips
with width $\delta_j$.

\begin{theorem} \label{thm:d2} 
{\rm (Quilted Floer cohomology)} For any $\ul{H}\in\Ham(\ul{L})$,
widths $\ul{\delta}=(\delta_j>0)_{j=0,\ldots,r}$, and $\ul{J}$ in a
comeager subset $\J^{\reg}_t(\ul{L},\ul{H}) \subset \J_t(\ul{L})$, the
Floer differential $\partial: CF(\ul{L}) \to CF(\ul{L})$ satisfies
$$\partial^2= w(\ul{L}) \Id , \qquad w(\ul{L})= \sum_{j=0}^r
w(L_{j(j+1)}) .$$
The pair $( CF(\ul{L}),\partial)$ is independent of the choice of
$\ul{H}$ and $\ul{J}$, up to cochain homotopy.
\end{theorem}

\begin{proof}   
  We sketch the proof, following Oh \cite[Addendum]{oh:fl1} in the
  case of $\Z_2$ coefficients.  For any $\ul{x}_\pm \in \cI(\ul{L})$,
  the zero dimensional component $\M(\ul{x}_-,\ul{x}_+)_0$ of Floer
  trajectories is a finite set.  As in \cite[Proposition 4.3]{oh:fl1}
  the one-dimensional component $\M(\ul{x}_-,\ul{x}_+)_1$ is smooth,
  but the ``compactness modulo breaking'' does not hold in general:
  Apart from the breaking of trajectories, a sequence of Floer
  trajectories of Maslov index $2$ could in the Gromov
  compactification converge to a constant trajectory and either a
  sphere bubble of Chern number one or a disk bubble of Maslov number
  two. All other bubbling effects are excluded by monotonicity.  Thus
  failure of ``compactness modulo breaking'' can occur only when
  $\ul{x}_- = \ul{x}_+$.

The proof follows from the claim that each one-dimensional moduli
space $\M(\ul{x},\ul{x})_1$ of self-connecting trajectories has a
compactification as a one-dimensional manifold with boundary
$$ 
\partial \ol{\M(\ul{x},\ul{x})_1} \;\cong\; 
\bigcup_{\ul{y}\in\cI(\ul{L})} \bigl( \M(\ul{x},\ul{y})_0 \times \M(\ul{y},\ul{x})_0 \bigr)
\; \cup \; \bigcup_{j=0,\ldots,r} \M_1^2(L_{j(j+1)},J_{j(j+1)},\{ (x_j,x_{j+1}) \})^-
$$
and that furthermore the orientations on these moduli spaces induced
by the relative spin structures are compatible with the inclusion of
the boundary.  Here $\M_1^2(\ldots)^-$ denotes the moduli space
$\M_1^2(\ldots)$ with orientation reversed.  The subset
$\J^{\reg}_t(\ul{L};\ul{H})$ consists of collections of time-dependent
almost complex structures $J_j:[0,\delta_j]\to\J(M_j,\omega_j)$ for
which all $\M(\ul{x}_-,\ul{x}_+)$ are smooth and the universal moduli
spaces of spheres $\M_1^1(M_j,\{J_j(t)\}_{t\in[0,\delta_j]},\{x_j\})$
are empty for all $\ul{x}=(x_j)_{j=0,\ldots,r}\in \cI(\ul{L})$.  This
choice excludes the Gromov convergence to a constant trajectory and a
sphere bubble.  We now restrict to those
$\ul{J}\in\J^{\reg}_t(\ul{L};\ul{H})$ such that
$$J_{j(j+1)}:=(- J_j(\delta_j) ) \oplus J_{j+1}(0) \in \J^{\rm
  reg}(L_{j(j+1)},\{(x_j,x_{j+1})\})$$ 
for all $\ul{x}\in \cI(\ul{L})$ and $j=0,\ldots,r$.  This still
defines a comeager subset in $\J_t(\ul{L})$.

To finish the proof of the claim we use a gluing theorem of
non-transverse type for pseudoholomorphic maps with Lagrangian
boundary conditions.  The required gluing theorem can be adapted from
\cite[Chapter 10]{ms:jh} as follows: Replace $\ul{L}$ with its
translates under the Hamiltonian flows of $\ul{H}$, so that the Floer
trajectories are unperturbed $J_j$-holomorphic strips (where the $J_j$
have suffered some Hamiltonian transformation, too).  Pick
$[v_{j(j+1)}]\in \M_1^2(L_{j(j+1)},J_{j(j+1)},\{ (x_j,x_{j+1}) \})$
and a representative $v_{j(j+1)}$.  The gluing construction gives a
map
\begin{equation} \label{pointglue}
(T,\infty) \longrightarrow \M(\ul{x},\ul{x})_1 \end{equation}
to the moduli space of parametrized Floer trajectories of index $2$,
where $T \gg 0$ is a real parameter.  This construction first
identifies $v_{j(j+1)}$ with a map from the half space $\bH\cong
D\setminus\{1\}$ to $M_j^-\times M_{j+1}$. For the pregluing choose a
{\em gluing parameter} $\tau \in (T,\infty)$.  Outside of a half disk
of radius $\frac{1}{2} \tau^{1/2}$ around $0$, interpolate the map to
the constant solution $(x_j,x_{j+1})$ outside of the half disk of
radius $\tau^{1/2}$ using a slowly varying cutoff function in
submanifold coordinates of $L_{j(j+1)}\subset M_j^-\times M_{j+1}$
near $(x_j,x_{j+1})$.  Then rescale this map by $\tau$ to a half-disk
of radius $\tau^{-1/2}$ centered around $0$ in the strip
$\R\times[0,\tau^{-1/2})$, again extended constantly.  The components
  give an approximately $J_{j+1}$-holomorphic map
  $u_{j+1}:\R\times[0,\tau^{-1/2}) \to M_{j+1}$ and, after reflection,
    an approximately $J_{j}$-holomorphic map $u_j:
    \R\times(\delta_j-\tau^{-1/2},\delta_j] \to M_j$.  For
  $T\geq\max\{\delta_j^{-2},\delta_{j+1}^{-2} \}$ these strips can be
  extended to width $\delta_j$ resp.\ $\delta_{j+1}$.  Together with
  the constant solutions $u_\ell \equiv x_\ell$ for
  $\ell\not\in\{j,j+1\}$ we obtain a tuple
$$\ul{u}=(u_{\ell}: \R \times [0,\delta_\ell] \to
  M_\ell)_{\ell=0,\ldots,r}$$
that is an approximate Floer trajectory.  An application of the
implicit function theorem gives an exact solution for $T$ sufficiently
large.  The uniqueness part of the implicit function theorem gives
that $[v_{j(j+1)}]$ is an isolated limit point of
$\M(\ul{x},\ul{x})_1$, so that $\ol{\M}(\ul{x},\ul{x})_1$ is a
one-dimensional manifold with boundary in a neighborhood of the nodal
trajectory with disk bubble $[v_{j(j+1)}]$.

It remains to examine the effect of the gluing on orientations for
which we need to recall the construction of orientations in
\cite{orient}.  Choose a parametrization
$$ [T,\infty] \to \ol{\M}(\ul{x},\ul{x})_1, \quad \infty \mapsto
     [v_{j(j+1)}] $$
homotopic to the gluing map.  Now the action on orientations is given
by the action on local homology groups, and homotopic maps induce the
same action.  So by replacing the gluing map with this parametrization
we may assume that the gluing map is an embedding.  The moduli space
$\M_1^2(L_{j(j+1)},J_{j(j+1)},\{(x_j,x_{j+1})\})$ has orientation at
$[u]$ induced by determinant line $\det(D_u)$ and the determinant of
the infinitesimal automorphism group $\aut(\R \times [0,1]) \cong \R$.
On the other hand, the orientation on $\M_1^2(L_{j(j+1)},J_{j(j+1)},\{
(x_j,x_{j+1}) \})$ is induced by the orientation of the determinant
line $\det(D_{v_{j(j+1)}})$ and the determinant line of the
automorphism group $\Aut(D,\partial D,1) \cong (0,\infty) \times \R$.
The gluing map induces an orientation-preserving isomorphism of
determinant lines $\det(D_u) \to \det(D_{v_{j(j+1)}})$ by
\cite{orient}.  Under gluing the second factor in $\Aut(D,\partial
D,1)$ agrees with the translation action on $\R \times [0,1]$.  On the
other hand, the first factor agrees approximately with the gluing
parameter, in the sense that gluing in a dilation of $v_{j(j+1)}$ by a
small constant $\rho > 0$ using gluing parameter $\tau$ is
approximately the same as gluing in $v_{j(j+1)}$ with gluing parameter
$\tau \rho$.  Thus $v_{j(j+1)}$ represents a boundary point of $
\ol{\M}(\ul{x},\ul{x})_1$ with the opposite of the induced orientation
from the interior.  Summing over the boundary of the one-dimensional
manifold $\partial \ol{\M}(\ul{x},\ul{x})_1$ proves that $\partial^2 -
\sum_{j=0}^r w(L_{j(j+1)})\Id =0$.  The proof that two choices lead to
cochain-homotopic pairs $(CF(\ul{L}),\partial)$ is given in
\cite{ww:quiltfloer}.
\end{proof}

\begin{remark} {\rm (Floer theory of a pair of Lagrangians)} 
In the special case $\ul{L}=(L_0,L_1)$ of a cyclic correspondence
consisting of two Lagrangian submanifolds $L_0,L_1\subset M$ we have
$w(\ul{L})=w(L_0)-w(L_1)$.  Indeed the $-J_1$-holomorphic discs with
boundary on $L_1\subset M^-\times \{ \on{pt} \}$ are identified with
$J_1$-holomorphic discs with boundary on $L_1\subset M$ via a
reflection of the domain.  This reflection is orientation reversing
for the moduli spaces. In particular, the differential for a monotone
pair $\ul{L}=(L,\psi(L))$ with any symplectomorphism $\psi\in\Symp(M)$
always squares to zero, since $w(L) = w(\psi(L))$.
\end{remark}

Following Sheridan \cite{sh:hmsfano}, for each integer $w \in \Z$
define the generalized Fukaya category $\GFuk(M,w)$ whose objects are
generalized Lagrangians $\ul{L}$ whose total disk invariant is
$w(\ul{L}) = w$, and whose morphisms are the quilted Floer cochain
groups defined above.

\begin{theorem}  Let $M$ be a symplectic background, $w \in \Z$ an integer and $\mu^d$ 
for $d \ge 1$ the maps defined in \eqref{highercomp} for some choice
of family of surfaces of strip-like ends and perturbation data.  Then
the maps $\mu^d, d \ge 1$ define an \ainfty structure on $\GFuk(M,w)$.
\end{theorem} 

\begin{proof} The possibility of disk bubbling occurs only the definition of
$\mu_1^2$, since it is only in this case that there exist holomorphic
  quilts in a moduli space that is not of expected dimension (the
  constant trajectories).  The proof is therefore the same as in the
  case of Maslov number at least three, with the added requirement
  that the perturbation data on the ends makes the disks in the
  definition of the disk invariant regular.
 \end{proof}

\section{Functors for Lagrangian correspondences} 
\label{multiplihedron}
\label{functors}

In this section we construct \ainfty functors associated to any
admissible Lagrangian correspondence equipped with a brane structure.

\subsection{Moduli of stable quilted disks} 

The multiplihedron is a polytope introduced by Stasheff in
\cite{st:ho} which plays the same role in the theory of \ainfty
morphisms as the associahedron does for \ainfty algebras.  

\begin{definition} \label{kd0def} For $d\geq 1$, the complex $\K^{d,0}$ is a compact
  cell complex of dimension $d - 1$ whose vertices correspond to the
  ways of maximally bracketing $d$ formal variables $a_1,\ldots,a_d$
  and applying a formal operation $h$.  The complex $\K^{d,0}$ is
  defined recursively as the cone over the union of products of
  lower-dimensional spaces $\K^e$ and $ \K^{f,0}$.
  \cite{st:ho}.  \end{definition}

\begin{example}
\begin{enumerate} 
\item {\rm (Second multiplihedron)} The second multiplihedron
  $\K^{2,0}$ is homeomorphic to a closed interval with end points
  corresponding to the expressions $h(a_1 a_2)$ and $h(a_1) h(a_2)$.
\item {\rm (Third multiplihedron)} The multiplihedron $\K^{3,0}$ is
  homeomorphic to a hexagon shown in Figure \ref{K31}, with vertices
  corresponding to the expressions $h((a_1a_2)a_3)$, $h(a_1(a_2a_3))$,
  $h(a_1)h(a_2a_3)$, $h(a_1a_2)h(a_3)$, $(h(a_1)h(a_2))h(a_3)$,
  $h(a_1)(h(a_2)h(a_3))$.
\end{enumerate} 
\end{example}

\begin{figure}[h]
\begin{picture}(0,0)%
\includegraphics{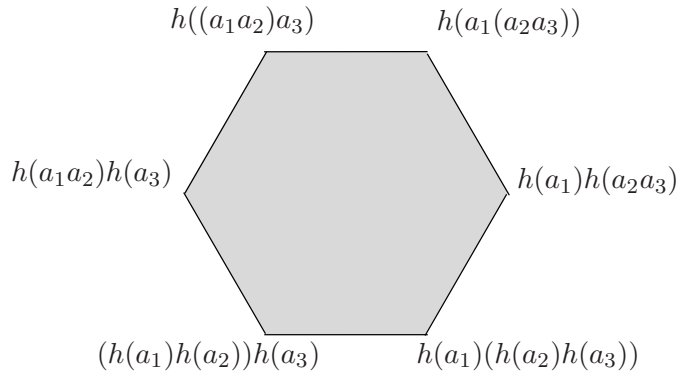}%
\end{picture}%
\setlength{\unitlength}{4144sp}%
\begingroup\makeatletter\ifx\SetFigFont\undefined%
\gdef\SetFigFont#1#2#3#4#5{%
  \reset@font\fontsize{#1}{#2pt}%
  \fontfamily{#3}\fontseries{#4}\fontshape{#5}%
  \selectfont}%
\fi\endgroup%
\begin{picture}(2557,2001)(872,-2662)
\put(1339,-727){\makebox(0,0)[lb]{{{{$h((a_1a_2)a_3)$}%
}}}}
\put(2933,-736){\makebox(0,0)[lb]{{{{$h(a_1(a_2a_3))$}%
}}}}
\put(3414,-1683){\makebox(0,0)[lb]{{{{$h(a_1)h(a_2a_3)$}%
}}}}
\put(387,-1643){\makebox(0,0)[lb]{{{{$h(a_1a_2)h(a_3)$}%
}}}}
\put(900,-2729){\makebox(0,0)[lb]{{{{$(h(a_1)h(a_2))h(a_3)$}%
}}}}
\put(2811,-2729){\makebox(0,0)[lb]{{{{$h(a_1)(h(a_2)h(a_3))$}%
}}}}
\end{picture}%
\caption{ Vertices of $\K^{3,0}$}
\label{K31}
\end{figure}
We review from \cite{mau:mult} the realization of the multiplihedron
as the moduli space of stable nodal {\em quilted disks} with marked
points. \label{qdhere}

\begin{definition} \label{qddef} {\rm (Quilted disks)}  
Let $d \ge 2$. A {\em quilted  disk with $d+1$ markings} is a tuple 
$(D,C,z_0,z_1 \ldots, z_d)$ where
\begin{enumerate}
\item $D$ is a holomorphic disk, which can be taken to be the closed
  unit disk in $\C$.
\item $C$ is a circle in $D$ with unique intersection point $C\cap
  \partial D = \{ z_0 \}$ equal to the $0$-th marking.  That is, if
  $D$ is identified with a disk in $\C$ then $C$ is a circle in $D$
  tangent to $\partial D$ at $z_0$.
\item $(z_0, z_1,\ldots, z_d)$ is a tuple of distinct points in
  $\partial D$ whose cyclic order is compatible with the orientation
  of $\partial D$.
\end{enumerate} 
An {\em isomorphism} of marked quilted disks from $(D,C,z_0,z_1
\ldots, z_d)$ to $(D',C',z'_0,z'_1 \ldots, z'_d)$ is a holomorphic
isomorphism preserving the quiltings and markings:
$$ \psi: D \to D', \quad \psi(C) = C', \quad \psi(z_j) = z_j', \quad j
= 0,\ldots, d .$$
\end{definition}

The space of isomorphism classes of marked quilted disks admits a
natural compactification by stable quilted disks, defined as follows.

\begin{definition}  \label{coltreedef}
\begin{enumerate} 
\item {\rm (Colored trees)} A {\em colored tree} is a pair
  $(\Gamma,\Edge_\infty(\Gamma))$ consisting of a tree $\Gamma =
  (\Ver(\Gamma),\Edge(\Gamma))$ with semiinfinite edges
  $\Edge_\infty(\Gamma) \subset \Edge(\Gamma)$ labelled
  $z_0,\dots,z_d$ equipped with a distinguished set of {\em colored
  vertices} $\Ver^{(1)}(\Gamma)$ with the following property: 
\begin{itemize}
\item[] For each
  $j = 1,\ldots, d,$ the unique shortest path in $\Gamma$ from the
  semi-infinite {\em root edge} marked $z_0$ to the semi-infinite edge
  $z_j$ crosses exactly one colored vertex.  
\end{itemize} 
A colored tree is {\em stable} if each colored resp. uncolored vertex
has valence at least two resp. three.
\item {\rm (Nodal quilted disks)} A {\em nodal $(d+1)$-marked quilted
  disk} $\ul{S}$ of combinatorial type equal to a colored tree
  $\Gamma$ is a collection $D_i, i = 1,\ldots, a$ and $(D_i', C_i'), i
  = 1,\ldots, b$ of quilted and unquilted marked disks, identified at
  pairs of points on the boundary $w_k^- \in \partial D_{i^-(k)},w_k^+
  \in \partial D_{i^+(k)}$ distinct from each other and the interior
  circles, together with a collection $z_0,\ldots,z_d$ of distinct
  smooth points on the boundary of $D = \sqcup_{i=1}^a D_i \sqcup
  \sqcup_{i=1}^b D_i'$ such that the graph obtained by replacing disks
  resp. quilted disks with unquilted resp. quilted vertices is the
  given colored tree $\Gamma$.
An {\em isomorphism} of nodal marked quilted disks $D,D'$ with the
same combinatorial type $\Gamma$ is a collection \label{colhere} of complex isomorphisms
$\phi_v: D_v \to D^\prime_v, v \in \Ver(\Gamma)$ of the corresponding
disk components, identifying nodal points, marked points, and/or inner
circles.

\item {\rm (Stable quilted disks)} A nodal quilted disk is {\em
    stable} if and only if it has no automorphisms, or equivalently
  the corresponding colored tree is stable.  More concretely, a nodal
  quilted disk is stable if and only if each quilted disk component
  $(D_i,C_i)$ contains at least $2$ singular or marked points
$$ D_i \text{ \ quilted \ } \implies \# \{ z_k \in D_i \} + \# \{ w_j
  \in D_i \} \ge 2 $$
and each non-quilted disk component $D_i$ contains at least $3$
singular or marked points
$$ D_i \text{ \ unquilted \ } \implies \# \{ z_k \in D_i \} + \# \{
w_j \in D_i \} \ge 3 .$$
Equivalently $(D,C,\ul{z})$ has no non-trivial automorphisms: 
$$\# \on{Aut}(D,C, \ul{z}) = 1 .$$
\end{enumerate}
\end{definition}

We introduce the following notations for moduli spaces of quilted
disks.  Denote by $\RR^{d,0}_\Gamma$ for the set of isomorphism
classes of combinatorial type $\Gamma$.  The moduli space
$\ol{\RR}^{d,0}$ is the union of $\RR^{d,0}_\Gamma$ over combinatorial
types $\Gamma$.  As before in \eqref{fi}, forgetting a marking and
stabilizing produces a fiber bundle $f_i: \RR^{d,0} \to \RR^{d-1,0}$
with interval fibers isomorphic to the part of the boundary between
the $i-1$ and $i+1$-st markings.  This implies by induction that
$\ol{\RR}^{d,0}$ is a topological disk.  There is also a forgetful map
\begin{equation} \label{forgetseam} f: \ol{\RR}^{d,0} \to
  \ol{\RR}^d \end{equation}
forgetting a seam which is, restricted to any stratum
$\RR^{d,0}_\Gamma$, a fiber bundle with open interval fibers.  For
example, the top-dimensional stratum $\RR^{d,0}$ may be identified
with the space of tuples $(w_2 < \ldots < w_d)$ by fixing $w_1 = 0$
and taking the seam to be given by $\on{Im}(z) =1$ in
$\H \cong D - \{ 1 \}$.  The forgetful morphism for forgetting the
seam is then in coordinates
$$ (w_2 < \ldots < w_d) \to (w_3/w_2 < \ldots < w_d/w_2 ) .$$
This map is smooth and surjective with fiber the space of quilted
disks represented by tuples
$[0,\lambda , \lambda w_3/w_2, \ldots, \lambda w_d/w_2]$.

\begin{example} {\rm (The third multiplihedron)}  The moduli space
$\ol{\RR}^{3,0}$ is the hexagon shown in Figure \ref{m31}.  The
  picture shows how the interior circle on the open stratum can
  ``bubble off'' into the bubble disks on the boundary.
\end{example} 

\begin{figure}[ht]
\includegraphics[height=3in]{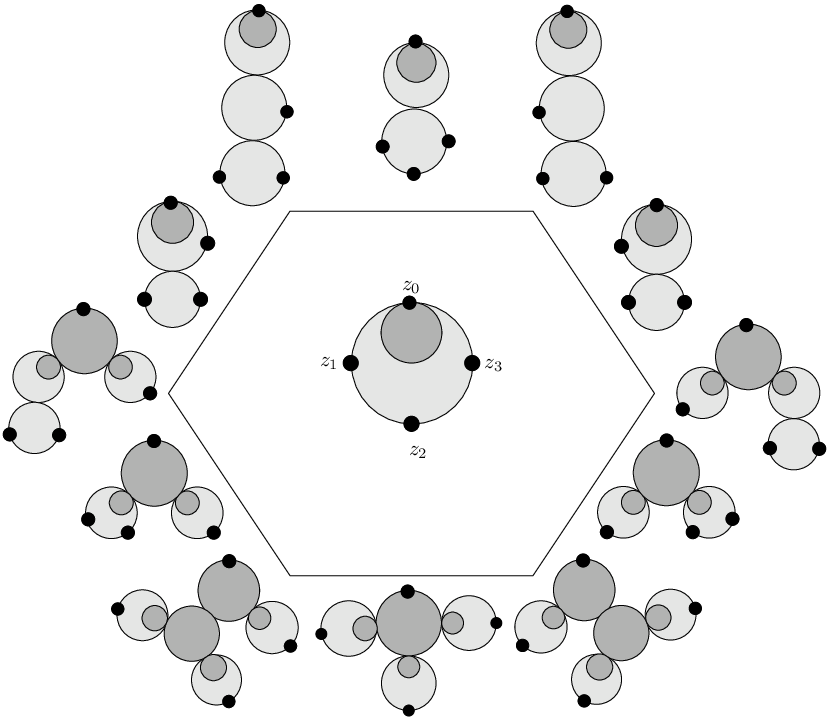}
\caption{$\ol{\RR}^{3,0}$}
\label{m31}
\end{figure}
The local structure of the moduli space of quilted disks is described
as follows, see Ma'u-Woodward \cite{mau:mult}.

\begin{definition} \label{balanced} 
{\rm (Balanced gluing parameters)} A collection of gluing parameters 
$$ \delta: \Edge_{< \infty}(\Gamma) \to [0,\infty)$$
is {\em balanced} if the following condition holds: for each pair
$v_1,v_2 \in \Ver^{(1)}(\Gamma)$ of colored vertices, let
$\gamma_{12}$ denote the shortest path in $\Gamma$ from $v_1$ to
$v_2$.  Then the relation
\begin{equation} \label{balancedrel}
\prod_{e \in \gamma_{12}} \delta(e)^{\pm 1} = 1, \quad \forall v_1,v_2
\in \Ver^{(1)}(\Gamma)\end{equation}
holds where the sign is $+1$ resp. $-1$ if the edge points towards
resp. away from the root edge.
\end{definition} 

The following theorem describes the local structure of the moduli
space of quilted disks near any stratum.  Let $Z_\Gamma \subset
\Map(\Edge_{< \infty}(\Gamma),\R_{\ge 0})$ be the set of balanced
gluing parameters from Definition \ref{balanced}.

\begin{theorem} 
 \label{glueRd0}
{\rm (Existence of compatible tubular neighborhoods)} For any integer $d \ge 1$ there exists a collection
of open neighborhoods $U_\Gamma$ of $0$ in $Z_\Gamma$ and collar
neighborhoods
$$ G_\Gamma:  \RR_\Gamma^{d,0} \times U_\Gamma \to \ol{\RR}^{d,0} $$
satisfying the following compatibility property: If
$\RR^{d,0}_{\Gamma'}$ is contained in the closure of
$\RR^{d,0}_{\Gamma}$ and the local coordinates on $\RR^{d,0}_\Gamma$
are induced via the gluing construction from those on
$\RR^{d,0}_{\Gamma'}$ then the diagram
$$ 
\begin{diagram}
\node{\RR_{\Gamma'}^{d,0} \times U_{\Gamma'}}  \arrow[2]{e}\arrow{se} \node{}
\node{\ol{\RR}_{\Gamma}^{d,0} \times U_{\Gamma} } \arrow{sw} \\
\node{} \node{\ol{\RR}^{d,0}} \node{} 
\end{diagram} 
$$
commutes. 
\end{theorem} 

\begin{proof}[Sketch of proof] 
The proof uses a version of the gluing construction for nodal disks to
the quilted case.  Let $D$ be a stable nodal quilted disk and $\delta
\in Z_\Gamma$.  The {\em glued disk} $G_\delta(D)$ is defined as
follows.  For each component $D_j$, let $w_j$ be the node connecting
the disk with component containing the root marking $z_0$, or $z_0$ if
$D_j$ is that component.  We assume that for each $D_j$ an
identification of $D_j - \{ w_j \}$ with half-space $\H = \{ \Im(z)
\ge 0 \}$ has been fixed.  Such an identification is given by fixing
two additional markings or nodes for an unquilted component, or one
additional marking or node and the seam of the quilt as the line $C =
\{ \Im(z) = 1 \}$.  Note that the space of such coordinates is convex.
\begin{enumerate} 
\item {\rm (Unquilted case)} In the case of a node not meeting any
  seam of a quilted disk corresponding to a gluing parameter
  $\gamma_j$, remove small balls around the node and glue together
  small annuli around the nodes using the map $z \mapsto \gamma_j z$.
  (Note the coordinate on the component ``further away from
  $z_0$'' \label{unquiltedcase} is already inverted.)
\item {\rm (Quilted case)} In the case of several nodes meeting seams
  of quilted disks, remove small balls around the nodes, glue together
  annuli around the nodes using the map $z \mapsto \gamma_j z$.
  Define the seam on the glued component is $C = \{ \Im(z) = \gamma_j
  \}$ independent of $j$ by the relation on the gluing parameters.
\end{enumerate} 

The collar neighborhoods of the strata are given by the following
global version of the gluing construction of the previous paragraph.
Suppose that for each point $r \in \RR^{d,0}_\Gamma$ a collection of
local coordinates as above is given varying smoothly in $r$.  Let
$U_\Gamma$ be a neighborhood of $0$ in $Z_\Gamma$.  Construct a collar
neighborhood $ G_\Gamma: \RR_\Gamma^{d,0} \times U_\Gamma \to
\ol{\RR}^{d,0} $ by mapping each disk to the isomorphism class of the
corresponding glued disk.  Since the space of coordinates on the disks
is contractible, we may assume that we have chosen local coordinates
so that whenever a point $r \in \RR^{d,0}_{\Gamma}$ is in the image of
such a gluing map from $\RR^{d,0}_{\Gamma'}$, the local coordinates
are those induced from $\RR^{d,0}_{\Gamma'}$.
\end{proof}  

\noindent Thus the moduli space of quilted disks is {\em equipped with
  quilt data} as in Definition \ref{quiltdata}: each stratum comes
with a collar neighborhood described by gluing parameters compatible
with the lower dimensional strata.  The space of gluing parameters is
the positive real part of an affine toric variety by
\cite[4.12]{morph}, and in particular, is isomorphic as a decomposed
space to a convex polyhedral cone.  Thus the moduli of quilted disks
is a locally polyhedral stratified space, and in particular there
exists a collar neighborhood of the boundary.

\begin{remark}  {\rm (Orientations)}  
  Orientations on $\RR^{d,0}$ for $d \ge 1$ can be constructed as
  follows.  The open stratum $\RR^{d,0}$ may be identified with the
  set of sequences $0 = w_1 < \ldots < w_{d}$, by identifying the
  complement of the $0$-th marking in the disk with the half-plane and
  using the translation symmetry to fix the location of the first
  marked point.  The bubbles form either when the points come
  together, in which case a disk bubble forms, or when the markings go
  to infinity, in which case one re-scales to keep the maximum
  distance between the markings constant and then has possibly quilted
  disk bubbles for the markings that come together at the same rate
  that the last marking goes to infinity.  In particular this
  realization induces an orientation on $\RR^{d,0}$.  The boundary
  strata are oriented by their identifications with $\RR^{e,0}$ and
  $\RR^f$.
\end{remark} 

\begin{proposition} \label{signs2} {\rm (Signs of boundary inclusions
for the multiplihedron)} The sign of the inclusions of boundary strata
  are
\begin{enumerate} 
\item {\rm (Facets corresponding to unquilted bubbles)}
  $(-1)^{ij + j} $ for facets given by embeddings
  $\RR^i \times \RR^{d-i+1,0} \to \ol{\RR}^{d,0}$ as for the
  associahedron in \eqref{boundary};
\item {\rm (Facets corresponding to quilted bubbles)}
  $ 1 + \sum_{j=1}^m (m-j) (i_j - 1)$ for facets given by
  embeddings
  $\RR^m \times \RR^{i_1,0}\times \ldots \times \RR^{i_m,0} \to
  \ol{\RR}^{d,0}$.
\end{enumerate} 
\end{proposition}  

\begin{proof} The first claim follows from the explicit description of
  the gluing map as in \eqref{signs} 
\begin{multline}
 (\delta,(z_3,\ldots,z_i), (w_2,\ldots,w_{d-i+1})) \\ \to (w_2,
  w_3,\ldots, w_{j+1}, w_{j+1} + \delta , w_{j+1} z_3 , \ldots, w_{j+1} + \delta 
  z_{i} , w_{j+2},\ldots, w_{d-i+1}) \end{multline}
which has gluing sign $ij +j$.  For the second, a map equivalent to
the gluing map is (taking $\RR^{m} \cong (0,1)^{m-2}$ by setting the
first resp. last point equal to $0$ resp. $1$)
$$ \R \times \RR^m \times \prod_{j=1}^m \RR^{i_j,0} \to \RR^{d,0}$$
\begin{multline} \label{gluemap2} (\delta, z_3,\ldots, z_{m},
  (w_{2,j},\ldots, w_{i_j,j} )_{j=1}^m)  \mapsto (w_{2,1},\ldots,
  w_{i_1,1}, \delta^{-1} ,\delta^{-1}  + w_{2,2}, \ldots , \\
  \delta^{-1}  + w_{i_2,2}, z_3 \delta^{-1}, \ldots, \delta^{-1} z_m, \delta^{-1} z_m +
  w_{2,m}, \ldots, \delta^{-1}z_m + w_{i_m,m}) \end{multline}
from which the claim follows. 
\end{proof} 

\subsection{The \ainfty functor for a correspondence} 

The definition of the functor on objects is trivial: by definition we
have allowed sequences of Lagrangian correspondences as objects and
the functor adds the correspondence to the sequence:

\begin{definition}  {\rm (Functor for Lagrangian correspondences
on objects)} Let $M_0,M_1$ be symplectic backgrounds 
 with
  the same monotonicity constant and $L_{01} \subset M_0^- \times M_1$
  an admissible Lagrangian brane equipped with a width $\delta_0 > 0$.
  Define
  \begin{equation} \label{phichange} \Phi(L_{01},\delta_0): \GFuk(M_0) \to
    \GFuk(M_1) \end{equation} 
  on objects by
  \begin{multline} \label{defobj}
    \Phi(L_{01},\delta_0)(L_{(-r)(-r+1)},
    \ldots,L_{(-1)0},\delta_{-r},\ldots, \delta_{-1}) \\=
    (L_{(-r)(-r+1)}, \ldots,L_{(-1)0},L_{01},\delta_{-r},\ldots,
    \delta_0) .\end{multline}
\end{definition} 

The definition of the functor on morphisms is by a count of quilted
surfaces with strip-like ends, where the domain of the quilt is one of
a family of quilted surfaces parametrized by the multiplihedron.  For
analytical reasons, this requires replacing the family of quilted
disks in the previous section with one in which degeneration is given
by neck-stretching:

\begin{proposition}
\label{Ad0} 
{\rm (Existence of families of quilted surfaces with strip-like ends)} 
Let $\ul{L}_0,\ldots, \ul{L}_d$ be a collection of generalized
Lagrangian branes as above, and $L_{01}$ a Lagrangian correspondence.
There exists a collection of families of quilted surfaces
$\ol{\SS}^{d,0}$ with strip-like ends over $\ol{\RR}^{d,0}$ for
$d \ge 1$ with the given widths of strip-like ends and the additional
properties that
\begin{enumerate}
\item {\rm (Recursive definition on boundary)} the restriction
  $\SS^{d,0}_\Gamma$ of the family to a stratum $\RR_\Gamma^{d,0}$
  isomorphic to a product of $\RR^{e_i,0}$ and $\RR^{f_j}$ is a
  product of the corresponding families of surfaces and quilted
  surfaces with strip-like ends, and
\item {\rm (Gluing near boundary)} collar neighborhoods of
  $\SS^{d,0}_\Gamma$ are given by gluing along strip-like ends.
\end{enumerate}
\end{proposition} 

\begin{proof} 
The claim follows by induction using Theorem \ref{A}, starting from
the case of three-marked disk. In that case we choose a genus zero
surface with strip-like ends, using the already constructed families
of surfaces with strip-like ends in Proposition \ref{Ad}.
\end{proof} 

\begin{proposition} 
  \label{Bd0} {\rm (Existence of compact families of pseudoholomorphic
    quilts)} Let $M_0,M_1$ be symplectic backgrounds with the same
  monotonicity constant, $\ul{L}^0,\ldots,\ul{L}^d$ admissible
  generalized Lagrangian branes in $M_0$, and
  $L_{01} \subset M_0^- \times M_1$ an admissible generalized
  Lagrangian correspondence from $M_0$ to $M_1$.  For generic choice
  of perturbation data the moduli space $\ol{\M}^{d,0}$ of
  pseudoholomorphic quilts with target in $M_0,M_1$ and boundary and
  seam conditions $\ul{L}^j, j =0,\ldots,d, L_{01}$ is such that
\begin{enumerate}
\item the expected dimension zero component $\M^{d,0}_0$ is finite;
  and
\item the expected dimension one component $\ol{\M}^{d,0}_1 \subset
  \ol{\M}^{d,0} $ has boundary given
  by the union
$$ \partial \ol{\M}^{d,0}_1 = \bigcup_{\Gamma}
  {\M}^{d,0}_{\Gamma,1} $$
where either (1) $\Gamma$ is a stable combinatorial type and
$\RR_{\Gamma}^{d,0}$ is a codimension one stratum in $\ol{\RR}^{d,0}$
in which case ${\M}^d_{\Gamma,1}$ is the product of (possibly more
than two!)  quilted and unquilted components corresponding to the
vertices of $\Gamma$, or (2) $\Gamma$ is unstable and corresponds to
bubbling off a Floer trajectory.
\end{enumerate}
\end{proposition} 

\begin{proof} Regular perturbation data exist by applying Theorem
  \ref{B} recursively to the family of quilts
  $\ol{\SS}^{d,0} \to \ol{\RR}^{d,0}$ constructed in Proposition
  \ref{Ad0}, taking the perturbation over the boundary of
  $\ol{\SS}^{d,0}$ to be that of product form for the
  lower-dimensional spaces over boundary strata.  Compactness and
  gluing are proved in Ma'u \cite{mau:gluing}.
\end{proof}  

From Theorem \ref{C} we obtain chain-level invariants from the moduli
spaces of pseudoholomorphic quilts.  The functor on morphism spaces is
related to these invariants by additional signs: Define
$$ \Phi(L_{01},\delta_0)^d: \Hom( \ul{L}^0,\ul{L}^1) \times \ldots \times
\Hom(\ul{L}^{d-1},\ul{L}^d) \to \Hom(\Phi(L_{01},\delta_0)(\ul{L}^0),
\Phi(L_{01},\delta_0)(\ul{L}^d)) $$
(see \eqref{defobj} for the definition of $\Phi(L_{01},\delta_0)$ on objects)
by setting for generalized intersection points $x_1,\ldots,
x_d$
\begin{equation} \label{phid} 
\Phi(L_{01},\delta_0)^d( 
\bra{x_1},\ldots, \bra{x_d}) = \Phi_{\SS^{d,0}}
(\bra{x_1},\ldots,\bra{x_d}) (-1)^{\heartsuit} \end{equation}
where $\heartsuit$ is defined in \eqref{heartsuit}.

\begin{theorem}
\label{ainftyfunctorthm} 
 {\rm (\ainfty functor for a Lagrangian correspondence)}.  Let
 $L_{01}$ be a Lagrangian correspondence from $M_0$ to $M_1$ with
 admissible brane structure and and $\Phi(L_{01},\delta_0)$ defined by
 \eqref{phid}.  Then $\Phi(L_{01},\delta_0)$ is an \ainfty functor from
 $\GFuk(M_0)$ to $\GFuk(M_1)$.
\end{theorem}

\begin{proof}   
  The proof uses the description of the boundary of the
  one-dimensional moduli space in Theorem \ref{Bd0}: Let $d \ge 0$ be
  an integer and $x_0,\ldots,x_d$ generalized intersection points of
  $(\ul{L}^0, L_{01}, L_{01}^t, \ul{L}^d), (\ul{L}^0,
  \ul{L}^1),\ldots, (\ul{L}^{d-1},\ul{L}^d)$.
  The boundary of the one-dimensional component
  $\ol{\M}^{d,0} (x_0,\ldots,x_d)_1$ of the moduli space of
  pseudoholomorphic quilts with limits $x_j, j = 0,\ldots d$ consists of
  three combinatorial types: configurations containing (1) a single
  unquilted bubble, (2) a collection of quilted bubbles, or (3) a
  bubbled trajectory.  These three types of terms correspond to the
  terms in the definition of \ainfty functor \eqref{faxiom}.  The
  signs for the terms of the first type are similar to those for the
  \ainfty axiom in the proof of \ref{ainftycat} were computed to equal
  $\sum_{k=1}^d (1+ |x_k|)$ due to offsetting contributions of $m$ in
  \eqref{signa}, \eqref{signc}.  For terms of the second type we
  suppose that the bubbles define a partition
  $I_1 \cup \ldots \cup I_m = \{ 1,\ldots d \}$, with the markings
  $z_j, j \in I_j$ on the $j$-th bubble, each containing an interior
  circle.  To check the signs we determine the sign of the gluing
  isomorphism
\begin{multline} 
 \det(\R) \det(T \ol{\M}^m(\ul{y}_0,\ldots,\ul{y}_m))
 \bigotimes_{j=1}^m \det(T \ol{\M}^{i_j,0}(\ul{y}_i,x_{I_j})) \\ \to
 \det(T \ol{\M}^{d,0}(\ul{y}_0,x_1,\ldots,x_d)) \end{multline}
where $x_{I_j} = (x_i)_{i \in I_j}$.  The orientation on the former is
determined by an isomorphism involving the determinant lines
$\DD_{x_j}^\pm, \DD_{\ul{y}_k}^\pm$ attached to the intersection
points in \cite{orient}, c.f.  \cite{wo:gdisk}:
\begin{multline} 
 \det( \R )  \det( T \ol{\RR}^m )
\DD_{\ul{y}_0}^+  \DD_{\ul{y}_1}^- \ldots  \DD_{\ul{y}_m}^- 
\bigotimes_{j = 1}^m
 \left( \det( T {\RR}^{i_j,0} )
\DD_{\ul{y}_j}^+  \bigotimes_{k \in I_j} \DD_{x_k}^- 
\right) 
.\end{multline}
Permuting each $\det(T \RR^{i_j,0}) $ past the previous determinant
lines produces a factor
\begin{equation} \label{signs2a} 
\sum_{j=1}^m (i_j - 1)  m + \sum_{j=1}^m (i_j - 1) \sum_{k=1}^{j-1} (i_k - 1)) . \end{equation} 
Permuting the $\DD_{y_j}^-$ to be adjacent to the corresponding 
$\DD_{y_j}^+$ produces signs of the amount 
\begin{equation} \label{signs2b} 
\sum_{j =1}^m |y_j| \sum_{k=1}^{j-1} (i_k - 1)
 . \end{equation} 
There is also a contribution from the signs in the definition of 
$\phi_{i_j}$ and the sign from the definition of $\mu^m$,
\begin{equation} \label{signs2c} \sum_{j = 1}^m \sum_{i=1}^{i_j}
  i|x_{i+l_j}| + \sum_{j=1}^m j |\ul{y}_j|  \end{equation}
where $l_i = i_1 + \ldots + i_{j-1}$.  The gluing map has sign given
in Proposition \ref{signs2}
\begin{equation} \label{signs2d} 1 +  \sum_{j=1}^m (m-j) (i_j -
  1). \end{equation}
Combining \eqref{signs2a}, \eqref{signs2b}, \eqref{signs2c},
\eqref{signs2d}, the total number of signs mod $2$ is 
\begin{multline} \label{total} 
\sum_{j=1}^m (i_j - 1)  \left(  m + \sum_{k=1}^{j-1} (i_k - 1) \right)
+ \sum_{j =1}^m |y_j| \sum_{k=1}^{j-1} (i_k - 1)
\\ +
\sum_{j = 1}^m \sum_{i=1}^{i_j}
  i|x_{i+l_j}| + \sum_{j=1}^m j |\ul{y}_j| 
 + \sum_{j=1}^m (m-j) (i_j -
  1) +1\\
\equiv 
\sum_{j=1}^m (i_j - 1)  m + 
\left( \sum_{j =1}^m 
(|x_{l_j+1}| + \ldots + | x_{l_i + i_j}| ) 
\right) 
\sum_{k=1}^{j-1} (i_k - 1) 
\\ +  \sum_{j = 1}^m \sum_{i=1}^{i_j}
  i |x_{i+l_j}| + \sum_{j=1}^m j |\ul{y}_j| 
+ \sum_{j=1}^m (m-j) (i_j -  1) +1  \\
\equiv 
\sum_{j=1}^m (i_j - 1)  m + \sum_{j=1}^d j |x_j| + 
 \left( \sum_{j =1}^m (|x_{l_j+1}| + \ldots + | x_{l_i + i_j}| ) 
\right) (j-1) 
\\ +
 \sum_{j=1}^m j (|x_{l_j+1}| + \ldots + | x_{l_i + i_j}| + i_j - 1) 
 + \sum_{j=1}^m (m-j) (i_j -  1)  +1  \\
\equiv 
1+ \sum_{j=1}^d (j+1) |x_j| . 
 \end{multline} 
This sign is opposite to the gluing sign \eqref{assocsign} for the other
 type of facet, proving the statement of the Theorem.
\end{proof}

\begin{remark} {\rm (Functors for generalized Lagrangian correspondences)}  
More generally, define functors for generalized Lagrangian
correspondences as follows.  Let $\ul{L} = (L_{01},L_{12},\ldots,
L_{(k-1)k})$ be an admissible generalized Lagrangian correspondence
with brane structure from $M_0$ to $M_k$ together with a sequence of
widths $ \delta_0,\ldots, \delta_{k-1} $.  Define an \ainfty functor
$$ \Phi(\ul{L},\ul{\delta}):\ \GFuk(M_0) \to \GFuk(M_k) $$ %
for objects by concatenation with $\ul{L}$ as in Section \ref{concat}
$$ \Obj(\GFuk(M_0)) \ni \ul{L}' \mapsto \ul{L}' \sharp \ul{L} \in
\Obj(\GFuk(M_k)) $$
and for morphisms by counting quilted disks of the following form:
Replace each seam in the family $\ol{\SS}^d$ by a sequence of infinite
strips of widths $\delta_0,\ldots, \delta_{k-1}$ to obtain a family
$\ol{\SS}^d(\delta_0,\ldots, \delta_{k-1})$.  Then the map on
morphisms is given by the formula \eqref{phid} but using the relative
invariant for $\ol{\SS}^d(\delta_0,\ldots, \delta_{k-1})$.  The proof
of the \ainfty axiom is similar to the unquilted case.  We show in
Section \ref{indep} that the functor $\Phi(\ul{L},\ul{\delta})$ is independent of
the choice of widths, up to quasiisomorphism.
\end{remark} 

\subsection{The functor for the empty correspondence}

In this section we discuss the functor associated to the empty
correspondence, that is, the sequence of length zero. Let $M$ be a
symplectic background.  The empty sequence $\emptyset$ may be
considered as an element of $\GFuk(M,M)$, namely a sequence of length
zero.  The corresponding functor
$\Phi(\emptyset): \GFuk(M) \to \GFuk(M) $ is defined by counting pairs
$r,u$ where $r \in \RR^{d,0}$ and $u : \ul{S}_r^{d,0,'} \to M$ is a
map from the surface with strip-like ends $\SS^{d,0,'}_r$ obtained
from $\SS^{d,0}_r$ by removing the seam.

\begin{proposition} \label{emptyset} {\rm (Functor for the length zero
    correspondence)} Let $M$ be a symplectic background.  For suitably
  chosen complex structures on the fibers of
  $\ol{\SS}^{d,0} \to \ol{\RR}^{d,0}$ and coherent, regular,
  perturbation data, the functor
  $\Phi(\emptyset): \GFuk(M) \to \GFuk(M) $ is the identity functor.
\end{proposition} 

\begin{proof} We wish to construct the moduli spaces $\ol{\M}^{d,0}$
  so that there is a {\em forgetful map}
$\ol{\M}^{d,0} \to \ol{\M}^{d}, \quad (r,u) \mapsto (r',u)$
where $r' \in \ol{\RR}^d$ is the image of $r$ under the forgetful map
$\ol{\RR}^{d,0} \to \ol{\RR}^d$.  It then follows that a pair $(r,u)
\in \ol{\M}^{d,0} $ can be isolated if and only if $u$ is constant and
$d = 1$.  Indeed, otherwise the fibers are one-dimensional.  The
resulting count gives
$ \Phi(\emptyset)_1 = \Id_{CF(L,L)}$ and $\Phi(\emptyset)_d = 0 , d
> 0 .$

However, it is {\em not} possible to achieve a moduli space admitting
a forgetful map with the construction that we gave before.  Namely,
our previous construction (which worked for any correspondence)
adapted the complex structure on the quilted surface so that the seam
was real analytic.  The family of complex structures that one obtains
depends on the location of the seam, hence destroys the forgetful map.
Fortunately in this case it is not necessary that the seam be real
analytic, since a quilted pseudoholomorphic map with diagonal seam
condition can be considered as an unquilted pseudoholomorphic map with
no seam, and so all the necessary analytic results apply.

A simpler construction of perturbation data produces the identity
functor.  Choose a family of strip-like ends for the universal curve
over $\ol{\RR}^d$, and pull these back to strip-like ends on the
universal quilted disk over $\ol{\RR}^{d,0}$.  With respect to these
strip-like ends (given as local coordinates near the boundary marked
points) the quilted circle does not appear linear near the ends, but
this does not affect the construction since there is no seam
condition.  Choose perturbation data $\ul{J}_{d,0},\ul{K}_{d,0}$ for
the quilted surfaces given by pull-back of perturbation data
$\ul{J}_d,\ul{K}_d$ from $\ol{\RR}^d$.  Since for any point $p \in
\ol{\RR}^{d,0}$ mapping to $q \in \ol{\RR}^d$ (we may assume that $q$
is in a codimension $0$ or $1$ stratum) the linearized forgetful map
$T_p \ol{\RR}^{d,0} \to T_q \ol{\RR}^d$ is surjective, regularity of a
map $\ul{u} \in \ol{\RR}^d$ implies regularity of any map in the
fiber.  Counting quilted pseudoholomorphic maps with no seam condition is
then the same as counting pseudoholomorphic maps of the underlying disks,
together with a choice of lift from $\ol{\RR}^d$ to $\ol{\RR}^{d,0}$.
\end{proof}  

\section{Natural transformations for Floer cocycles}
\label{natural}

In this section we associate to any Floer cocycle for a pair of
correspondences a natural transformation between the functors
constructed in the previous section.  We then complete the proof of
Theorem \ref{mainfunc}.  

\subsection{Quilted disks with boundary and seam markings}
\label{biassoc}

The natural transformations are defined by counts of quilted disks with
markings on the interior circle.  In the manner of Stasheff
\cite{st:ho} we define for any pair $d,e$ of positive integers a cell
complex $\K^{d,e}$ inductively.  Each face of $\K^{d,e}$ corresponds
to an expression in formal {\em variables } $a_1,\ldots,a_d$, formal
{\em 1-morphism symbols} $ h$, and formal {\em $2$-morphism symbols
  $t_1,\ldots,t_e$}. Each type of symbol must appear in the given
order with parentheses in an expression of the form
$$ 
h \left( \begin{array}{c} t_1 \ldots t_e \\ a_1 \ldots 
  a_d \end{array} \right) .$$
Each such expression corresponds to a {\em pair} of colored trees
$\Gamma = (\Gamma_1,\Gamma_2)$ equipped with an isomorphism of
subtrees $\Gamma_1^- \to \Gamma_2^-$ of subtrees above the colored
vertices.  The expression is built from the tree by using the parts of
the tree below the colored vertices to build the parenthesized
expressions in the variables $t_i$ and $a_j$, using the rule that any
uncolored vertex corresponds to a parenthesis enclosing the symbols
attached to the incoming markings, and the part of the tree above the
colored vertices to build the parenthesized expressions in the
expressions $h(\cdot)$, where $\cdot$ represents the expression below
the colored vertex.  Thus in the figure on the right Figure
\ref{fig:Q35}, where the red color is used for the second tree below
the colored vertices, the resulting expression is \label{resultingexp}
$$ h \left( \begin{array}{c} (t_1 t_2) \\ \   \end{array} \right) h
\left( \begin{array}{c}  \\ a_1(a_2 a_3) \end{array} \right) 
 h \left( \begin{array}{c} t_3 \\ a_4 a_5 \end{array} \right) .$$
We define
$$\Ver(\Gamma) = (\Ver(\Gamma_1) \sqcup \Ver(\Gamma_2))/\sim $$ 
the set of vertices modulo the natural identification of the subtrees
and similarly for the edge set.  The subtrees
$\Gamma_1^- \cong \Gamma_2^-$ describe the parenthesizations of the
symbols of the form $h(\cdot)$, while the trees below the colored
vertices describe the parenthesization of symbols $t_i$ resp. $a_j$.
There is a natural partial order on such tree pairs where
$\Gamma \preceq \Gamma'$ if there is a morphism of trees
$\Gamma_k \to \Gamma'_k, k = 1,2$ contracting edges compatible with
the identifications $\Gamma_1^- \cong \Gamma_2^-$.  For any
combinatorial type of bicolored tree $\Gamma$ and each colored
resp. uncolored vertex $v$ of $\Gamma$, let $t(v) = (d(v),e(v))$
resp. $t(v) = d(v)$ denote the number of edges of each type incident
to $v$.  The pair $\Gamma$ is {\em stable} if each uncolored vertex
$v$ has valence $t(v)$ at least three, while each colored vertex $v$
has valences $(d(v),e(v))$ with either $d(v) \ge 1$ or $e(v) \ge 1$.
Thus for example the expressions
$h \left( \begin{array}{c} (t_1 t_2) \\ \ \end{array} \right) $
$h \left( \begin{array}{c} t_1 \\ \ \end{array} \right) h 
\left( \begin{array}{c} t_2 \\ \ \end{array} \right) $
correspond to stable combinatorial types.  On the other hand the expressions 
$ \left( h \left( \begin{array}{c} (t_1 t_2) \\ \ \end{array} \right) \right)$
and 
$h \left( \begin{array}{c} (t_1) \\ \ \end{array} \right) h 
\left( \begin{array}{c} (t_2) \\ \ \end{array} \right) $
are unstable combinatorial types, since a single expression in a 
parenthesis corresponds to an uncolored vertex with valence two.

The definition of the Stasheff-style complexes is inductive.  Suppose
$\K^{d',e'}$ have been defined for $d' <d$ or $e' < e$ in such a way
that each $\K^{d',e'}$ is a stratified space with decomposition into
stable combinatorial types
$$ \K^{d',e'} = \bigcup_\Gamma \K^\Gamma, \quad \K^{\Gamma} :=
\prod_{v \in \Ver(\Gamma)} \K^{t(v)} .$$
If $\pi: \Gamma \to \Gamma'$ is a morphism of stable combinatorial
types then there is a natural inclusion
$$ \iota_\Gamma^{\Gamma'} : \K^{\Gamma} \to \K^{\Gamma'} $$
induced by the maps $\K^{\pi^{-1}(v)} \to \K^{t(v)}$, where
$\pi^{-1}(v) \subset \Gamma$ is the subgroup that maps to $v$.   
Define 
$$ \partial \K^{d,e} = \bigcup_{\Gamma} \K^\Gamma / \sim $$
where the union is over all combinatorial types with at least two
vertices, and $\sim$ is the equivalence relation induced by the maps
$\iota_\Gamma^{\Gamma'}$.  Then $\partial \K^{d,e}$ is naturally a
stratified space whose links are the disjoint unions of the links in
the stratified spaces $\K^\Gamma$.  Let
$$\K^{d,e} = \Cone( \partial \K^{d,e} ) $$
denote the cone on $\partial \K^{d,e}$.  Define
$ \on{int} (\K^\Gamma) = \K^\Gamma - \partial \K^\Gamma .$ We call the
sets $\on{int}(\K^\Gamma)$ {\em faces} of $\K^{d,e}$.

\begin{example} {\rm (One or two seam markings)}
  $\K^{1,1}$ is the interval with faces 
corresponding to the expressions 
$$ h \left( \begin{matrix} t_1 \\ a_1 \end{matrix} \right), \quad h
  \left( \begin{matrix} \ \\ a_1 \end{matrix} \right) h
  \left( \begin{matrix} t_1 \\\ \end{matrix} \right), \quad h
    \left( \begin{matrix} t_1 \\\ \end{matrix} \right) h
      \left( \begin{matrix}\ \\ a_1 \end{matrix} \right) .$$
$\K^{2,1}$ is the octagon with open face corresponding to the
expression $ h \left( \begin{matrix}
   t_1 \\ a_1 a_2 \end{matrix} \right)$ and facets corresponding to
the expressions 
$$
 h \left( \begin{matrix} t_1 \\ (a_1 a_2) \end{matrix} \right), 
 h \left( \begin{matrix} t_1 \\ a_1  \end{matrix} \right) 
 h \left( \begin{matrix}  \\  a_2 \end{matrix} \right) 
 h \left( \begin{matrix}  \\ a_1  \end{matrix} \right)
 h \left( \begin{matrix}  t_1 \\  a_2 \end{matrix} \right), 
 h \left( \begin{matrix} t_1 \\ \  \end{matrix} \right)
 h \left( \begin{matrix}  \\ a_1  \end{matrix} \right)
 h \left( \begin{matrix}  \\  a_2 \end{matrix} \right) ,$$
$$  h \left( \begin{matrix}  \\ a_1  \end{matrix} \right)
 h \left( \begin{matrix} t_1 \\  \  \end{matrix} \right)
 h \left( \begin{matrix}  \\  a_2 \end{matrix} \right) ,
 h \left( \begin{matrix}  \\ a_1  \end{matrix} \right)
 h \left( \begin{matrix}  \\  a_2 \end{matrix} \right)
h \left( \begin{matrix} t_1 \\ \  \end{matrix} \right) ,
 h \left( \begin{matrix} t_1 \\ \   \end{matrix} \right)
 h \left( \begin{matrix}  \\  a_1 a_2 \end{matrix} \right),
 h \left( \begin{matrix}  \\ a_1 a_2  \end{matrix} \right)
 h \left( \begin{matrix} t_1 \\ \   \end{matrix} \right)  .$$
\end{example} 

\begin{remark} \label{facrem} {\rm (Facets)}  
For any positive integers $d,e$, there is a bijection between facets
of $\K^{d,e}$ and the following expressions:
\begin{enumerate} 
\item insertion of parenthesis around a sub-expression $a_{i+1} \ldots
  a_{i+j-1}$;
\item insertion of parentheses around a sub-expression $t_i \ldots t_{i+j - 1}
  t_{i+j}$;
\item a product of expressions $h( \cdot) h(\cdot) \ldots h( \cdot)$,
  corresponding to a partition of the symbols $t_1,\ldots, t_e,
  a_1,\ldots,a_d$, such that each element of the partition contains at
  least one symbol.
\end{enumerate} 
In particular, the facets of $\K^{d,1}$ correspond to the terms in the
definition \eqref{mu1} of $\mu^1$ for pre-natural transformations.
\end{remark} 

The spaces above may be identified with moduli spaces of {\em stable
  quilted disks with markings on the boundary or seam}:

\begin{definition} \label{innerouterdef} {\rm (Quilted disks with
    inner and outer markings)} Let $d,e$ be positive integers.  A {\em
    quilted disk with $d$ outer markings and $e$ inner markings} is a
  tuple $(D,C,\zeta,z_1 \ldots, z_d, w_1,\ldots,w_e)$ where
\begin{enumerate}
\item $D$ is a holomorphic disk;
\item $C$ is a circle in $D$ with unique intersection point $C\cap
  \partial D =\{ \zeta\}$;
\item $(\zeta, z_1,\ldots, z_d)$ is a tuple of distinct points in
  $\partial D$ whose cyclic order is compatible with the
  counterclockwise orientation of $\partial D$;
\item $(\zeta,w_1, \ldots, w_e)$ is a tuple of distinct points in $C$
  whose cyclic order is compatible with the counterclockwise
  orientation of $C$.
\end{enumerate}
An {\em isomorphism} of nodal $(d,e)$-marked quilted disks from
$(D,C,\ul{z},\ul{w})$ to $(D',C',\ul{z}',\ul{w}')$ is a holomorphic
isomorphism of the ambient disks mapping the circles resp. markings of
the first to those of the second:
$$ \psi: D \to D', \quad \psi(C) = C', \quad \psi(z_j) = z_j', \ \ j =
0,\ldots, d, \quad \psi(w_k) = w_k', \ \ k = 1,\ldots, e .$$
\end{definition}
\noindent Denote by $\RR^{d,e}$ the moduli space of isomorphism classes of
stable $(d,e)$-marked quilted disks.  An element of $\RR^{d,e}$ can be
identified with a configuration in the upper half plane $\bH = \{
\on{Im}(z) \ge 0 \}$, of points
$$z_1, \ldots, z_d \in \R, \quad w_1, \ldots, w_e \in \R + i$$
modulo the group $\Aut(\bH)$ generated by real translations.  It
follows that $\RR^{d,e}$ is a manifold with $\dim_{\R} \RR^{d,e} = d+e-1$.  A compactification $\ol{\RR}^{d,e}$ of
$\RR^{d,e}$ is obtained by allowing stable nodal quilted disks.

\begin{definition}    {\rm (Stable quilted disks with inner and outer markings)}  
\label{stableinnerandouter}
\begin{enumerate} 
\item A {\em nodal $(d,e)$-marked quilted disk} $C$ is a collection of
  unquilted disks $D_i$, quilted disks $(D_j,C_j)$, and {\em quilted
    spheres} $(S_k,C_k)$: holomorphic spheres $S_k$ equipped with a
  circle {\em seam} $C_k \subset S_k$, isomorphic to the usual
  projective line $S_k \cong \C P^1$ with its real locus
  $C_k \cong \R P^1$ identified with the equator, with $d$ markings on
  the boundary $z_1,\ldots, z_d \in \partial D$ and $e$ markings on
  the seams disjoint from each other and the nodes.  The quilted
  spheres are attached to each other and the quilted disk by points on
  the seam, as explained below.
\item The {\em combinatorial type} of any nodal $(d,e)$-marked quilted
  disk $(D,C)$ is a tree $\Gamma$ whose vertices $\Ver(\Gamma)$
  correspond to components that are either a marked disk, a quilted
  disk (with or without inner markings) or a quilted sphere, finite
  edges $e \in \Edge_{< \infty}(\Gamma)$ represent nodes, and
  semi-infinite edges $e \in \Edge_\infty(\Gamma)$ represent the
  marked points.  The set of such edges is equipped with a bijection
  to the sets
  $\{ 0 \} \sqcup \{ 1 \} \times \{ 1,\ldots, d \} \sqcup \{ 2 \}
  \times \{ 1,\ldots, e \}$
  representing the root, outer and inner markings.  The other
  semi-infinite edges not corresponding to the root marking are called
  {\em leaves}.  The combinatorial type of $\Gamma$ is required to be
  a tree, and is equipped with a distinguished subset of {\em colored
    vertices} $\Ver^{\on{col}}(\Gamma)$ corresponding to quilted
  disks.  The combinatorial type $\Gamma$ is required to satisfy the
  following {\em monotonicity condition}: on any non-self-crossing
  path of vertices $v_1,\ldots, v_l \in \Ver(\Gamma)$ from the root
  edge to a leaf in $\Gamma$, exactly one vertex $v_i$ in the path is
  colored, that is, corresponds to a quilted disk, and the components
  after the quilted disk are either all unquilted disks or all quilted
  spheres.  One obtains a tree $\Gamma_1$ by taking $\Ver(\Gamma_1)$
  to be consist of those vertices on paths between the root edge and a
  leaf corresponding to an outer marking, and $\Gamma_2$ to be the
  vertices on a path from the root edge to an edge corresponding to an
  inner marking.  Each subtree $\Gamma_1,\Gamma_2$ then has a planar
  structure (ordering of the edges meeting each vertex) determined by
  the ordering of the leaves. \label{planarstr}
\item A nodal $(d,e)$-marked quilted disk is {\em stable} if it has no
  automorphisms, or equivalently, each disk component has at least $3$
  markings, seams, and nodes, and each sphere component has at least
  one seam and three markings or nodes.
\end{enumerate} 
\end{definition} 

\begin{example} \label{q35} 
An example of a nodal $(3,5)$-marked disk is given in Figure
\ref{fig:Q35}.  In the figure, we use black leaves for outer circle
markings, and red (in the online version) leaves for inner circle
markings. The black leaves are required to respect the order of the
outer marked points, likewise the red leaves respect the order of the
inner marked points.
\end{example} 

\begin{figure}[h]
\center{\includegraphics[height=2in]{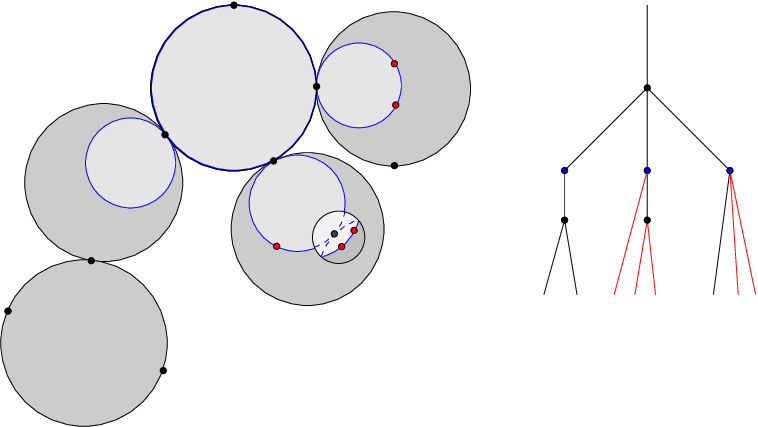}}
\caption{A nodal (3,5)-marked disk}\label{fig:Q35}
\end{figure}

\label{detop} A topology on the moduli space of marked quilted disks
can be defined by introducing a suitable notion of convergence.
Consider a sequence $(D_\nu, C_\nu, \ul{z}_\nu, \ul{w}_\nu)$ of marked
quilted disks with smooth domain.  That is, we assume that $D_\nu$ has
no nodes so that $D_\nu \cong \P^1 \cong \C \cup \{ \infty \}$. For
simplicity we may assume that the seam is given by
$\on{Im}(z) = \rho_\nu$ and the boundary given by
$\on{Im}(z) = \eps_\nu$.  Let $(D,C,\ul{z},\ul{w})$ be another marked
quilted disk. For $(D_\nu,C_\nu,\ul{z}_\nu,\ul{w}_\nu)$ to converge to
$(D,C,\ul{z},\ul{w})$ we require the following.  For each irreducible
component $D_i$ of $D$ let $\pi_i:D \to D_i$ denote the unique map
that is the identity on $D_i$ and constant on all other components.
We require a family $\phi_{\nu}:D_i \to D_\nu$ of holomorphic
embeddings (preserving the seams, if $D_i$ is quilted) such that for
each $k$, either $\phi_{\nu,i}^{-1}( z_{k,\nu} )$ converges to $z_k$,
if $z_k$ lies on $D_i$, or otherwise to a node $\pi_i(z_k)$ of $D_i$;
and similarly for each $j$ either $\phi_{\nu,i}^{-1}(w_{j,\nu})$
converges to $w_j$ or to a node $\pi_i(w_j)$.  The definition of
convergence for a fixed combinatorial type is similar, treating each
irreducible component separately.  In the case of sequences with
varying combinatorial type, a subsequence of each fixed type is
required to converge to the same limit.  We equip $\ol{\RR}^{d,e}$
with the topology given by declaring a subset to be closed if it is
closed under limits.  An argument analogous to \cite[5.6.5]{ms:jh}
shows that the convergence in the resulting topology is in fact the
above notion of convergence.  As before in \eqref{univcurve},
$\ol{\RR}^{d,e}$ comes with a universal curve $\ol{\UU}^{d,e}$
containing a fiberwise boundary $\partial \UU^{d,e}$ as well as a {\em
  universal seam} consisting of isomorphisms classes of tuples
$(D,C,\ul{z},\ul{w},y)$ where $y \in C$.

The local structure of these moduli spaces of quilted disks with seam
and boundary markings can be described in terms of balanced gluing
parameters.  Suppose that $\Gamma$ is a combinatorial type of stable
$(d,e)$-marked disk.  Recall that
$$Z_\Gamma \subset \Map(\Edge_{< \infty}(\Gamma),\R_{\ge 0})$$
is the set of balanced gluing parameters from Definition
\ref{balanced}.

\begin{theorem} [Existence of compatible tubular neighborhoods 
for the biassociahedra] \label{collarde} For positive integers $d,e$
  and for each combinatorial type $\Gamma$ of $(d,e)$-marked disk,
  there exist a collection of open neighborhoods $U_\Gamma$ of $0$ in
  $Z_\Gamma$ and collar neighborhoods
$$ G_\Gamma: \RR_\Gamma^{d,e} \times U_\Gamma \to \ol{\RR}^{d,e} $$
onto an open neighborhood of $\RR^{d,e}_\Gamma$ in $\ol{\RR}^{d,e}$ that
satisfy the following compatibility condition: If
$\RR^{d,e}_{\Gamma'}$ is contained in the closure of
$\RR^{d,e}_{\Gamma}$ and the local coordinates on $\RR^{d,e}_\Gamma$
are induced via the gluing construction from those on
$\RR^{d,e}_{\Gamma'}$ then the diagram
$$ 
\begin{diagram}
\node{\RR_{\Gamma'}^{d,e} \times U_{\Gamma'}} \arrow[2]{e}\arrow{se}
\node{} \node{\RR_{\Gamma}^{d,e} \times U_{\Gamma} } \arrow{sw}
\\ \node{} \node{\ol{\RR}^{d,e}} \node{}
\end{diagram} 
$$
commutes.  \end{theorem} 

\begin{proof}  The proof is a combination of disk gluing and sphere gluing; 
the former already appeared in the proof of Theorem \ref{glueRd0}.
Suppose the following are given:
\begin{itemize}
\item a collection of balanced gluing parameters
  $\delta_1,\ldots,\delta_m$,
\item a collection of coordinates $D_j - \{ w_j \} \to \H = \{ \Im(z)
  \ge 0 \}$ (for the disk components) so that the seam $C_j$ is
  identified with the set $\{ \Im(z) = 1 \}$; and
\item a collection of coordinates  $D_j - \{ w_j \} \to \C$ so that
the seam $C_j$ is $\{ \Im(z) = 0 \}$ (for the quilted sphere
components).
\end{itemize} 
Given these data define a {\em glued $(d,e)$-marked disk}
$G_\Gamma(D,\delta)$ by removing small balls around the nodes and
gluing small annuli via the relation $z \sim \delta_j z$ in the given
coordinates, starting recursively with the components furthest away
from the root marking $z_0$.  The seam on the glued component is given
by $\{ \Im(z) = \delta_j \}$.  The definition of the seam is
independent of $\delta_j$ by the balanced condition \eqref{balancedrel}.
This construction also works in families: given a family of such
coordinates over a stratum $\RR^{d,e}_\Gamma$, one obtains a collar
neighborhood as in the statement of the theorem.  Since the space of
coordinates on the disks is contractible, we may assume that we have
chosen local coordinates so that whenever a point $r \in
\RR^{d,e}_{\Gamma}$ is in the image of such a gluing map from
$\RR^{d,e}_{\Gamma'}$, the local coordinates are those induced from
$\RR^{d,e}_{\Gamma'}$.
\end{proof} 

In this sense, the stratified space $\ol{\RR}^{d,e}$ is {\em equipped
  with quilt data} as in Definition \ref{quiltdata}: each stratum
comes with a collar neighborhood described by gluing parameters
compatible with the lower dimensional strata.  As with the
multiplihedra, the nature of the relations on the gluing parameters
only give toric singularities.  That is, all corners are polyhedral.
The moduli space $\ol{\RR}^{2,1}$ is shown in Figure \ref{m21}.

\begin{figure}[h]
\includegraphics[width=3in]{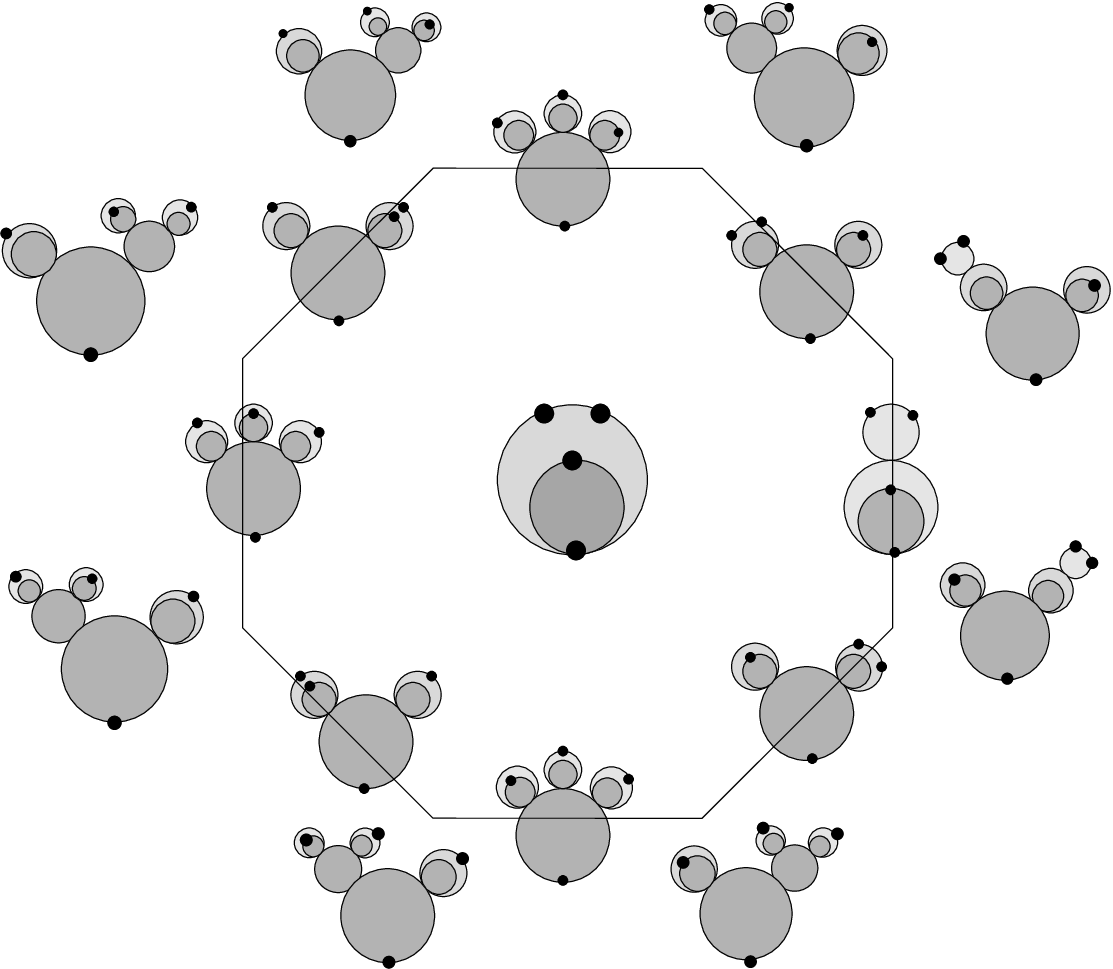}
\caption{$\ol{\RR}^{2,1}$}
\label{m21}
\end{figure}

The moduli space of quilted disks with boundary and seam markings
admits natural forgetful maps forgetting both boundary and seam
markings generalizing those in \eqref{fi}: For indices
$i \in \{ 1,\ldots, d \} $ and $j \in \{ 1,\ldots, e \}$ maps (for
$d,e \ge 1$)
\begin{equation} \label{fide} f_i: \ol{\RR}^{d,e} \to 
  \ol{\RR}^{d-1,e}, \quad f^j: \ol{\RR}^{d,e} \to 
  \ol{\RR}^{d,e-1} \end{equation}
given by forgetting the $i$-th boundary resp. $j$-th seam marking and
collapsing unstable components.  As in \eqref{univcurve}, the
forgetful morphisms $f_i$ may be identified with projection from part
of the boundary of the universal curve while the morphism $f^j$ may be
identified with the part of the universal seam between two markings,
with the additional caveat that each nodal point on the seam is
potentially replaced with two points as in Figure \ref{spart} (so the
map to the universal seam is two-to-one in this case.)

\begin{figure}[ht]
\begin{picture}(0,0)%
\includegraphics{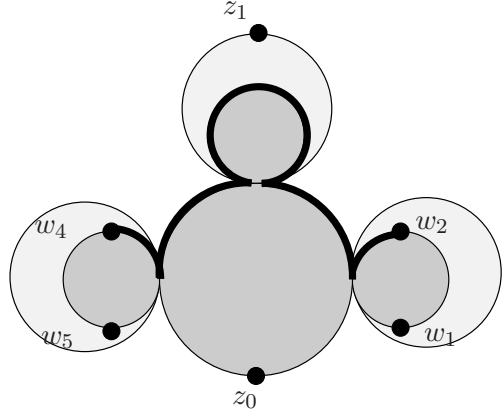}%
\end{picture}%
\setlength{\unitlength}{3947sp}%
\begingroup\makeatletter\ifx\SetFigFont\undefined%
\gdef\SetFigFont#1#2#3#4#5{%
  \reset@font\fontsize{#1}{#2pt}%
  \fontfamily{#3}\fontseries{#4}\fontshape{#5}%
  \selectfont}%
\fi\endgroup%
\begin{picture}(3101,2680)(2659,-2423)
\put(4059,-2381){\makebox(0,0)[lb]{{{{$z_0$}%
}}}}
\put(5265,-1959){\makebox(0,0)[lb]{\smash{{{{{$w_1$}%
}}}}}}
\put(2863,-1974){\makebox(0,0)[lb]{\smash{{{{{$w_5$}%
}}}}}}
\put(2821,-1280){\makebox(0,0)[lb]{\smash{{{{{$w_4$}%
}}}}}}
\put(5211,-1253){\makebox(0,0)[lb]{\smash{{{{{$w_2$}%
}}}}}}
\put(4002, 98){\makebox(0,0)[lb]{\smash{{{{$z_1$}%
}}}}}
\end{picture}%
\caption{The part of the seam between two markings} 
\label{spart}
\end{figure}

We leave it to the reader to check that these maps are continuous with
respect to the topology on $\ol{\RR}^{d,e}$ constructed above and
(using the gluing construction) a topological fibration with interval
fibers.  By induction and the fact that the moduli space
$\ol{\RR}^{1,0}$ is a point, $\ol{\RR}^{d,e}$ is a topological disk.
In particular, $\ol{\RR}^{d,e} $ is compact and Hausdorff.

\begin{proposition} \label{isode} For any positive integers $d,e$,
  there exists an isomorphism of decomposed spaces from
  $\ol{\RR}^{d,e}$ to $\K^{d,e}$. \end{proposition}

\begin{proof} In the base cases $(d,e) \in \{ (1,0), (0,1) \}$ both
  moduli spaces are points and so the result holds.  By the
  combinatorial description of $\K^{d,e}$ and induction, it suffices
  to show that the compactified moduli space $\ol{\RR}^{d,e}$ is the
  cone on its boundary.  But since $\ol{\RR}^{d,e}$ is a topological
  disk, it is homeomorphic to the cone on its boundary, a sphere.
\end{proof}

In order to define the natural transformations, we count quilted
surfaces with strip-like ends for a family constructed from the space
of quilted disks with interior and exterior markings constructed in
the previous section.  For analytical reasons, this requires replacing
the family of quilted disks in the previous section with one in which
degeneration is given by neck-stretching:

\begin{proposition} 
\label{Ade} 
{\rm (Existence of families of quilted surfaces with strip-like ends
  over the bimultiplihedra)} For any positive integers $d,e$, there
exists a collection of families of quilted surfaces $\ol{\SS}^{d,e}$
with strip-like ends over $\ol{\RR}^{d,e}$, with the property that the
restriction $\SS^{d,e}_\Gamma$ of the family to a stratum
$\RR_\Gamma^{d,e}$ isomorphic to a product of multiplihedra,
biassociahedra, and associahedra is a product of the corresponding
families of surfaces and quilted surfaces with strip-like ends, and
collar neighborhoods of $\SS^{d,e}_\Gamma$ are given by gluing along
strip-like ends.
\end{proposition} 

\begin{proof} 
By induction using Theorem \ref{A}, starting from the case of
three-marked disk where we choose a genus zero surface with strip-like
ends, using the already constructed families of surfaces with
strip-like ends in Proposition \ref{Ad}.
\end{proof} 
  
\subsection{Natural transformations for cocycles} 

In this section we construct pseudoholomorphic quilts with interior seam
conditions on various Lagrangian correspondences.

\begin{proposition}  \label{Bde} 
{\rm (Existence of families of pseudoholomorphic quilts)} Let $M_0,M_1$ be symplectic backgrounds with the
same monotonicity constant, $d,e$ positive integers,
$\ul{L}^0,\ldots,\ul{L}^d$ admissible generalized Lagrangian branes in
$M_0$, and $L_{01}^0,\ldots, L_{01}^e$ admissible Lagrangian
correspondences from $M_0$ to $M_1$.  For generic choice of
perturbation data the moduli spaces of pseudoholomorphic quilts with
boundary and seam conditions $(\ul{L}^i)_{i=0}^d$ and
$(L_{01}^j)_{j=0}^e$ has
\begin{enumerate} 
\item finite zero dimensional
component $\M^{d,e}_0$ and 
\item the one-dimensional component $\ol{\M}^{d,e}_1$ has boundary
  equal to the union
$$ \partial \ol{\M}^{d,e}_1 = \bigcup_{\Gamma}
{\M}^{d,e}_{\Gamma,1} $$
where either (1) $\Gamma$ is stable so that $\RR_{\Gamma}^{d,e}$ is a
codimension one stratum in $\ol{\RR}^{d,0}$ in which case
${\M}^{d,e}_{\Gamma,1}$ is the product of quilted and unquilted
components corresponding to the vertices of $\Gamma$, or (2) $\Gamma$
is unstable and corresponds to bubbling off a Floer trajectory.
\end{enumerate} 
\end{proposition} 

\begin{proof} The result follows by applying Theorem \ref{B} recursively
to the stratified space $\ol{\RR}^{d,e}$ constructed in Theorem
\ref{collarde}.  The perturbation over the boundary of $\SS^{d,e}$ is
the product of those for the lower-dimensional moduli spaces.
\end{proof}  

From Theorem \ref{C} and the family of quilted surfaces over the
moduli space of quilted disks with boundary and seam markings we
obtain cochain-level family quilt invariants giving natural
transformations:

\begin{definition} \label{nattransform} {\rm (Pre-natural
    transformations for quilted cochains)} Let $M_0,M_1$ be symplectic
  backgrounds with the same monotonicity constant, $d,e$ positive
  integers, $\ul{L}^0,\ldots,\ul{L}^d$ admissible generalized
  Lagrangian branes in $M_0$, and $L_{01}^0,\ldots, L_{01}^e$
  admissible Lagrangian correspondences from $M_0$ to $M_1$ equipped
  with widths $\delta^0,\ldots, \delta^e$.  To save space in the
  notation we write $\Phi(L_{01}^k) := \Phi(L_{01}^k, \delta_0^k)$.
  Given a sequence of homogeneous elements
  $\alpha_j \in CF(L_{01}^{j-1},L_{01}^j), j = 1,\ldots, e $
define
$$\TT^e(\alpha_1,\ldots,\alpha_e) \in \Hom(\Phi(L_{01}^0),
\Phi(L_{01}^e)) $$
as follows: For intersection points $x_1,\ldots,x_d$ of
$\ul{L}^0,\ldots,\ul{L}^d$ set
$$ (\TT^e(\alpha_1,\ldots,\alpha_e))^d(\bra{x_1},\ldots,\bra{x_d}) =
\Phi_{\SS^{d,e}}(\bra{x_1},\ldots,\bra{x_d},
\alpha_1,\ldots,\alpha_e) (-1)^{\heartsuit + \Box} $$
where $ \Box = \sum_{i=1}^e i | \alpha_i | .$
\end{definition} 

\begin{theorem} 
\label{firstfunc}
{\rm (Categorification functor, first version)} Let $M_0,M_1$ be
symplectic backgrounds with the same monotonicity constant. The maps
$$L_{01}
\mapsto \Phi(L_{01}), (\alpha_1,\ldots,\alpha_e) \mapsto
\TT^e(\alpha_1,\ldots,\alpha_e) $$ 
define an \ainfty functor $\Fuk(M_0^- \times M_1) \to
\Fun(\GFuk(M_0),\GFuk(M_1))$.
\end{theorem} 

\begin{proof}  
We show the \ainfty functor axiom \eqref{faxiom}.  The axiom
follows from a signed count of the boundary components of the
one-dimensional component of the moduli space $\M^{d,e}$ in
Proposition \ref{Bde}.  The boundary components are of three
combinatorial types:
\begin{enumerate}
\item {\rm (Quilted sphere bubbles)} 
Facets where some subset of the
  markings $w_1,\ldots,w_e$ on the interior circle have bubbled off
  onto a quilted sphere with values in $M_0,M_1$;
\item {\rm (Quilted disk bubbles)} Facets corresponding to partitions
  of the interior and exterior markings, corresponding to quilted disk
  bubbles;
\item {\rm (Disk bubbles)}
Facets where some subset of the
  markings $z_1,\ldots,z_d$ on the boundary have bubbled off
  onto a quilted sphere with values in $M_0$; 
\item {\rm (Trajectory bubbling)} Bubbling off trajectories at the
  interior or exterior markings.
\end{enumerate}  
See Figure \ref{m12} for the case of two interior markings; the facet
on the left represents the limit when the two interior marked points
come together.  Counting boundaries of the first type gives an
expression
\begin{multline} 
 \sum_{i,j} \pm \TT^{e-j+1}(\alpha_1,\ldots,\alpha_i,
 \\ \mu^j_{\Fuk(M_0^- \times M_1)}(\alpha_{i+1},\ldots,\alpha_{i+j}),
 \alpha_{i+j+1},\ldots,\alpha_e)(\bra{x_1},\ldots,\bra{x_d}) .\end{multline}
The second type of boundary component contributes a sum of terms of
the form
\begin{multline}
 \sum \pm \mu^{m}_{\GFuk(M_1)} (\Phi(L_{01}^0))^{i_1}(\ldots) \ldots
 (\Phi(L_{01}^0))^{i_{j_1 -1}}(\ldots) \TT^{i_{j_1},k_1} (\ldots) \\
(\Phi(L_{01}^1))^{i_{j_1+1}}(\ldots) \ldots
 (\Phi(L_{01}^1)^{i_{j_2-1}}(\ldots)
 \TT^{i_{j_2},k_2}(\ldots) \\ 
\ldots \TT^{i_{j_s},k_s}(\ldots) \Phi(L_{01}^{e})^{i_{j_s + 1}}(\ldots) \ldots (\Phi(L_{01}^{e})^{i_m}(\ldots))
\end{multline} 
where each $\ldots$ is an expression in the $\alpha_i$ and
$\bra{x_j}$'s, $r$ is the number of bubbles, $s$ is the number of
bubbles with interior markings, $i_l$ represents the number of
exterior markings on the bubbles, $k_l$ represents the number of
interior markings on bubbles with interior markings, and $j_1,\ldots,
j_s$ are the indices of the bubbles with interior markings.  The third
and fourth type of boundary are similar to those considered before and
will not be discussed.  Combining this with \eqref{T2T1} and
\eqref{mu1} proves the \ainfty functor axiom up to sign.
\begin{figure} 
\includegraphics[width=3in,height=3in]{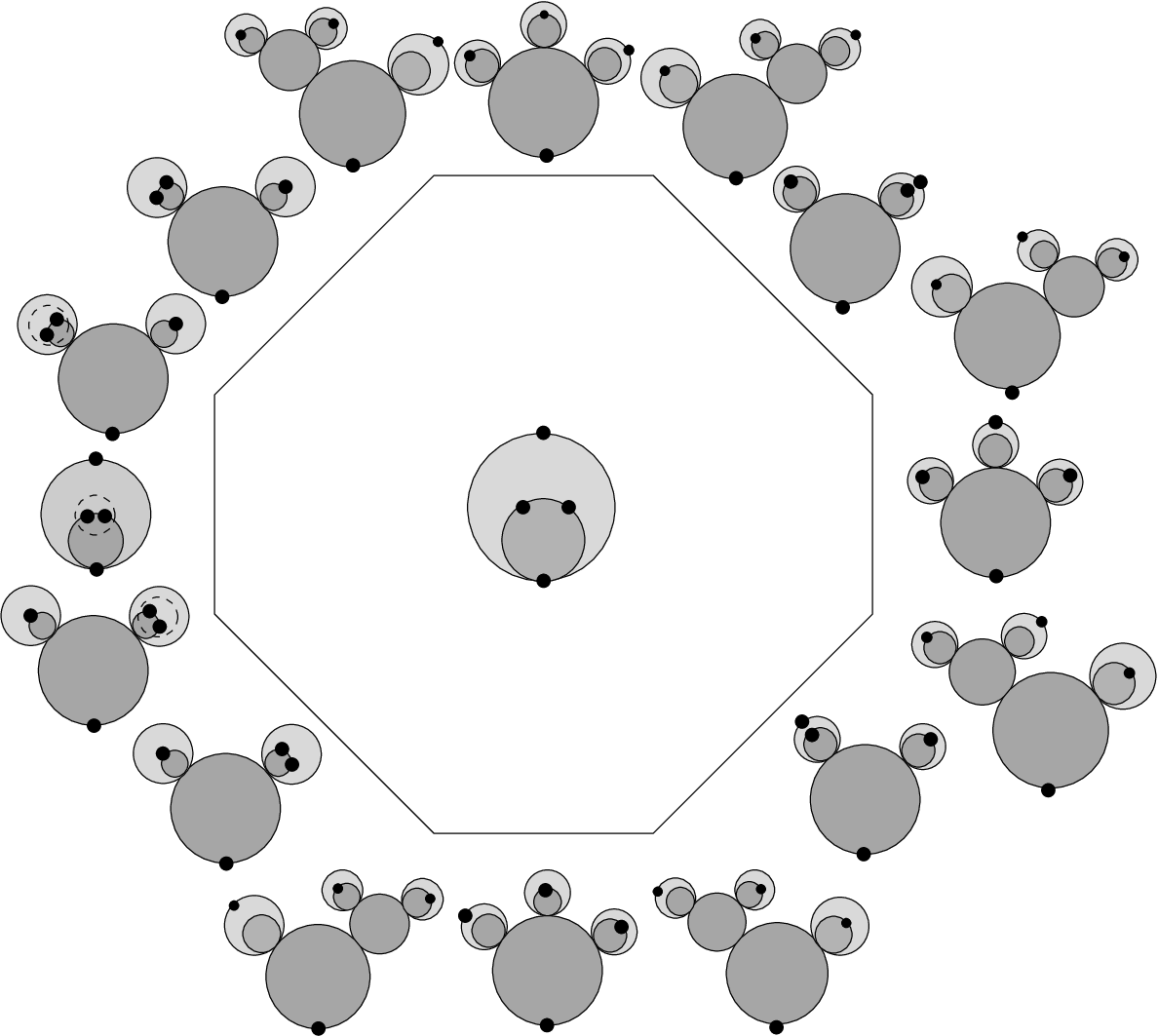}
\caption{The moduli space $\ol{\RR}^{1,2}$}
\label{m12}
\end{figure}  
It remains to check the signs.  The sign for the inclusion of a facet
equal to the image of an embedding
$ \prod_i {\RR}^{d_i,e_i} \times {\RR}^{f,0} \to \ol{\RR}^{d,e} $ is
$(-1)^{\sum_{i=1}^m (d_i - 1 + e_i)i + \sum_{i < j}(d_i - 1)(e_i)}$.
The degree of $\TT^j(\alpha_1,\ldots,\alpha_j)$ is
$ 1 - \sum_{i = 1}^j ( 1 - |\alpha_i| ) .$ The signs appearing in
\eqref{T2T1} and in the higher compositions are given by sums
$ \sum_{i<k,j,l} (|x_i^j| - 1) ( |y_k^l| - 1) $
where $x_i^j,y_k^l$ are intersection points corresponding to the inner
and outer markings respectively on on the $i$-th resp. $k$-th bubble.
The terms of the form $|x_i^j| |y_k^l|$ are accounted for by Koszul
signs.  Combining with the two occurrences of $\heartsuit$, this gives
the signs claimed in \eqref{faxiom}.
\end{proof}  

\begin{remark} \label{units} {\rm (Behavior of units under categorification)}  
The functor of Theorem \ref{mainfunc} is cohomologically unital in the
sense that the associated functor 
$$H(\Fuk(M_0^- \times M_1)) \to \Fun(H \GFuk(M_0), H \GFuk(M_1))$$ 
is unital, by the results of the cohomology level discussion in
\cite{we:co}.
\end{remark} 

To complete the proof of Theorem \ref{mainfunc}, it remains to extend
the definition of natural transformations above to the case of
generalized Lagrangian correspondences from $M_0$ to $M_1$:

\begin{remark} {\rm (Natural transformations for cochains for generalized correspondences)}  
\label{completemainfunc}
Suppose a sequence of generalized Lagrangian correspondences
$\ul{L}_{01}^j = L_{01}^{j,1},\ldots, L_{01}^{j,n_j}$ of length $n_j$
is given.  Construct (inductively on the strata) a bundle
$\ol{\SS}^{d,e} \to \ol{\RR}^{d,e}$ whose fiber at $r \in
\ol{\RR}^{d,e}$ is a quilted surface with $d$ exterior and $e$
interior quilted strip-like ends, and whose combinatorial type is that
of $r$ except that the $j$-th interior segment has been replaced by a
collection of strips of length $n_j - 1$.  Counting pairs $(r,\ul{u})$
consisting of a point $r \in \ol{\RR}^{d,e}$ together with a
pseudoholomorphic quilt $\ul{u}: \ol{\SS}^{d,e}_r \to \ul{M}$ with the given
Lagrangian boundary and seam conditions defines the natural
transformations.  The discussion of signs and gradings is similar to
that for Theorem \ref{firstfunc}.
\end{remark}

\label{indep}  As an application of Theorem \ref{mainfunc}, we show that the functors
for Lagrangian correspondences constructed in the previous section are
independent of all choices up to quasiisomorphism.

\begin{theorem} \label{indepthm}
{\rm (Independence of the functors up to quasiisomorphism)} Let
$M_0,M_1$ be symplectic backgrounds with the same monotonicity
constant, and let $\ul{L}_{01}$ be an object in $\GFuk(M_0,M_1)$.  The
functor $\Phi(L_{01})$ constructed in Section \ref{functors} is
independent up to quasiisomorphism of all choices (the choice of
family of quilts, that is, holomorphic structures on the fibers of
$\ol{\SS}^{d,e} \to \ol{\RR}^{d,e}$, and the perturbation data.)
\end{theorem} 

\begin{proof}  Suppose two such choices are given, with corresponding functors $\Phi(L_{01}),\Phi'(L_{01})$.  The Floer cocycle
$\alpha \in CF(L_{01},L_{01})$ corresponding to the identity in
  $HF(L_{01},L_{01})$ defines a natural transformation $\beta$ from
  $\Phi(L_{01})$ to $\Phi'(L_{01})$.  Its transpose defines a natural
  transformation from $\Phi'(L_{01})$ to $\Phi(L_{01})$.  The
  composition of the two natural transformations is given by the
  product of $\alpha$ with $\beta$ under the composition map $\mu^2$
  in $\GFuk(M_0,M_1)$.  By Theorem \ref{mainfunc}, the composition of
  natural transformations is the identity transformation.
\end{proof} 

\begin{proposition} \label{diagonal} 
{\rm (Functor for the diagonal correspondence)} Let $M$ be a compact
monotone symplectic background.  Suppose that $M$ is spin and $\Delta
\subset M^- \times M$ is the diagonal equipped with the relative spin
structure corresponding to a spin structure on $M$.  Then
$\Phi(\Delta)$ is quasiisomorphic to the identity functor from
$\GFuk(M)$ to $\GFuk(M)$.
\end{proposition}

\begin{proof}  The correspondence 
$\Delta$ is quasiisomorphic to $\emptyset$ in $\GFuk(M,M)$, with
  isomorphism given by the cohomological unit in $CF(\Delta)$.  By
  Proposition \ref{emptyset} and Theorem \ref{mainfunc} $\Phi(\Delta)$
  is quasiisomorphic to the identity functor in
  $\Fun(\GFuk(M),\GFuk(M))$.
\end{proof}  

\section{Algebraic and geometric composition}
\label{compose}

In this section we study the composition of \ainfty functors for
Lagrangian correspondences.  In particular, we prove Theorem
\ref{maincompose} on the homotopy equivalence of the \ainfty functor
for a geometric composition and the \ainfty composition of the
corresponding \ainfty functors.

\subsection{Quilted disks with multiple seams}

The proof of the composition result \ref{maincompose} depends on
generalizations of the multiplihedra which feature multiple interior
circles.  For $d\geq 1$, let $\K^{d,0,0}$ denote the cell complex
whose cells correspond to parenthesized expressions in formal
variables $a_1,\ldots,a_d$ and operations $h_1,h_2$ with $h_2$ always
following $h_1$.  More precisely, these expressions correspond to {\em
  bicolored trees} which are trees equipped with {\em two} types of
colored vertices:

\begin{definition} \label{bitree} {\rm (Bicolored trees)}  
A {\em bicolored, rooted tree with $d$ leaves} consists of data
$(\Edge(\Gamma),\on{Vert}(\Gamma),\Edge_\infty(\Gamma),\Ver^{(1)}(\Gamma),\on{Vert}^{(2)}(\Gamma))$
where:
\begin{enumerate}
\item {\rm (Tree)} $\Gamma=(\Edge(\Gamma),\Ver(\Gamma))$ is a tree
  with vertices $\Ver(\Gamma)$, a collection of (possibly
  semi-infinite) edges $\Edge(\Gamma)$, and labelling of the
  semi-infinite edges $\Edge_\infty(\Gamma)$ by
  $\{e_0,e_1,\ldots,e_d\}$.  We call $e_0$ the {\em root edge} and
  $e_1, \ldots, e_d$ the {\em leaves}. 
\item \label{colver} {\rm (Colored vertices)} There are distinguished subsets
  $\Ver^{(1)}(\Gamma) \subset \Ver(\Gamma)$
  resp. $\Ver^{(2)}(\Gamma)\subset \Ver(\Gamma)$ of {\em
    vertices of color 1 resp. color 2}, such that 
\begin{enumerate} 
\item any geodesic from a
  leaf to the root passes through exactly one vertex in
  $\Ver^{(1)}(\Gamma)$ and exactly one vertex in
  $\Ver^{(2)}(\Gamma)$;
\item \label{allornothing} either $\Ver^{(1)}(\Gamma) =
  \Ver^{(2)}(\Gamma)$ or $\Ver^{(1)}(\Gamma) \cap \Ver^{(2)}(\Gamma) =
  \emptyset$;
\item 
in the case that $\Ver^{(1)}(\Gamma) \cap \Ver^{(2)}(\Gamma) = \emptyset$, 
the vertices $\Ver^{(1)}(\Gamma)$ are closer to
  the root edge.
\end{enumerate}
\end{enumerate}

Such a bicolored tree $\Gamma$ is {\em stable} if the valency of every
vertex $v\in V$ is 3 or more for $v \notin \Ver^{(1)} \cup
\Ver^{(2)}$, and $2$ or more otherwise.
\end{definition}

An example of a bicolored tree with one vertex in $\Ver^{(1)}$ shaded darkly
and four \label{fivetofour} vertices in $\Ver^{(2)}$ shaded lightly is shown in Figure \ref{bicol};
the remaining vertices are filled. 

\begin{figure}[ht]
\begin{picture}(0,0)%
\includegraphics{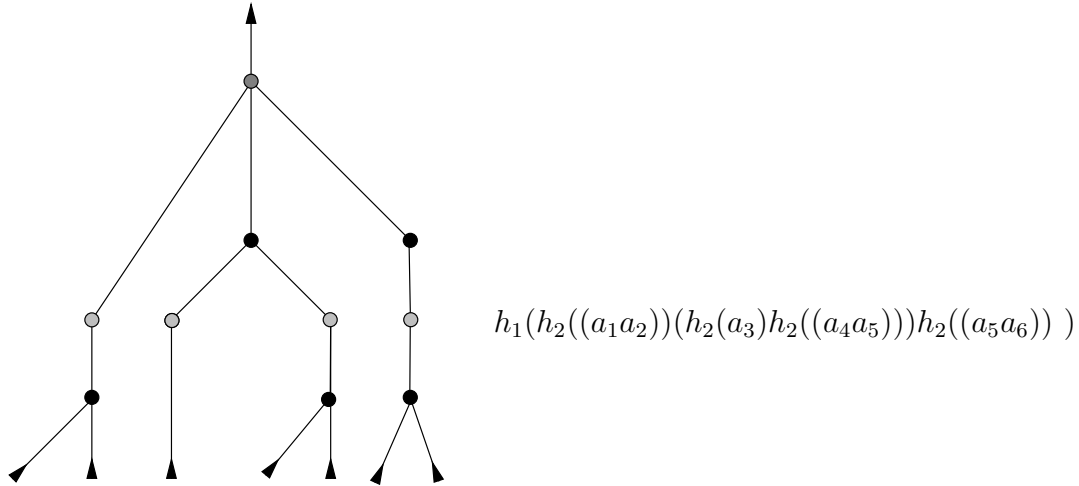}%
\end{picture}%
\setlength{\unitlength}{3947sp}%
\begingroup\makeatletter\ifx\SetFigFont\undefined%
\gdef\SetFigFont#1#2#3#4#5{%
  \reset@font\fontsize{#1}{#2pt}%
  \fontfamily{#3}\fontseries{#4}\fontshape{#5}%
  \selectfont}%
\fi\endgroup%
\begin{picture}(8608,3057)(1770,-3406)
\put(4825,-2435){\makebox(0,0)[lb]{\smash{{\SetFigFont{12}{14.4}{\rmdefault}{\mddefault}{\updefault}{\color[rgb]{0,0,0}$h_1(h_2((a_1a_2))( h_2(a_3) h_2((a_4a_5)))  h_2 ((a_5 a_6)   ) $                 )}%
}}}}
\end{picture}%

\caption{A stable bicolored tree and the associated algebraic expression}
\label{bicol} 
\end{figure}

Given such a bicolored tree, the corresponding expression is obtained
by using the part of the tree below the second type of colored vertex
to determine the parenthesization of the expressions $a_1,\ldots,
a_d$, the part of the tree in between the first and colored vertices
to determine the parenthesizations of the expressions $h_2( \cdot)$,
the part of the tree above the first colored vertices to determine the
parenthesizations of the expressions $h_1(\cdot)$.

A space associated to markings and two operations is constructed
inductively from those with fewer markings.  For each bicolored tree
$\Gamma$ and each vertex $v$ of $\Gamma$, let $\K^v$ denote the
corresponding associahedron, multiplihedron, or lower dimensional
space depending on the valence and whether the vertex is once or twice
colored.  Define \label{KGamma}
$$ \K_\Gamma = \prod_{v \in \Ver(\Gamma)} \K^v $$
and 
$ \partial \K^{d,0,0} = \cup_{\Gamma} \K_\Gamma / \sim $
where $\sim$ are the equivalences obtained by boundary inclusions.
The space is then the cone
$ \K^{d,0,0} = \Cone(\partial \K^{d,0,0} ) .$ For example, the cell
complex $\K^{2,0,0}$ is a pentagon with vertices labelled by the
expressions
$$h_1(h_2(a_1)h_2(a_2)), h_1(h_2(a_1))h_1(h_2(a_2)),
(h_1h_2)(a_1)(h_1h_2)(a_2), h_1(h_2(a_1a_2)), (h_1h_2)(a_1a_2). $$
Similar spaces also appear in Batanin \cite[end of Section 8]{bat:sym}
and \cite[Example 8.1]{bat:notes}. \label{batnotes}

The generalized multiplihedron above has a realization as a moduli
space of biquilted disks, described as follows.

\begin{definition}\label{2qdef} {\rm (Biquilted disks)}  For an integer
$d \ge 1$ a {\em biquilted disk with $d+1$ markings} is a tuple
  $(D,C_1, C_2, z_0, \ldots, z_d)$ where
\begin{enumerate}
\item $D$ is a holomorphic disk; 
\item $(z_0, \ldots, z_d)$ is a tuple of distinct points in $\partial
  D$ whose cyclic order is compatible with the orientation of
  $\partial D$;
\item $C_1$ and $C_2$ are nested circles in $D$ with
$$0 < \radius(C_1) < \radius(C_2) < \radius(D)$$ 
and unique intersection point 
$$C_1\cap \partial D = \{z_0\} = C_2 \cap \partial D .$$
\end{enumerate}
An {\em isomorphism} of biquilted disks $(D,C_1, C_2, z_0, \ldots,
z_d), (D',C_1', C_2', z_0', \ldots, z_d')$ is a holomorphic
isomorphism of the disks 
$$\psi: D \to D' ,\quad \psi(C_j) = (C_j'), \ j = 1,2, \quad \psi(z_k)
= z_k', \ k = 0,\ldots, d .$$
Let $\RR^{d,0,0}$ denote the set of isomorphism classes of 
biquilted disks with $d+1$ markings. 
\end{definition}

The moduli space of biquilted disks has a compactification which
includes nodal biquilted disks.

\begin{definition} \label{twicequilted} {\rm (Stable biquilted disks)}
  A {\em nodal biquilted disk} with combinatorial type a bicolored
  tree $\Gamma$ is a collection of unquilted, quilted, and biquilted
  disks corresponding to the vertices of $\Gamma$ attached at disjoint
  collections of pairs of points called {\em nodes}, together with a
  collection of {\em markings} on the boundary in cyclic order
  disjoint from the nodes, with the properties that
\begin{enumerate}  
\item a single inner circle circle appears in the component $D_v$
  corresponding to $v \in \Ver(\Gamma)$ if and only if $ (\Ver^{(1)}(\Gamma) \cap
  \Ver^{(2)}(\Gamma)) = \emptyset$ and $v$ lies in
  $(\Ver^{(1)}(\Gamma) \cup \Ver^{(2)}(\Gamma))$;
\item exactly two inner circles appear in the component $D_v$
  corresponding to $v \in \Ver(\Gamma)$ if and only if $v$ lies in
$ (\Ver^{(1)}(\Gamma)
  \cap \Ver^{(2)}(\Gamma)) $;
\item \label{ratiopart} for any biquilted disk component $D_v, v \in
  \Ver^{(2)}(\Gamma)$ the ratio
$$\radius(C_2)/\radius(C_1) \in [1,\infty)$$ 
of radii of the two inner circles $C_1,C_2$ is independent of the
choice of component
\end{enumerate} 
and satisfying the {\em root marking property} that the {\em root
  marking} $z_0$ is, if it lies on a quilted component, the
intersection of the inner circle(s) with the boundary and the {\em
  monotonicity property} that on any non-self-crossing path of
components from the root marking $z_0$ to the component containing a
marking $z_i$, there is exactly one component with an inner circle and
one component with an outer circle.  A nodal biquilted disk is {\rm
  stable} if the corresponding bicolored tree is stable, that is, each
disk with at least one resp. no interior circles has at least two
resp. three special points.  Denote by $\ol{\RR}^{d,0,0}$ the set of
isomorphism classes of stable nodal biquilted disks.
\end{definition}   

The topology on the moduli space of stable nodal biquilted disks is
defined in a similar way as that for the moduli space of stable nodal
quilted disks with interior and seam markings in Section
\ref{biassoc}.  A sequence of biquilted marked disks is said to
converge to a stable limit if, for each component of the limit, there
exist a family of reparametrizations satisfying certain properties
that we will not detail here.  The moduli space $\ol{\RR}^{1,0,0}$ is
a topological interval while the moduli space $\ol{\RR}^{2,0,0}$ is
shown in Figure \ref{pentagon3}.

  \begin{figure}[ht]
\includegraphics[height=2in]{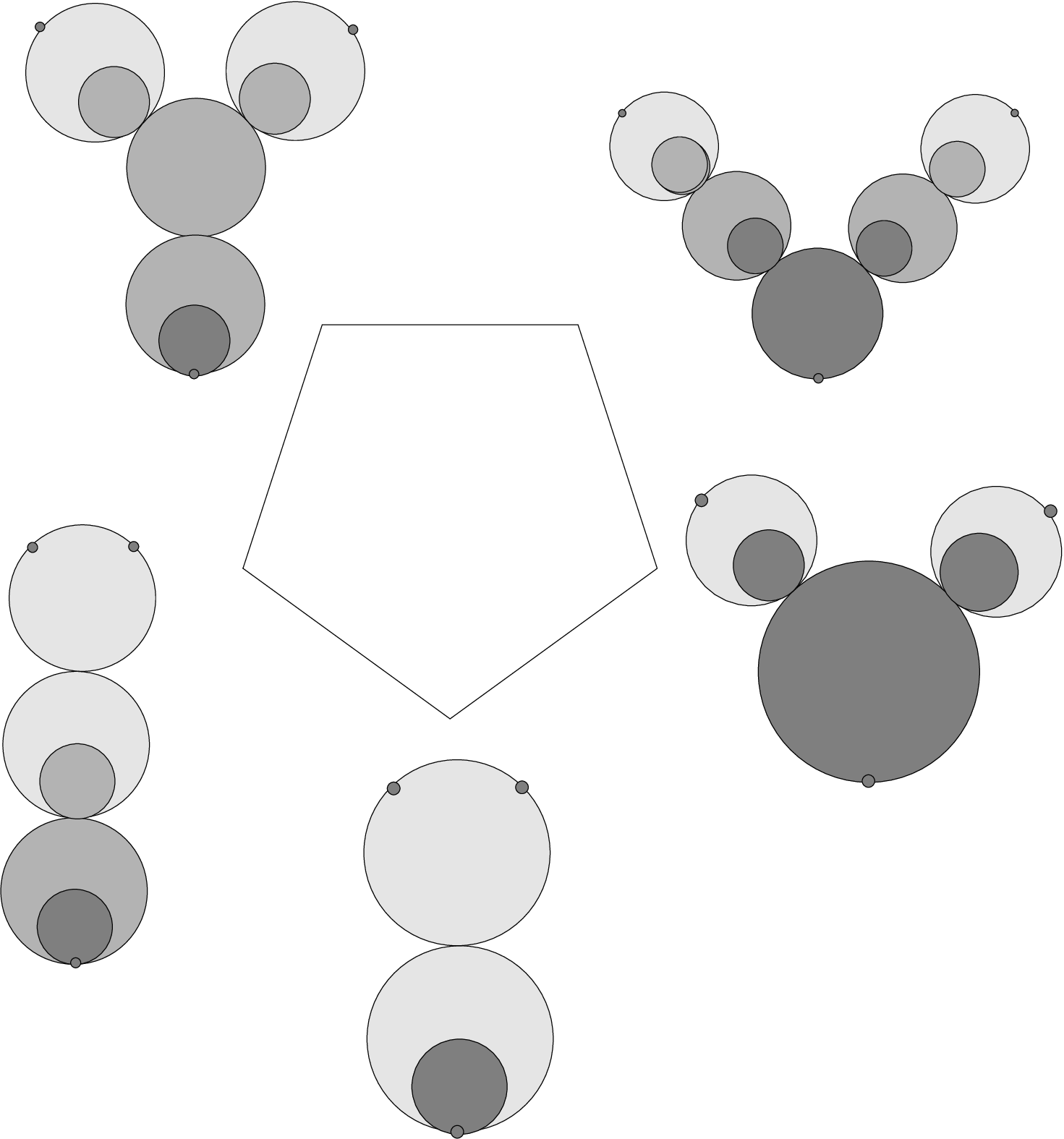}
\caption{Vertices of $\ol{\RR}^{2,0,0}$}
\label{pentagon3}
\end{figure}

The local structure of the moduli space of stable nodal biquilted
disks may be described in terms of gluing parameters taking values in
the positive real part of a toric variety.
  
\begin{definition} {\rm (Balanced gluing parameters)}  
Let 
$$\Gamma=(\Edge_{<
  \infty}(\Gamma),\Ver(\Gamma),\Edge_\infty(\Gamma),\Ver^{(1)}(\Gamma),\Ver^{(2)}(\Gamma))$$
be a combinatorial type of a nodal biquilted disk with $d+1$ markings.
The set of {\em balanced gluing parameters} for $\Gamma$ is the subset
$Z_\Gamma$ of functions 
$$\delta: \Edge_{< \infty}(\Gamma) \to [0,\infty)$$
satisfying the following relations:
\begin{itemize}
\item[] for each $j = 1,2$ and for each pair of vertices $v,v'$ in
  $\Ver^{(j)}(\Gamma)$, $\gamma$ the shortest path in $\Gamma$ from
  $v$ to $v'$, the relation 
$$1= \prod_{e \in \gamma} \delta(e)^{ \pm 1}$$ 
holds as in Definition \ref{balanced}.
\end{itemize}
\end{definition}

The various strata are glued together by means of local charts.

\begin{theorem}
\label{glueRd00} 
{\rm (Existence of Compatible Tubular Neighborhoods)} For any integer $d \ge 1$ and any combinatorial
type $\Gamma$ of $d+1$-marked biquilted disks there exists a
neighborhood $U_\Gamma$ of $0$ in $Z_\Gamma$ and a collar neighborhood
$$ G_\Gamma: \RR^{d,0,0}_\Gamma \times U_\Gamma \to
\ol{\RR}^{d,0,0} $$
of $\RR^{d,0,0}_\Gamma$ mapping onto onto an open neighborhood of
$\RR_\Gamma^{d,0,0}$ in $\ol{\RR}^{d,0,0}$ satisfying the following
compatibility property: Suppose that $\RR^{d,0,0}_{\Gamma'}$ is
contained in the closure of $\RR^{d,0,0}_{\Gamma}$ and the local
coordinates on $\RR^{d,0,0}_\Gamma$ are induced via the gluing
construction from those on $\RR^{d,0,0}_{\Gamma'}$.  Then the diagram
$$ 
\begin{diagram}
\node{\RR_{\Gamma'}^{d,0,0} \times U_{\Gamma'}} \arrow[2]{e}\arrow{se}
\node{} \node{\ol{\RR}_{\Gamma}^{d,0,0} \times U_{\Gamma} } \arrow{sw}
\\ \node{} \node{\ol{\RR}^{d,0,0}} \node{}
\end{diagram} 
$$
commutes. 
\end{theorem} 

\begin{proof}   The proof uses the same gluing construction 
as in Theorem \ref{glueRd0}, and is left to the reader.
\end{proof} 

In this sense, the stratified space of biquilted disks is {\em
  equipped with quilt data} as in Definition \ref{quiltdata}: each
stratum comes with a collar neighborhood described by gluing
parameters compatible with the lower dimensional strata.  As before,
there are forgetful morphisms
$$ f_i: \ol{\RR}^{d,0,0} \to \ol{\RR}^{d-1,0,0} $$
giving $\ol{\RR}^{d,0,0}$ the structure of a topological fibration
with interval fibers as in \eqref{univcurve}.  By induction and the
fact that $\ol{\RR}^{1,0,0}$ is a topological interval, which may be
checked explicitly, the space $\ol{\RR}^{d,0,0}$ is a topological
disk, and in particular a cone on its boundary.  By another induction,
there is a homeomorphism from $\ol{\RR}^{d,0,0}$ to $\K^{d,0,0}$ which
respects the combinatorial structure.

\begin{lemma}  \label{ratio} {\rm (Ratio of radii map)}  For any integer
$d \ge 1$ there is a continuous map $\rho: \ol{\RR}^{d,0,0} \to
  [0,\infty]$ given on the open stratum by
\begin{equation} \label{rhomap} 
\rho([ D,C_2,C_1,z_0,\ldots,z_d ]) = \radius(C_2)/\radius(C_1) -1
  .\end{equation} 
\end{lemma} 

\begin{proof} The map  $\rho$ is given by the forgetful morphism 
$\ol{\RR}^{d,0,0} \to \ol{\RR}^{1,0,0} \cong [0,\infty]$ defined by 
  forgetting all but the $0$-th marking and recursively collapsing
  unstable components, starting with the components furthest away from
  the $0$-th marking.
\end{proof} 

The following description of the boundary of the moduli space of
twice-quilted disks is immediate from Definition \ref{twicequilted}:

\begin{proposition} {\rm (Facets of moduli of biquilted disks)}  
  Suppose that a combinatorial type $\Gamma$ of $d$-marked biquilted
  disks contains $k+1$ vertices with vertices $v_1,\ldots, v_k$
  corresponding to biquilted disks, and a vertex $w$ corresponding to
  an unquilted disk, and corresponds to a facet.  Then there exists a
  homeomorphism
  \begin{equation} \label{fiberprod} \RR^{d,0,0}_\Gamma \cong
    (\RR^{|v_1|,0,0}\times_{[0,\infty]} \ldots \times_{[0,\infty]} \RR^{|v_k|,0,0})
    {\times} \RR^{|w|}
\end{equation} 
where the fiber product $ \RR^{v_1,0,0}\times_{[0,\infty]} \ldots \times_{[0,\infty]}
\RR^{v_k,0,0} $ is such that the functions $\rho$ as defined in
\eqref{rhomap}, are all equal on all the components.
\end{proposition} 

\begin{proposition}  \label{d00facets} {\rm (Classification of facets of the bimultiplihedron)}   The facets of $\ol{\RR}^{d,0,0} \cong \K^{d,0,0}$ are of the following
types:
\begin{enumerate}
\item {\rm (Once-quilted bubbles)} a collection of $k$ quilted disks
  all with seam $C_2$ attached to a $k+1$-marked quilted disk with
  seam $C_1$, corresponding to an expression given by
$$h_1( h_2(\ul{a}_{1,1}), \ldots, h_2( \ul{a}_{1,l_1} ) ) \ldots h_1(
  h_2(\ul{a}_{i_r,1}), \ldots, h_2(\ul{a}_{i_r,l_r})) $$
where $\ul{a}_{1,1} \cup \ldots \cup \ul{a}_{i_r,l_r} = (a_1,\ldots,
a_d)$ is an ordered double partition of the inputs; in which case the
facet is the image of an embedding
$$(\K^{i_1,0} \times \ldots \times \K^{i_r,0}) \times \K^{r,0} \to
  \K^{d,0,0} $$
with $i_1+\ldots + i_r = d$;
\item {\rm (Unquilted bubbles)} an unquilted disk attached to a
  biquilted disk, corresponding to an expression given by
$$h_1h_2(a_1,\ldots,a_{i- 1}(a_i,\ldots,a_{i+j-1}),\ldots,a_d)$$ 
in which case the facet is the image of an embedding
$$\K^{d_1} \times \K^{d_2,0,0} \to \K^{d_1 + d_2,0,0} ;$$
\item {\rm (Biquilted bubbles)} a collection of $k$ biquilted disks
  with seams $C_1$ and $C_2$ and the same ratio of radii, attached to
  a single unquilted disk with $k+1$ markings; corresponding to an
  expression given by
$$h_1h_2(a_1 \ldots
  a_{i_1}) \ldots h_1h_2(a_{d - i_k + 1},\ldots, a_d) ;$$ 
in which case the facet is the image of an embedding
$$\K^{i_1,0,0} \times_{[0,\infty]} \ldots \times_{[0,\infty]} \K^{i_k,0,0}\times \K^j
\mapsto \K^{d,0,0}$$
for some partition $i_1 + \ldots + i_k = d$;
\item {\rm (Seams coming together)} a quilted disk with a single seam
  $C_1 = C_2$, 
corresponding to the expression $(h_1h_2)(a_1,\ldots,a_d)$.  In this
case the facet is the image of an embedding $\K^{d,0} \to \K^{d,0,0}$.
\end{enumerate}
\end{proposition} 

\begin{proof} By Theorem \ref{glueRd00}, the codimension of a
  combinatorial type $\Gamma$ is the number of edges modulo the number
  of independent relations generated by the balanced condition.
  Suppose that the root component is not quilted.  If there are
  twice-quilted components, then since the \label{balancedpage} number
  of independent balanced relations is at least the number of quilted
  components minus one, and the number of gluing parameters is at
  least the number of quilted components, no other components besides
  the root component and twice-quilted components can occur in the
  codimension one case.  If the root component is once-quilted, then
  the same dimension count implies that all components are
  once-quilted.  Thus codimension one types occur if either there are
  no nodes and the seams must coincide, a single node, or all nodes
  must be attached to root component of the configuration.  The
  Proposition follows.
\end{proof} 

\begin{example} \label{facex} The facets of $\ol{\RR}^{2,0,0}$ are shown in Figure
  \ref{M211facets}.  The five facets are homeomorphic to
  $\RR^{1,0}\times \RR^{1,0} \times \RR^{2,0},
  (\RR^{1,0,0}\times_{[0,\infty]} \RR^{1,0,0} )\times \RR^2,
  \RR^{2,0}, \RR^{2} \times \RR^{1,0,0}, \RR^{2,0} \times \RR^{1,0}$
  respectively.
\end{example} 

\begin{figure}[h]
\includegraphics[height=2.5in]{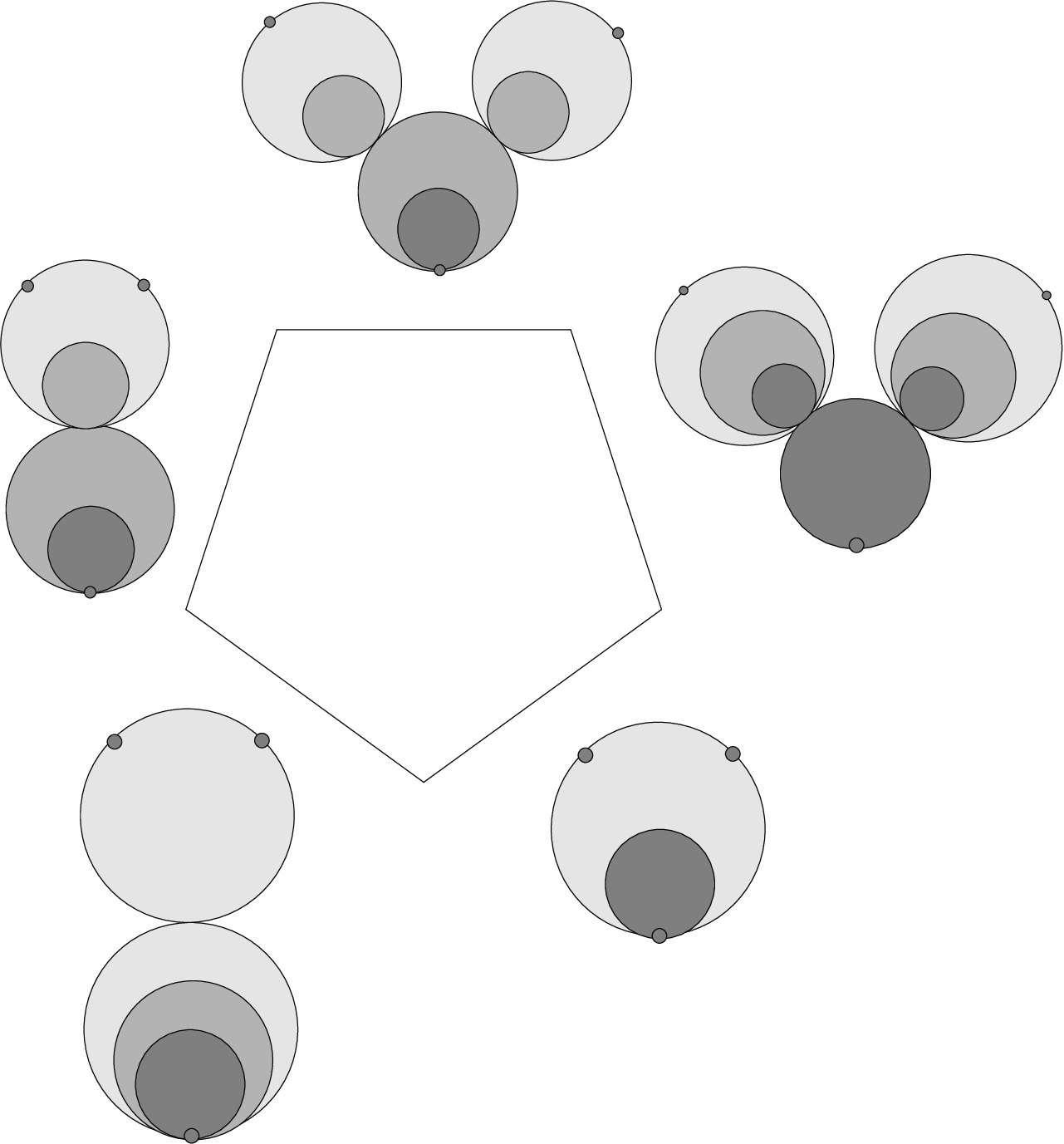}
\caption{Facets of $\ol{\RR}^{2,0,0}$}
\label{M211facets}
\end{figure}

The moduli space of biquilted disks as cut out of a larger moduli
space of {\em free biquilted disks} whose definition is obtained by
dropping the ratio equality \eqref{ratiopart} from Definition
\ref{twicequilted}, see Figure \ref{free}.  More precisely,

\begin{definition} \label{pre}  {\rm (Moduli space without relations)}  
Let $\Gamma$ be a combinatorial type of biquilted disk with $k > 0$
biquilted disks.  Define $\ol{\RR}^{\pre}_\Gamma$ as the product of
moduli spaces for the vertices,
$$ \ol{\RR}_\Gamma^{\pre} = \prod_{v \in \Ver(\Gamma)} \ol{\RR}_v .$$
Let
$$ \rho_\Gamma: \ol{\RR}_\Gamma^{\pre} \to [0,\infty]^{k}, \quad r :=
(r_v)_{v \in \Ver(\Gamma)} \mapsto \rho_\Gamma(r) := 
( \rho(r_v) )_{v \in \Ver^{(1)}(\Gamma) \cap
  \Ver^{(2)}(\Gamma)} $$
be the map derived from the ratios of the radii of the circles for
each biquilted component as in \eqref{rhomap}, if there are biquilted
components, or taking the value $\infty$, if there are no biquilted
components.
\end{definition}

\begin{figure}[ht]
\includegraphics[height=1in]{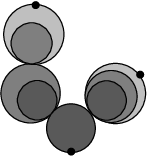}
\caption{A free biquilted disk} 
\label{free} 
\end{figure} 

In terms of the map $\rho_\Gamma$ defined in Definition \ref{pre} we
have
\begin{equation} \label{rhodiag}
 \ol{\RR}_\Gamma = \rho_\Gamma^{-1}(\Delta) \end{equation}
where $\Delta = \{ (x,\ldots, x) | x \in \R \} \subset \R^{k}$ is the
thin diagonal.  The combinatorial type of a free biquilted disk is a
free bicolored tree which is defined in the same way as Definition
\ref{bitree}, but without item \eqref{allornothing}.  If $\Gamma'$ is
a combinatorial type of free biquilted disk that is a refinement of
$\Gamma$ then the subset
$\ol{\RR}_\Gamma^{\pre} \cap \ol{\RR}_{\Gamma'}^{\pre}$ has a
neighborhood in $\ol{\RR}_\Gamma$ given by the gluing construction
\ref{glueRd00}.

\subsection{Transversality for the ratio map} 

In this and the following section we define a homotopy between the
functor for the geometric composition and the algebraic composition by
counting biquilted surfaces with strip-like ends for families of
quilted surfaces parametrized by the polytopes $\ol{\RR}^{d,0,0}$.
Proving transversality of pseudoholomorphic biquilted disks for this
family is more delicate than in all the other cases treated
previously.  This difficulty is due to the fact that in our particular
realization of the homotopy multiplihedron as the moduli space
$\ol{\RR}^{d,0,0}$, some of the boundary strata are {\em fiber
  products} of lower-dimensional strata rather than simply products as
in \eqref{rhodiag}.

In order to construct moduli spaces of expected dimension, we allow
perturbations in the construction of fiber products.  Let
$M_0,M_1,M_2$ be symplectic backgrounds with the same monotonicity
constant and $L_{01} \subset M_0^- \times M_1$ and $L_{12} \subset
M_1^- \times M_2$ admissible Lagrangian correspondences with brane
structures.  Let $\ul{L}^0,\ldots,\ul{L}^d$ be admissible generalized
Lagrangian branes in $M_0$.  Given a marked biquilted disk
$(D,C_1,C_2, z_0,\ldots,z_d)$ we label the inner disk resp. middle
region resp. outer region by $M_2$ resp.  $M_1$ resp. $M_0$, the seams
by $L_{01}$ and $L_{12}$, and the components of the outer boundary by
$\ul{L}^0,\ldots, \ul{L}^d$.  Let $\M^{d,0,0}$ denote the moduli space
of biquilted pseudoholomorphic disks, where the surfaces and perturbation
data are to be defined inductively.

To carry out the induction we introduce the following notation. 

\begin{definition} \label{notations} 
Let $\Gamma$ be a combinatorial type of free biquilted disk with
uncolored root vertex $v_0$.  Denote as follows:
\begin{center} 
\begin{tabular}{l|l} 
$v_1, \ldots, v_k$ & vertices corresponding to outermost quilted or biquilted disk
  components \\ 
$\Gamma_0$ & smallest subtree of $\Gamma$
containing $v_0, v_1, \ldots, v_k$ \\ 
$\gamma_i \subset \Gamma_0$ & non-self-crossing path from $v_i$ to the
root component $v_0$ \\
$|v|$ & the valency of
  $v$ \\ 
$\M_v$ & moduli space of pseudoholomorphic quilts $(r,u)$ 
  parametrized by \\
\ &
$r \in \RR^{|v|-1,0,0}$
  if $v \in \{v_1, \ldots, v_k\}$ is doubly-quilted  \\
\ & $r \in \RR^{|v|-1,0}$
  if $v \in \{v_1, \ldots, v_k\}$ is singly-quilted, and \\
\ & $r \in \RR^{|v|-1}$ otherwise.
\end{tabular} 
\end{center} 
If $k = 0$ we take $\Gamma_0$ to be the smallest subtree containing
all the quilted components.
\end{definition} 

We suppose that the moduli spaces $\M_v$ have been constructed
inductively; the actual construction of $\M_v$ is carried out in
Section \ref{inductive}.  

\begin{definition} \label{delay} {\rm (Delay functions etc.)}  Let
  $\Gamma$ be a combinatorial type of free biquilted disk with
  unquilted root component.  We think of $[0,\infty]$ as a smooth
  manifold with boundary, via identification with a finite closed
  interval.  Any element of $\tau \in \R$ defines an automorphism of
  $[0,\infty]$ fixing $0$ and $\infty$, given by
  $x \mapsto x {\exp(\tau)}$.
\begin{enumerate} 
\item {\rm (Delay functions)} A {\em delay function} for $\Gamma$ is a
  collection of smooth functions depending on $r \in
  \ol{\RR}^{\pre}_\Gamma$ (see Definition \eqref{pre})
$$ \tau_{\Gamma} = \left( \tau_e \in C^\infty(\ol{\RR}^{\pre}_\Gamma)
  \right)_{e\in \Edge(\Gamma_0)}.$$
\item {\rm (Delayed evaluation map)} Letting $\rho_i := \rho(r_{v_i})$
  where $\rho$ is the map of \eqref{rhomap} taking the ratio of radii
 of  circles, the {\em delayed evaluation map} is
\begin{eqnarray}
\rho_{\tau_\Gamma} : \prod\limits_{v\in \Ver(\Gamma)} {\M}_{v}
 & \to &
    [0,\infty]^k \label{delayeval} \\ (r_v, u_v)_{v\in \Ver{\Gamma} }
    & \mapsto & \left( \rho_i \exp\left( \sum_{e\in \gamma_i}
    \tau_e(r)\right) \right)_{i=1,\ldots,k} \nonumber
.\end{eqnarray}
That is, the evaluation map is shifted by the sum of delays along each
path $\gamma_i$ from the root vertex $v_0$ to the vertex $v_i$
corresponding to a biquilted or outer quilted disk component.
\item {\rm (Regular delay functions)} Call $\tau_\Gamma$ {\em
  regular} if the delayed evaluation map $\rho_{\tau_\Gamma}$ is
  transverse to the diagonal $\triangle \subset (0,\infty)^k \subset
  [0,\infty]^k$: \label{There}
$$ \on{Im}( D_{r,u} \rho_{\tau_\Gamma}) \oplus T_{\rho_{\tau_\Gamma}(r,u)}
  \triangle = \R^k, \quad \forall (r,u) \in \prod\limits_{v\in
    \Ver(\Gamma)} {\M}_{v} .$$
\item {\rm (Delayed fiber product)} Given a regular delay function
  $\tau_\Gamma$, we define 
\begin{equation} \label{delayedfiber} 
{\M}_\Gamma := \rho_{\tau_\Gamma}^{-1}(\triangle) . \end{equation}
For a regular delay function $\tau_\Gamma$, the delayed fiber product
has the structure of a smooth manifold, of local dimension
\begin{equation} \label{localdim}
 \dim \M_\Gamma = 1-k + \sum_{v\in \Ver \Gamma} \dim \M_v   \end{equation}
where $k$ is the number of biquilted disk components.

\end{enumerate}
\end{definition}

The delay functions make the fiber product transverse, so that the
zero dimensional moduli spaces behave ``as expected'':

\begin{proposition} \label{mostly0s} {\rm (Only one zero-dimensional bubble for a regular
delay function)} Let $\Gamma$ be a combinatorial type consisting of a
  single unquilted disk indexed by a vertex $v \in \Ver \Gamma$ and $k
  > 0$ biquilted disks indexed by vertices $v_1, \ldots, v_k$.  If
  $\tau_\Gamma$ is regular, then an isolated point in $\M_\Gamma$
  consists of an isolated $k$-marked unquilted disk $(r_v,u_v)$ in
  $\M_v$, together with a tuple of pseudoholomorphic quilted disks
  $(r_{v_i}, \ul{u}_{v_i})$ in $\M_{v_i}, i = 1,\ldots, k$, where
  exactly one of the entries $(r_{v_j}, \ul{u}_{v_j})$ in the tuple
  comes from a zero-dimensional moduli space $\M_{v_j}^0$, and the
  remaining entries $(r_{v_i}, \ul{u}_{v_i}), i \neq j $ come from
  one-dimensional moduli spaces $\M_{v_i}^1$.
\end{proposition} 

\begin{proof}  Note the dimension formula 
\eqref{localdim}.  The regularity condition on $\tau_\Gamma$ implies
that $\dim(\M_v) = 0$ and $\dim(\M_{v_j}) =1$ for all $j$ except for
one $j$ for which $\dim(\M_{v_j}) = 0$.
\end{proof} 

Of course in order to retain compactness the delay functions must be
chosen for the strata in a compatible way, detailed below.  A delay
function for a combinatorial type $\Gamma$ not containing any
biquilted disks by convention assigns to any edge the number zero.
Let $\tau^d = \{ \tau_{\Gamma} \}_{\Gamma}$ be a collection of delay
functions for each combinatorial type $\Gamma$ of $\partial
\ol{\RR}^{d,0,0}$.  By $\tau_{\Gamma} | \Gamma'$ we mean the subset of
$\tau_\Gamma$ given by edges of $\Gamma'$, that is, $ \{ \tau_e , e
\in \Edge(\Gamma') \}$.

\begin{definition}  \label{delayfns}
{\rm (Compatible collections of delay functions)} A collection
$\{\tau_\Gamma \}$ of delay functions is {\em compatible} if the
following properties hold.  Let $\Gamma$ be a combinatorial type of
free nodal biquilted disk and $v_0,\ldots,v_k$ as in Definition
\ref{notations}.
\begin{enumerate}

\item {\rm (Subtree property)} \label{subtreeprop}  Let
  $\Gamma_1, \ldots, \Gamma_{|v_0|-1}$ denote the subtrees of $\Gamma$
  attached to $v_0$ at its incoming edges; then
  $\Gamma_1, \ldots, \Gamma_{|v_0|-1}$ are combinatorial types for
  nodal biquilted disks.  Let $r_i = (r_v)_{v \in \Ver(\Gamma_i)}$ be
  the components of
  $r = (r_v)_{v \in \Ver(\Gamma)} \in \RR^{\pre}_\Gamma$ corresponding
  to $\Gamma_i$.  We require that
  $ \tau_{\Gamma}(r) \big\lvert_{\Gamma_i} = \tau_{\Gamma_i}(r_i) .$
  That is, for each edge $e$ of $\Gamma_i$, the delay function
  $\tau_{\Gamma,e}(r)$ is equal to $\tau_{\Gamma_i,e}(r_i)$.  See
  Figure \ref{subtreefig}.

\begin{figure}[ht]
\begin{picture}(0,0)%
\includegraphics{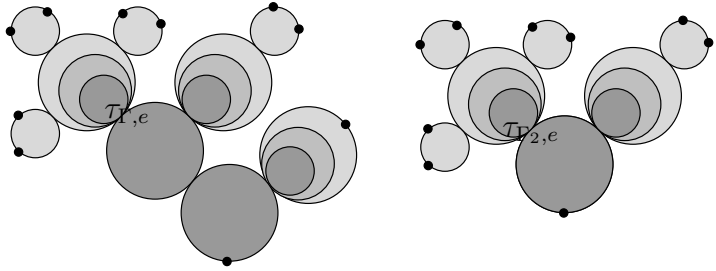}%
\end{picture}%
\setlength{\unitlength}{3947sp}%
\begingroup\makeatletter\ifx\SetFigFont\undefined%
\gdef\SetFigFont#1#2#3#4#5{%
  \reset@font\fontsize{#1}{#2pt}%
  \fontfamily{#3}\fontseries{#4}\fontshape{#5}%
  \selectfont}%
\fi\endgroup%
\begin{picture}(4453,1667)(472,-3792)
\put(3595,-2967){\makebox(0,0)[lb]{\smash{{\SetFigFont{12}{14.4}{\rmdefault}{\mddefault}{\updefault}{\color[rgb]{0,0,0}$\tau_{\Gamma_2,e}$}%
}}}}
\put(1096,-2856){\makebox(0,0)[lb]{\smash{{\SetFigFont{12}{14.4}{\rmdefault}{\mddefault}{\updefault}{\color[rgb]{0,0,0}$\tau_{\Gamma,e}$}%
}}}}
\end{picture}%

\caption{The (Subtree property)} 
\label{subtreefig}
\end{figure}

\item {\rm (Infinite or zero ratio property)} For each $i$, there
  exists an open neighborhood of $\rho_i^{-1}(0)$ resp. an open
  neighborhood $\rho_i^{-1}(\infty)$ in which all the delays
  $\tau_{\Gamma,e}$ between the root vertex $v_0$ and $v_i$ vanish.
  See Figure \ref{infratio}.
\begin{figure}[ht]
\begin{picture}(0,0)%
\includegraphics{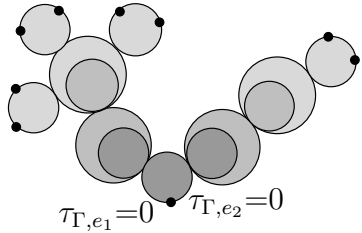}%
\end{picture}%
\setlength{\unitlength}{3947sp}%
\begingroup\makeatletter\ifx\SetFigFont\undefined%
\gdef\SetFigFont#1#2#3#4#5{%
  \reset@font\fontsize{#1}{#2pt}%
  \fontfamily{#3}\fontseries{#4}\fontshape{#5}%
  \selectfont}%
\fi\endgroup%
\begin{picture}(2219,1448)(491,-3707)
\put(812,-3638){\makebox(0,0)[lb]{\smash{{\SetFigFont{12}{14.4}{\rmdefault}{\mddefault}{\updefault}{\color[rgb]{0,0,0}$\tau_{\Gamma,e_1}$=0}%
}}}}
\put(1626,-3557){\makebox(0,0)[lb]{\smash{{\SetFigFont{12}{14.4}{\rmdefault}{\mddefault}{\updefault}{\color[rgb]{0,0,0}$\tau_{\Gamma,e_2}$=0}%
}}}}
\end{picture}%

\caption{ The (Infinite ratio property)} 
\label{infratio}
\end{figure}

\item {\rm (Refinement property)} Suppose that the combinatorial type
  $\Gamma^\prime$ is a refinement of $\Gamma$.  That is, suppose there
  is a surjective morphism $f: \Gamma^\prime \to \Gamma$ of trees
  preserving the labels and mapping colored vertices to colored
\label{refineprop}  vertices; let $r$ be the image of $r'$ under gluing.
Let $U$ \label{Uhere} be an open neighborhood of
$\ol{\RR}^{\pre}_\Gamma$ that is the image of an open neighborhood of
$\ol{\RR}^{\pre}_{\Gamma^\prime}$ in
$\ol{\RR}^{\pre}_{\Gamma^\prime} \times \cG_{\Gamma^\prime}$ by the
gluing procedure.  We require that $\tau_\Gamma\big\lvert_U$ is
determined by $\tau_{\Gamma^\prime}$ as follows: for each
$e \in \Edge(\Gamma)$, and $r \in U$, the delay function for $\Gamma$
is given by \label{isgivenby}
\begin{equation} \label{delayedsum} \tau_{\Gamma,e}(r) = \tau_{\Gamma',e} + \sum_{e'}
\tau_{\Gamma',e'}(r') \end{equation} 
where the sum is over edges $e'$ in $\Gamma'$ that are collapsed by
$f: \Gamma' \to \Gamma$, and the $e$ is the next-furthest-away edge
from the root vertex.  If $\Gamma^\prime$ has no twice-quilted
vertices then by the previous item $\tau_{\Gamma}$ vanishes in the
gluing region.  See Figures \ref{refine} and \ref{casetwo}.

\begin{figure}[ht]
\begin{picture}(0,0)%
\includegraphics{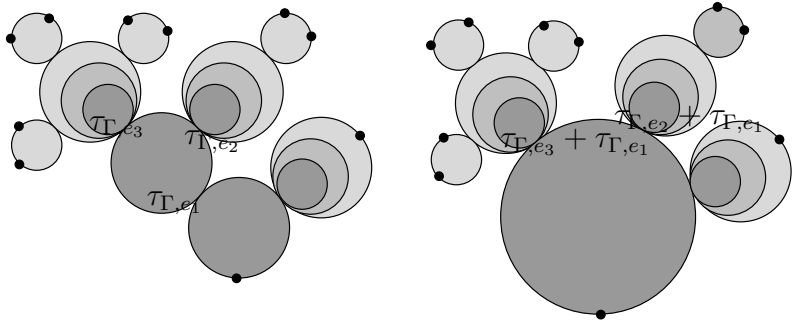}%
\end{picture}%
\setlength{\unitlength}{3947sp}%
\begingroup\makeatletter\ifx\SetFigFont\undefined%
\gdef\SetFigFont#1#2#3#4#5{%
  \reset@font\fontsize{#1}{#2pt}%
  \fontfamily{#3}\fontseries{#4}\fontshape{#5}%
  \selectfont}%
\fi\endgroup%
\begin{picture}(4944,2004)(472,-4074)
\put(1361,-3350){\makebox(0,0)[lb]{\smash{{\SetFigFont{12}{14.4}{\rmdefault}{\mddefault}{\updefault}{\color[rgb]{0,0,0}$\tau_{\Gamma,e_1}$}%
}}}}
\put(999,-2875){\makebox(0,0)[lb]{\smash{{\SetFigFont{12}{14.4}{\rmdefault}{\mddefault}{\updefault}{\color[rgb]{0,0,0}$\tau_{\Gamma,e_3}$}%
}}}}
\put(1597,-2967){\makebox(0,0)[lb]{\smash{{\SetFigFont{12}{14.4}{\rmdefault}{\mddefault}{\updefault}{\color[rgb]{0,0,0}$\tau_{\Gamma,e_2}$}%
}}}}
\put(3580,-2988){\makebox(0,0)[lb]{\smash{{\SetFigFont{12}{14.4}{\rmdefault}{\mddefault}{\updefault}{\color[rgb]{0,0,0}$\tau_{\Gamma,e_3}+\tau_{\Gamma,e_1}$}%
}}}}
\put(4297,-2842){\makebox(0,0)[lb]{\smash{{\SetFigFont{12}{14.4}{\rmdefault}{\mddefault}{\updefault}{\color[rgb]{0,0,0}$\tau_{\Gamma,e_2}+\tau_{\Gamma,e_1}$}%
}}}}
\end{picture}%
\caption{The (Refinement property), first case} 
\label{refine}
\end{figure}

In the case that the collapsed edges connect twice quilted components
with unquilted components, this means that the delay functions are
equal for both types, as in Figure \ref{casetwo}.

\item {\rm (Core property)} Let two combinatorial types $\Gamma$ and
  $\Gamma^\prime$ have the same core $\Gamma_0$, let $r,r'$ be disks
  of type $\Gamma$ resp. $\Gamma^\prime$ and $r_0,r'_0$ the disks of
  type $\Gamma_0$ obtained by removing the components except for those
  corresponding to vertices of $\Gamma_0$.  If $r_0 = r'_0$ then
$$\tau_{\Gamma,e}(r) = \tau_{\Gamma^\prime,e}(r') .$$
That is, the delay functions depend only on the region between the
root vertex and the outermost colored vertices.
\end{enumerate}
\label{positivehere} A collection of compatible delay functions $\{ \tau_\Gamma \}$ is {\em
  positive} if, for each type $\Gamma$ and every vertex
$v \in \Gamma_0$ with $k$ incoming edges labeled in counterclockwise
order by $e_1, \ldots, e_{k}$ (that is, the ordering induced by the
ordering of the labels on the leaves) their associated delay functions
$\tau_{\Gamma,e}$ satisfy
$\tau_{\Gamma,e_1} < \tau_{\Gamma,e_2} < \ldots < \tau_{\Gamma,e_{k}}
.$

\begin{figure}[ht]
\begin{picture}(0,0)%
\includegraphics{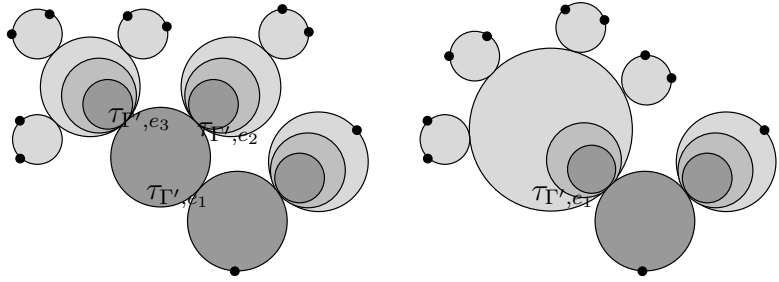}%
\end{picture}%
\setlength{\unitlength}{3947sp}%
\begingroup\makeatletter\ifx\SetFigFont\undefined%
\gdef\SetFigFont#1#2#3#4#5{%
  \reset@font\fontsize{#1}{#2pt}%
  \fontfamily{#3}\fontseries{#4}\fontshape{#5}%
  \selectfont}%
\fi\endgroup%
\begin{picture}(4843,1730)(472,-3837)
\put(1110,-2865){\makebox(0,0)[lb]{\smash{{\SetFigFont{12}{14.4}{\rmdefault}{\mddefault}{\updefault}{\color[rgb]{0,0,0}$\tau_{\Gamma',e_3}$}%
}}}}
\put(1668,-2950){\makebox(0,0)[lb]{\smash{{\SetFigFont{12}{14.4}{\rmdefault}{\mddefault}{\updefault}{\color[rgb]{0,0,0}$\tau_{\Gamma',e_2}$}%
}}}}
\put(1350,-3346){\makebox(0,0)[lb]{\smash{{\SetFigFont{12}{14.4}{\rmdefault}{\mddefault}{\updefault}{\color[rgb]{0,0,0}$\tau_{\Gamma',e_1}$}%
}}}}
\put(3769,-3351){\makebox(0,0)[lb]{\smash{{\SetFigFont{12}{14.4}{\rmdefault}{\mddefault}{\updefault}{\color[rgb]{0,0,0}$\tau_{\Gamma',e_1}$}%
}}}}
\end{picture}%

\caption{The (Refinement property), second case} 
\label{casetwo} 
\end{figure}

\end{definition}

In Figure \ref{casetwo}, the second case of the (Refinement property)
is illustrated by a situation in which the two nodes labelled $e_2$
and $e_3$ are resolved, creating a biquilted disk with five unquilted
disks attached.

\begin{remark}   \label{explainprops}
  The (Subtree property) allows the inductive construction of delay
  functions, starting with the strata of lowest dimension, so that the
  moduli spaces have expected dimension.  The (Refinement property)
  implies that the boundary of the one-dimensional component of the
  moduli space is the union of moduli spaces for the combinatorial
  types corresponding to refinements, since the sum \eqref{delayeval}
  matches the sum in \eqref{delayedsum}.  The (Core Property) implies
  that for the types corresponding to bubbling off an unquilted disk
  is the product of moduli spaces for the components.  The (Infinite
  or Zero Ratio Property) implies that the fiber product is the one
  expected in the case that there are no biquilted disks.
\end{remark}

\subsection{Inductive construction of regular, positive, compatible delay functions.}
\label{inductive}

In order to achieve transversality we construct regular, positive
delay functions compatibly by induction.  The next lemma furnishes the
inductive step.

\begin{lemma}\label{compat} {\rm (Inductive definition of regular positive delay functions)}  
Let $d \ge 1$ be an integer and $\ul{L} =
(\ul{L}^0,\ul{L}^1,\ldots,\ul{L}^d, L_{01}, L_{12})$ a Lagrangian
labeling for biquilted disks with $d+1$ boundary markings.  Suppose
that for each $ 1 \leq k < d$, the moduli spaces
$\M^{k,0,0}(\ul{L}^\prime)$ have been constructed for all Lagrangian
labellings $\ul{L}^\prime$ for biquilted disks in $\RR^{k,0,0}$ using
a compatible, regular, positive collection $\{\tau^k(\ul{L}^\prime)
=(\tau_\Gamma(\ul{L}^\prime)) \}_{1 \leq k < d}$ of delay functions
for less than $d$ leaves. Then there exists an extension of this
collection to a regular, compatible, positive collection
$\{\tau^k(\ul{L})=(\tau_\Gamma(\ul{L}))\}_{1 \leq k \leq d}$
for at most $d$ leaves.
\end{lemma}

\begin{proof}  
  Let $\Gamma$ be a combinatorial type with $d$ incoming markings.  We
  suppose that we have constructed inductively regular delay functions
  for types $\Gamma'$ with $e$ incoming markings for $e < d$, as well
  as for types $\Gamma'$ appearing in the (Refinement property) for
  $\Gamma$.  Note that the case that $\Gamma$ is a maximally-refined
  combinatorial type, so that the corresponding stratum of the moduli
  space is zero-dimensional, is covered by the inductive step although
  in this case there is no ``boundary'' of the stratum.

  We now construct a regular functions
  $\tau_\Gamma := \tau_\Gamma(\ul{L})$.  We may assume that $\Gamma$
  has no components ``beyond the quilted components''.  Indeed, by the
  (Core property) the delay functions are independent of the
  \label{compatproof} 
  additional components.  (For example, in the case that $\Gamma$ has
  a single unquilted component as the root, with two nodal points on
  its boundary, at each of which is attached a biquilted component
  with a single boundary marking, the moduli space
  $\ol{\RR}^{\pre}_\Gamma$ is a square, with coordinates given by the
  ratios between the quilted circles.  The inductive hypothesis
  prescribes the delay functions on the boundary.)  The (Subtree
  property) implies that all the delay functions in
  $\tau_\Gamma := \tau_\Gamma(\ul{L})$ {\em except} those for the
  finite edges adjacent to $v_0$, the root component, are already
  fixed.  It remains to find regular delay functions for the finite
  edges adjacent to the root component of each combinatorial type, in
  a way that is also compatible with conditions (Infinite or zero
  ratio property) and (Refinement property).  Choose an open
  neighborhood $U$ of $\partial \ol{\RR}^{\pre}_\Gamma $ in
  $\ol{\RR}^{\pre}_\Gamma$ in which the delay functions $\tau_\Gamma$
  for the incoming edges adjacent to the root vertex are already
  determined by the gluing construction and the delay functions on the
  boundary. By an argument similar to that of Theorem
  \ref{familytransversality} there exists a smaller open neighborhood
  $V$ of $\partial \RR^{\pre}_\Gamma $ in $\ol{\RR}^{\pre}_\Gamma$
  with $\ol{V} \subset U$ such that every element in the zero and
  one-dimensional components of the $\tau_\Gamma$-deformed moduli
  space $\M_\Gamma$ is regular.  We show that $\tau_\Gamma$ extends
  over the interior of $\ol{\RR}^{\pre}_\Gamma$.  To set up the
  relevant function spaces let $l\ge 0$ be an integer.  Let
  $C_{\tau_{\Gamma,i}}^l(\RR^{\pre}_\Gamma)$ denote the Banach
  manifold of functions with $l$ bounded derivatives on
  $\RR^{\pre}_\Gamma$, equal to the $i$-th component $\tau_{i}$ of
  $\tau_i$ on $\ol{V}$.  Let $\Gamma_i, i =1,\ldots, n$ be the trees
  attached to the root vertex $v_0$.  Consider the evaluation
  map \label{evalmap}
\bea & & \on{ev}: \M_{{\Gamma_1}}\times \ldots \times
\M_{{\Gamma_{n}}} \times \M_{v_0} \times \prod_{i=1}^{n} C^l_{
  \tau_{\Gamma,i}}(\RR^{\pre}_{\Gamma}) \to \R^{n-1}\\ & & \left((r_1, u_1),
\ldots, (r_n, u_n), (r_0,u_0), \tau_1, \ldots, \tau_n\right) \mapsto
\\ & & \left( \rho_{\Gamma_j}(r_j) \exp( \tau_j(r)) -
\rho_{\Gamma_{j+1}}(r_{j+1}) \exp( \tau_{j+1}(r)) \right)_{j=1}^{n-1}
\eea
where $ r = (r_0,\ldots, r_n)$.  Note that $0$ is a regular value.
The Sard-Smale theorem implies that for $l$ sufficiently large the
regular values of the projection
\bea \Pi : \on{ev}^{-1}(0) \to \T^l(\RR^{\pre}_\Gamma)
:=\prod_{i=1}^{n} C^l_{
  \tau_i}(\RR^{\pre}_\Gamma) \eea
form a comeager subset.  Denote the subset of $
\T^l(\RR^{\pre}_\Gamma)$ consisting of smooth functions by
$\T(\RR^{\pre}_\Gamma)$.

We use Taubes' argument (see \cite[Section 3.2]{ms:jh}) to show
that the subset $\T_{\reg}(\RR^{\pre}_\Gamma)$ of
$\T(\RR^{\pre}_\Gamma)$ consisting of regular delay functions of class
$C^\infty$ is dense in $\T(\RR^{\pre}_\Gamma)$.  (Note here that each
of the functions in such a collection extends to the boundary and so
defines a function on the closure $\ol{\RR}^{\pre}_\Gamma$.)  For each
$i = 1,\ldots, k$ fix a component of each $\M_{\Gamma_i}$ and
$\M_{v_0}$ of fixed dimension.  The product of these components is a
connected finite dimensional manifold $X$.  Let $K$ be a compact
subset of $X$ and let $\T_{\reg, K}(\RR^{\pre}_\Gamma)$ be the subset
of smooth delay functions that are regular on $K$.  We will show that
$\T_{\reg, K}(\RR^{\pre}_\Gamma)$ is open and dense in
$\T(\RR^{\pre}_\Gamma)$.

To show that $\T_{\reg, K}(\RR^{\pre}_\Gamma)$ is open, we show that
the complement is closed.  Let $\tau_\nu$ be a sequence of smooth
delay functions in the complement $\T_{\reg, K}(\RR^{\pre}_\Gamma)^c$
converging to a smooth delay function $\tau \in
\T(\RR^{\pre}_\Gamma)$.  Thus we can find a sequence of points $p_\nu
\in K$ where each $p_\nu$ is a critical point of the delayed
evaluation map $\rho_{\tau_\nu}: X \to \R^k$.  Passing to a
subsequence if necessary, we can assume by the compactness of $K$ that
$p_\nu \to p \in K$.  The delayed evaluation map $\rho_{\tau}: X \to
\R^k$ for the limit $\tau$ cannot be regular at $p$.  Indeed, because
if it were regular then for sufficiently large $\nu$ the delay
evaluations $\rho_{\tau_\nu}$ would be regular at $p_\nu$.  Hence
$p\in K$ is a critical point of $\rho_{\tau}$, and $\tau \in \T_{\reg,
  K}(\RR^{\pre}_\Gamma)^c$.  So $\T_{\reg, K}(\RR^{\pre}_\Gamma)^c$ is
closed in $\T(\RR^{\pre}_\Gamma)$.  Similarly, let $\T_{\reg,
  K}^l(\RR^{\pre}_\Gamma)$ be the set of $C^l$ delay functions that
are regular on $K$.  By the same argument as above, $\T_{\reg,
  K}^l(\RR^{\pre}_\Gamma)$ is open in $\T^l(\RR^{\pre}_\Gamma)$.
Moreover $\T_{\reg, K}^l(\RR^{\pre}_\Gamma)$ is dense in
$\T^l(\RR^{\pre}_\Gamma)$.  Indeed $\T_{\reg, K}^l(\RR^{\pre}_\Gamma)
\supset \T_{\reg}^l(\RR^{\pre}_\Gamma)$ and we already know that
$\T_{\reg}^l(\RR^{\pre}_\Gamma)$ is dense in $\T^l(\RR^{\pre}_\Gamma)$
for sufficiently large $l$.

To show that $\T_{\reg, K}(\RR^{\pre}_\Gamma)$ is dense, fix $\tau \in
\T(\RR^{\pre}_\Gamma)$.  Since $\T^l_{\reg, K}(\RR^{\pre}_\Gamma)$ is
dense in $\T^l(\RR^{\pre}_\Gamma)$ we can find  
$$\tau_l \in
\T^l_{\reg, K}(\RR^{\pre}_\Gamma), \quad \|\tau - \tau_l\|_{C^l}
\leq 2^{-l} .$$ 
Moreover, $\tau_l \in \T^l_{\reg, K}(\RR^{\pre}_\Gamma)$ is open in
$\T^l(\RR^{\pre}_\Gamma)$.  So there exists an $\epsilon_l>0$ such
that
$$\| \widehat{\tau} -
\tau_l\|_{C^l} < \epsilon_l \implies \widehat{\tau} \in \T^l_{\reg,
  K}(\RR^{\pre}_\Gamma) .$$  
Since smooth functions are dense in $C^l$ this means that we can find
$$\widehat{\tau}_l \in \T(\RR^{\pre}_\Gamma) \cap \T^l_{\reg,
  K}(\RR^{\pre}_\Gamma), \quad \|\widehat{\tau}_l - \tau_l\|_{C^l} <
\min(\epsilon_l, 2^{-l}) .$$  
It therefore follows that $\widehat{\tau}_l \in \T_{\reg,
  K}(\RR^{\pre}_\Gamma)$, and $\widehat{\tau}_l$ converges as $l \to
\infty$ to $\tau$ in the $C^\infty$ topology.

Thus $ \T_{\reg, K}(\RR^{\pre}_\Gamma)$ is open and dense in
$\T(\RR^{\pre}_\Gamma)$.  We exhaust $X$ with a countable sequence
of compact subsets $K$, and there are at most countably many
components of each $\M_{\Gamma_i}$ and $\M_{v_0}$ of a given dimension
(i.e. there are countably many $X$).  Hence $\T_{\reg}
(\RR^{\pre}_\Gamma)$ is comeager in $\T(\RR^{\pre}_\Gamma)$. Finally,
the positivity condition is an open condition in
$\T(\RR^{\pre}_\Gamma)$.  So the set of smooth, regular, compatible
and positive delay functions is non-empty.

Thus, by induction, there exists a smooth, positive, compatible,
regular delay function $\tau_\Gamma$ for each combinatorial type of
biquilted $d+1$-marked disk, and hence a regular compatible collection
$\tau^d$.

To apply the induction, note that for $d = 1$ the regularity condition
is vacuously satisfied (there is no transversality condition to the
diagonal) so we may take that as the base step.
\end{proof}

\subsubsection{Perturbed quilts parametrized by the bimultiplihedron}  

We now build the family of quilts that will be the domains for our
biquilted holomorphic disks.  The definition incorporates the
inductive procedure for choosing delay functions in the last
subsection.  The construction here is somewhat different from the
constructions in previous chapters in that the families of quilts
depend on the choice of Lagrangians.

\begin{definition} \label{biquilts} 
{\rm (Family of quilts parametrized by the bimultiplihedron)} Given a
positive integer $d$, Lagrangians $\ul{L} = (L_0,\ldots, L_d)$, and a
delay function $\tau^d$, we first define the bundle $\partial
\SS^{d,0,0}(\ul{L}) \lra \partial \RR^{d,0,0}$ of nodal quilted
surfaces with striplike ends over the boundary, $\partial
\RR^{d,0,0}$.  Then we extend it over a neighborhood of the boundary.
Finally we choose a smooth interpolation over the remaining interior of
$\RR^{d,0,0}$.  Let
$$\ol{\RR}_\Gamma^\tau := \rho_{\tau_\Gamma}^{-1}(\triangle) \subset
\ol{\RR}_\Gamma^{\pre}$$
denote the inverse image of the shifted diagonal as in
\eqref{delayedfiber}.  For different choices of $\tau$ (which are all
homotopic) these spaces are isomorphic as decomposed spaces (see
Definition \ref{decomposed}) and so in particular
$\ol{\RR}_\Gamma^\tau$ is isomorphic to $\ol{\RR}_\Gamma$ as a
decomposed space.  Let
$$ \pi_\Gamma: \ \ol{\SS}_\Gamma^{\pre} \to \ol{\RR}_\Gamma^{\pre} $$
denote the product of surface bundles for the vertices of $\Gamma$
already constructed.  Let
$$ \ol{\SS}_\Gamma^\tau := \pi_\Gamma^{-1}(\ol{\RR}_\Gamma^\tau) $$
denote the restriction to the shifted fiber product.  Thus each element
of $\ol{\M}_\Gamma^\tau$ in \eqref{delayedfiber} has domain an element
of $\ol{\SS}_\Gamma^\tau$.

Having defined the family of quilts over the boundary, we extend the
bundle over a neighborhood of the boundary.  On the facets of
$\ol{\RR}^{d,0,0}$ for which $\rho \in (0,\infty)$, we use the gluing
construction to extend the bundle over a neighborhood of those
facets. The codimension one facet where $\rho = 0$ does not correspond
to the formation of a nodal disk, but rather corresponds to the
convergence of the inner and outer seams.  Given a family of quilted
surfaces over the moduli space of quilted disks with a single seam, we
extend the family in a neighborhood of the boundary by replacing the
seam by a strip of small width $\rho$ so that is compatible with the
gluing construction near its higher codimension boundary strata.  For
the codimension one facets corresponding to $\rho = \infty$, we take
the delay functions to vanish as above in the (Infinite or zero ratio
property) and then use the usual gluing construction for the perturbed
gluing parameters to extend the bundle over a neighborhood of these
boundary strata.  Once the bundle is extended over a neighborhood of
the boundary we fix a smooth interpolation over the remainder of the
interior.  This ends the definition. 
\end{definition} 

By Theorem \ref{familytransversality} we may also extend the regular
collection of perturbation data over the boundary
$\partial \RR^{d,0,0}$ to get regular perturbation data over all of
$\ol{\RR}^{d,0,0}$.  With this data the moduli space $\M^{d,0,0}
$ \label{:=} of pseudoholomorphic quilts whose domain is in the family
$\SS^{d,0,0,\tau}$, is transversally defined.

\begin{corollary} 
\label{Bd00} {\rm (Description of the 
ends of the one-dimensional components of the moduli space of
biquilted disks)} Let $M_0,M_1,M_2$ be symplectic backgrounds with the
same monotonicity constant and 
$$L_{01} \subset M_0^- \times M_1,\quad L_{12} \subset M_1^- \times
M_2$$ 
admissible Lagrangian correspondences with brane structure.  For
regular, coherent perturbation data the moduli spaces of
pseudoholomorphic quilts ${\M}^{d,0,0}$ have finite zero dimensional
component $\M^{d,0,0}_0$ and one-dimensional component
$\ol{\M}^{d,0,0}_1$ that admits a compactification as a one-manifold
with boundary the union
\begin{equation} \label{bounddes}
 \partial \ol{\M}^{d,0,0}_1 = \bigcup_{\Gamma}
 {\M}^{d,0,0}_{\Gamma,1} \end{equation}
where either (1) $\Gamma$ is stable and so $\RR_{\Gamma}^{d,0,0}$ is a
codimension one stratum in $\ol{\RR}^{d,0,0}$, or (2) $\Gamma$ is
unstable and corresponds to bubbling off a Floer trajectory.
\end{corollary}

\begin{proof}   Compactness is by Theorem \ref{gromov}. 
The description \eqref{bounddes} of the boundary of the
one-dimensional moduli spaces follows from the monotonicity
assumptions in Definition \ref{branes}:  Any limiting configuration
with more than one component has a component with negative index, and
so does not exist.
\end{proof} 

\begin{remark} {\rm (Orientations for moduli of biquilted disks)}  
 The construction of orientations for $\M^{d,0,0}(x_0,\ldots,x_d) $ is
 similar to the previous cases.  It depends on a choice of orientation
 on $\ol{\RR}^{d,0,0}$, which itself depends on a choice of slice for
 the $SL(2,\R)$ action on the set of biquilted disks with marking.  We
 take as slice the set of disks with first two marked points fixed and
 the first inner circle fixed at radius $1/2$.  Define a
 diffeomorphism
$$ \RR^{d,0,0} \cong \{ (z_2 < \ldots < z_d, \rho_2) \} \subset
 \R^d $$
where $\rho_2 \in ( 1/2,1)$ is the radius of the second disk.  The
standard orientation on $\R^d$ induces an orientation on
$\RR^{d,0,0}$.
\end{remark} 

\begin{figure}
\includegraphics[width=3in,height=1in]{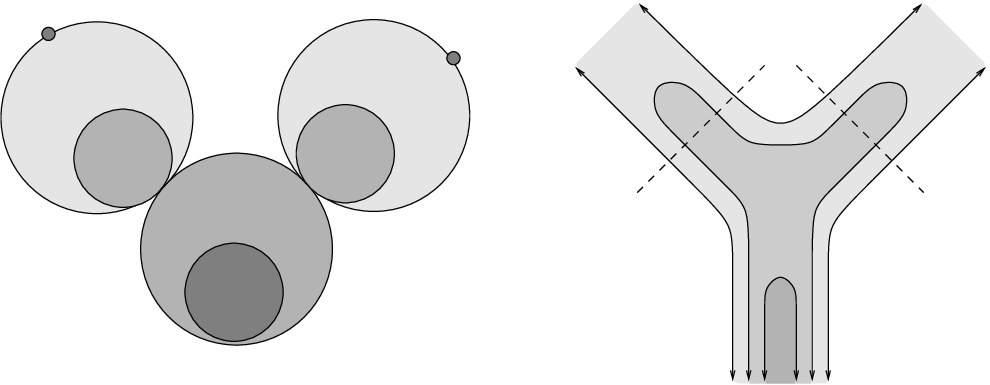}
\caption{Desingularization of a biquilted nodal disk}
\label{desing}
\end{figure}

\subsection{Homotopy for algebraic composition of correspondences}

In this section we compare the composition of \ainfty functors for
Lagrangian correspondences with the \ainfty functor for their
algebraic composition.  Recall from \eqref{algcomp} the definition of
algebraic composition: Let
$$L_{01} \subset M_0^- \times M_1, \quad L_{12} \subset
M_1^- \times M_2$$ 
be admissible Lagrangian correspondences with brane structure and
$\zeta > 1$.  The {\em algebraic composition} of $L_{01}, L_{12}$ with
width $\zeta$ is the generalized Lagrangian correspondence
$L_{01} \sharp_\zeta L_{12}$ with width $\zeta$ associated to the
strips taking values in the manifold $M_1$.  We compare the functor
for the algebraic composition with the composition of functors:
\begin{equation} \label{comparephis} 
\Phi(L_{01}
\sharp_\zeta L_{12}), \ \Phi(L_{12}) \circ
\Phi(L_{01}): \GFuk(M_0) \to \GFuk(M_2)  .\end{equation} 
Note that these functors act the same way on objects of $\GFuk(M_0)$.
The following is a preliminary result towards Theorem
\ref{maincompose}:

\begin{theorem} \label{compose1}  
{\rm (Algebraic composition theorem)} Let $M_0,M_1,M_2$ be symplectic
backgrounds with the same monotonicity constant and $L_{01}, L_{12}$
Lagrangian correspondences with admissible brane structures.  The
\ainfty composition $\Phi({L_{12}}) \circ \Phi({L_{01}})$ is \ainfty
homotopic to $\Phi( L_{01} \sharp_\zeta L_{12})$ in the sense of
\eqref{homotopy}.
\end{theorem}  

\begin{proof}  We suppose that we have constructed
smooth moduli spaces of biquilted disks of expected dimension by
choosing the perturbation data as in Corollary \ref{Bd00}.  Given
admissible generalized Lagrangian branes $\ul{L}^0,\ldots, \ul{L}^d$
for $M_0$, define maps
\begin{equation} \label{hmaps} \HH^0_d: \Hom(\ul{L}^0,\ul{L}^1) \times
  \ldots \Hom(\ul{L}^{d-1},\ul{L}^d) \to \Hom(\Phi(L_{01} \sharp
  L_{12}) \ul{L}^0 , \Phi(L_{01} \sharp_\zeta L_{12}) \ul{L}^d
  ) \end{equation}
by setting for generalized intersection points
$x_1,\ldots,x_d$
\begin{equation} \label{HH0} \HH^0_d(\bra{x_1},\ldots,\bra{x_d}) = \sum_{u \in 
  \M^{d,0,0}(x_0,\ldots,x_d)_0}
(-1)^\heartsuit \eps(u) \bra{x_0} \end{equation} 
where $\eps(u)$ are the orientation signs as in \eqref{epseq} and
$\heartsuit$ is defined in \eqref{heartsuit}.   We write 
\begin{equation} \label{HH0rho} \HH^{0,\rho'}_d = \sum_{\rho'}
  \HH^{0,\rho'}_d \end{equation}
where $\HH^{0,\rho'}_d$ is the contribution from biquilted disks
$(r,u)$ with \eqref{rhomap} $\rho(r) = \rho'$.   Since 
the locus of pairs of elements
\begin{equation} \label{samerho} \{ (r_1,u_1),(r_2,u_2) \in
  \M^{d_1,0,0}(x_0,\ldots,x_{d_1}) \times
  \M^{d_2,0,0}(x_0,\ldots,x_{d_2}) \ | \ \rho(r_1) = \rho(r_2)
  \} \end{equation}
is negative expected dimension, for generic perturbations the
restriction of $\rho$ to the union of the spaces
$\M^{d,0,0}(x_0,\ldots,x_d)_0$ is injective.  Similarly, define
$$ \HH^{1,\rho'}_d(\bra{x_1},\ldots,\bra{x_d}) = \sum_{u \in 
  \rho^{-1}(\rho') \subset \M^{d,0,0}(x_0,\ldots,x_d)_1}
(-1)^\heartsuit \eps(u) \bra{x_0} $$
for any $\rho'$ such that the intersection in the definition is
transverse.

Consider the boundary of the one-dimensional moduli space
$\ol{\M}^{d,0,0}(x_0,\ldots,x_d)_1$. \label{typeshere} By Corollary
\ref{Bd00}, the boundary points correspond to stable types
corresponding to either boundary facets in $\ol{\RR}^{d,0,0}$ or
bubbled off trajectories.  The types of facets are listed in
Proposition \ref{d00facets}.  The first type of facet, in which at
least two singly-quilted components appear, corresponds to the terms
in the definition of the composition of \ainfty functors
\eqref{composefunc}.  The last type of facet corresponds to the terms
in $\Phi({L_{12} \sharp L_{01}})$.  The remaining boundary components
of $\ol{\M}^{d,0,0}(x_0,\ldots,x_d)_1$ are elements of the strata
$ {\M}_{\Gamma}^{d,0,0}(x_0,\ldots,x_d)_1 $ where $\Gamma$ is either
unstable, corresponding to bubbling off a Floer trajectory, or a
stable combinatorial type.  The stable combinatorial types are either
an unquilted disk mapping to $M_0$ and a biquilted disk, or a
collection of biquilted disks attached to a unquilted disk mapping to
$M_2$.  Facets corresponding to bubbling off unquilted disks
correspond to the last set of terms in the definition of homotopy of
\ainfty functors in \eqref{mu1} with the homotopy $\TT^d = \HH^0_d$
defined in \eqref{hmaps}.  It remains to show that facets representing
bubbling off quilted disks correspond to the first set of terms in
\eqref{mu1}.  \label{remainstoshow} On the $m$ biquilted disks we have $m-1$ relations,
requiring that the inner/outer ratios be equal up to the shifts
$\tau_\Gamma$.  By Proposition \ref{mostly0s}, for $m-1$ of the
bubbles, the unconstrained moduli space is dimension $1$, and exactly
for one of the bubbles, say the $i$-th, the unconstrained moduli space
is dimension $0$.  Thus the contribution of this type of facet is
\begin{multline} \label{facetcontrib2}
\sum_{\rho,I_1,\ldots,I_r,i} \mu^m_{\GFuk(M_2)}( \HH^{1,\rho +
  \tau_{\Gamma,1}}(\bra{x_{I_1}}), \ldots, \HH^{1,\rho +
  \tau_{\Gamma,i-1}}(\bra{x_{I_{i-1}}}), \HH^{0,\rho +
  \tau_{\Gamma,i}}(\bra{x_{I_i}}), \\
\HH^{1,\rho +
  \tau_{\Gamma,i+1}}(\bra{x_{I_{i+1}}}) ,\ldots, \HH^{r,\rho +
  \tau_{\Gamma,r}}(\bra{x_{I_r})})
\end{multline}
where
$$ \HH^{1,\rho_0}_d(\bra{x_1},\ldots,\bra{x_d}) = \sum_{u \in
  \M^{d,0,0}({x_0},\ldots,{x_d})_1, \rho(u) = \rho_0}
(-1)^{\heartsuit} \eps(u) \bra{x_0} $$
counts over the moduli space of expected dimension one, of fixed ratio
$\rho_0$.

In order to define a homotopy between
$\Phi(L_{12}) \circ \Phi(L_{01})$ and
$\Phi(L_{01} \sharp_\zeta L_{12})$, we ``integrate over $\rho$'' in
the following sense.  First we consider the case that there are
finitely many contributions to $\HH^0$.  Since the restriction of
$\rho$ to the zero-dimensional component of the moduli space is
injective, see the discussion after \eqref{samerho}, we may divide
$(0,\infty)$ into finitely many intervals
$[a_i,a_{i+1}], i =0,\ldots,s$ such that there is at most one
contribution to $\HH^0$ in each interval occurring at say
$\rho^{-1}(\delta_i)$. \label{aipar} The intersections
$\M^{d,0,0}( x_1,\ldots, x_d )_1 \cap \rho^{-1}(a_i)$ and
$\M^{d,0,0}( x_1,\ldots, x_d)_1 \cap \rho^{-1}(a_{i+1})$ are
components in the boundary of
$\M^{d,0,0}(x_1,\ldots,x_d)_1 \cap \rho^{-1}([a_i,a_{i+1}])$.  The
other boundary components correspond to bubbling off unquilted disks
from a twice-quilted disk, or bubbling off a number of quilted
disks. Thus
\begin{multline}
\label{bigH}
 \HH^{1,a_{i+1}}(\bra{x_1},\ldots,\bra{x_d}) =
 \HH^{1,a_i}(\bra{x_1},\ldots,\bra{x_d}) + \\ \pm \sum_{I_1,\ldots,I_r}
 \mu^m_{\GFuk(M_2)}( \HH^{n_1,\delta + \tau_{\Gamma,1}}(\bra{x_{I_1}}),
 \ldots, \HH^{n_m,\delta + \tau_{\Gamma,m}}(\bra{x_{I_m}})) \\ \pm
 \sum_{\delta_i,j,k} \HH^{0,\delta}(\bra{ x_1}, \ldots,
 \mu^k_{\GFuk(M_0)} (\bra{x_{j+1}}, \ldots,\bra{x_{j+k}}),
 \bra{x_{j+k+1}},\ldots, \bra{x_d}) .\end{multline}
where each $n_l \in \{0,1 \}$ and $\sum_{i=1}^m (n_i - 1) = -1$, by
the transversality assumption.  For each $u \in
\HH^{0,\delta_i}(\cdot)$, we have $n_{k(u)} = 0$ for some $k(u)$ and
otherwise $n_l = 1, l \neq k(u)$, see Proposition \ref{mostly0s}.  Now
by assumption, there are no other values of the restriction of $\rho$
to $\M^{l,0,0}(\cdot)_0, l \leq d$ in $[a_i,a_{i+1}]$.  It follows
that the moduli spaces $\M^{|I_j|,0,0}(x_{I_j})_1 \cap
\rho^{-1}([a_i,\delta])$ have boundary given by
\begin{equation} \label{Ijhere} \partial \M^{|I_j|,0,0}(x_{I_j})_1
\cap \rho^{-1}([a_i,\delta]) = \M^{|I_j|,0,0}(x_{I_j})_1 \cap
\rho^{-1}( \{a_i,\delta \}), j < k(u) \end{equation}
$$ \partial \M^{|I_j|,0,0}(x_{I_j})_1 \cap
\rho^{-1}([\delta,a_{i+1}]) = \M^{|I_j|,0,0}(x_{I_j})_1 \cap
\rho^{-1}( \{\delta, a_{i+1} \}), j > k(u) .$$
Since the delay functions are positive by assumption, these equalities
hold after replacing $\delta$ with the nearby values $a_i, a_{i+1}$:
\begin{equation} \label{sameHH} \HH^{1,\delta}(\bra{x_{I_j}}) =
\HH^{1,a_i}(\bra{x_{I_j}}), \quad j < k(u), \quad 
 \HH^{1,\delta}(\bra{x_{I_j}}) =
\HH^{1,a_{i+1}}(\bra{x_{I_j}}), \quad j > k(u) .\end{equation}
By substituting these equalities into \eqref{bigH}, we obtain the
terms in the definition of \ainfty homotopy between the functors with
ratio $a_i$ and those for ratio $a_{i+1}$.  Taking the composition of
these homotopies as in \eqref{composehom} proves the theorem for $\rho
> 0$, up to sign in the case that the number of contributions to
$\HH^0$ is finite.

In general we define the homotopy by an inductive limit.  For each
$d_0$ there are finitely many contributions to $\HH^{0,d}$ for $d \leq
d_0$.
The construction of the previous paragraph yields a map
$$ \TT^{\leq d} = \HH^{0,\delta_1} \circ ( \HH^{0,\delta_2} \circ ( \ldots
\circ \HH^{0,\delta_s} ) \ldots ) $$
that is a homotopy of \ainfty functors from $\Phi(L_{12}) \circ
\Phi(L_{01})$ to $\Phi(L_{01} \sharp_\zeta L_{12})$ up to terms involving
composition maps involving more than $d$ entries.  That is, the
collections
$$ (\Phi(L_{12}) \circ \Phi(L_{01}))_{d \leq d_0}, \quad (\Phi(L_{01}
\sharp L_{12}))_{d \leq d_0}, \quad (\HH^{0,n})_{d \leq d_0} $$
satisfy equation \eqref{mu1} for $d \leq d_0$.  Furthermore by
construction if $d < e$ then $\TT^{\leq d,i} = \TT^{\leq e,i}$ for
$i \leq d$, since the higher corrections only involve maps with high
numbers of inputs.  It follows that the limit
$$ \TT := \lim_{d_0 \to \infty} \TT^{\leq d_0 } $$
is well-defined.  Furthermore the ``differential'' of $\TT$ is the
limit
$$ (\mu^1_{\Hom(\F_1,\F_2)} \TT)^d = \lim_{d_0 \to \infty}
(\mu^1_{\Hom(\F_1,\F_2)} \TT^{\leq d_0})^d .$$
So $\TT = (\TT^d)_{d \ge 0}$ is a homotopy from $\Phi(L_{12}) \circ
\Phi(L_{01})$ to $\Phi(L_{01} \sharp_\zeta L_{12})$.

It remains to check the signs.  Since $|\TT| = |\HH| = -1$, the signs
in the formula \eqref{homotopy} vanish.  Consider the signs of the
inclusions of strata into $\ol{\RR}^{d,0,0}$: An embedding
$\RR^f \times \RR^{e,0,0} \to \ol{\RR}^{d,0,0}$ corresponding to an
unquilted bubble containing the markings $i+1,\ldots,i+f$ has sign
$ (-1)^{if + i }$, cf. \eqref{signs}.  For the facets induced by
embeddings
$ (\RR^{i_1,0} \times \ldots \times \RR^{i_m,0}) \times \RR^{e,0} \to
\ol{\RR}^{d,0,0} $
gluing acts on signs by $ 1+ (-1)^{\sum_{j=1}^m (m- j) (i_j - 1)} $
c.f. Lemma \ref{signs2}.  For the facets induced by embeddings
$ (\RR^{i_1,0,0} \times_{[0,\infty]} \ldots \times_{[0,\infty]} \RR^{i_m,0,0}) \times
\RR^{m} \to \ol{\RR}^{d,0,0} $
(where the real number is the ratio of the radii of the two interior
circles) the gluing map has sign
$1+ (-1)^{ \sum_{j = 1}^m (m-j)i_j }$.  For the facet given by the
embedding $\RR^{d,0} \to \ol{\RR}^{d,0,0}$ the gluing map is orientation
preserving.  The signs for the embeddings of the facets combine with
the Koszul signs to give the signs in the formulas \eqref{homotopy}
and \eqref{mu1}.
\end{proof} 

\subsection{Homotopy for geometric composition of correspondences}

In this section we prove Theorem \ref{maincompose} relating the
composition of functors with the functor for the composition of
correspondences.  First we show that the \ainfty functors are
quasi-isomorphic.  Let $L_{01},L_{12}$ be Lagrangian correspondences
as above with the property that the geometric composition $L_{02} :=
L_{01} \circ L_{12} $ is smooth and embedded by projection into $M_0^-
\times M_2$ as in \eqref{geomcomp}.

\begin{proposition} \label{versusprop} {\rm (Algebraic versus geometric composition)}  
  Let $M_0,M_1,M_2$ be symplectic backgrounds with the same
  monotonicity constant and $L_{01} \subset M_0^- \times M_1$ and
  $L_{12} \subset M_1^- \times M_2$ admissible Lagrangian
  correspondences with brane structure.  Suppose that
  $ L_{02} = L_{01} \circ L_{12}$ is transverse and embedded, and
  admissible in $M_0^- \times M_2$. Then for any $\zeta > 0$ the
  functors $\Phi(L_{01} \sharp_\zeta L_{12})$ and $\Phi(L_{02})$ are
  quasi-isomorphic in $\Fun(\GFuk(M_0),\GFuk(M_2))$.
\end{proposition}

\begin{proof} Theorem \ref{compose1} shows that
  $\Phi(L_{12}) \circ \Phi(L_{01})$,
  $\Phi(L_{01} \sharp_\zeta L_{12})$ are homotopic, in particular,
  quasiisomorphic.  To show that
  $\Phi(L_{01} \sharp_\zeta L_{12}), \Phi(L_{02})$ are
  quasi-isomorphic recall that in \cite{we:co} the second two authors
  constructed cocycles
  $ \phi \in CF(L_{01} \sharp_\zeta L_{12},L_{02})$ and
  $ \psi \in CF(L_{02}, L_{01} \sharp_\zeta L_{12}) $ with the
  property that
\begin{equation} \label{identity} 
[\psi] \circ [\phi] = 1_{L_{02}} \in HF(L_{02},L_{02}), \quad [\psi]
\circ [\phi] = 1_{L_{01} \sharp_\zeta L_{12}} \in HF(L_{01} \sharp
L_{12},L_{01} \sharp_\zeta L_{12}) .\end{equation}
An alternative argument for the existence of the cocycle is given in
Lekili-Lipyanskiy \cite{ll:geom}.  Let $T(\phi), T(\psi)$ denote the
corresponding natural natural transformations.  It follows from
\eqref{identity} that 
$$T(\phi) \circ T(\psi) \in \Aut(\Phi(L_{01} \sharp_\zeta L_{12})), \quad 
 T(\psi) \circ T(\phi)\in \Aut(\Phi(L_{02})) $$
are cohomologous to the identity natural transformations.  The
proposition follows by combining \eqref{identity}, Theorem
\ref{mainfunc}, and Theorem \ref{compose1}.
\end{proof} 

\begin{proof}[Proof of Theorem \ref{maincompose}] 
\label{maincomposeproof}  The statement of the Theorem is a parametrized version of the main
  result of \cite{ww:isom}.  We construct a family of quilted surfaces
  $\ol{\SS}^{d,0,0}$ over the bimultiplihedron
  $\ol{\RR}^{d,0,0}$ for which the strip with boundary conditions
  $L_{01},L_{12}$ has varying width between $0$ and $\infty$; this
  differs from the algebraic composition theorem where the width was
  bounded below by a non-zero constant $\rho_0$.  Such a family can be
  obtained from $\ol{\SS}^{d,0}$ by inserting a strip of width $\zeta$
  for $ \zeta$ small and varying, and extended to a family
  $\ol{\SS}^{d,0,0}$ over $\ol{\RR}^{d,0,0}$ by the previous
  procedure.  The argument will be the same as that for the algebraic
  composition in Theorem \ref{compose1}, but the facet of
  $\ol{\RR}^{d,0,0}$ where the seams come together has fibers given by
  quilted disks with a single seam.  The corresponding boundary
  stratum in the moduli space of pseudoholomorphic quilts now
  corresponds to the terms for the geometric composition
  $\Phi(L_{01} \circ L_{12}) $ in the statement of Theorem
  \ref{maincompose}.

To show that the moduli space of pseudoholomorphic quilts over this new
family has the same properties as before requires the arguments of
\cite{ww:isom}:  Near any pseudoholomorphic quilt with seam in $L_{02}$ there
is a unique nearby pseudoholomorphic quilt with small width $\zeta$ of the
strip in between the seams labelled $L_{01}$ and $L_{12}$, by a
parametrized version of the implicit function theorem given in
\cite{ww:isom}.  The operation of replacing the seam with a strip of
width $\zeta$ defines a {\em thickening map}
$$ T_r : \RR^{d,0} \times [0,\delta) \to \ol{\RR}^{d,0,0} .$$
Given a quilt $\ul{u} :\ul{S} \to \ul{M}$ with seam $I$ mapping to
$L_{02}$, we denote by $\ul{S}(\zeta)$ the quilt obtained by replacing
the seam $I$ with a strip $S_\zeta = \R \times [0,\zeta]$.  Let $\gamma:
I \to L_{01} \times_{M_1} L_{12}$ be the unique lift of the
restriction of $\ul{u}$ to $I$.  Let $\ul{u}(\zeta): \ul{S}(\zeta) \to
\ul{M}$ be the map equal to $\ul{u}$ on the components except
$S_\zeta$, and $ \ul{u}(\zeta)(s,t) = \pi_1(\gamma(s))$ on $S_\zeta$.
Given a section $\ul{\xi}$ of $\ul{u}(\zeta)^* T \ul{M}$ with
Lagrangian boundary and seam conditions, a suitable {\em exponential}
$e_{\ul{u}}(\ul{\xi}): \ul{S} \to \ul{M}$ with the Lagrangian boundary
and seam conditions is defined in \cite[20]{ww:isom}.  Pseudoholomorphic
quilts near $\ul{u}$ correspond to pairs $r' \in \SS^{d,0,0},
\ul{v}: \ul{S}_r \to \ul{M}$ with $ r' = T_r(\sigma,\zeta), \quad
\ul{v} = e_{\ul{u}}(\ul{\xi}) $.  These are zeroes of the map
\begin{multline} 
\cF_{\ul{u},\zeta} : T_r \RR^{d,0} \times
\Omega^0(\ul{S}_{T_r(\sigma,\zeta)}^{d,0},\ul{u}(\zeta)^* T \ul{M}) \to
\Omega^{0,1}(\ul{S}_{T_r(\sigma,\zeta)}^{d,0},\ul{u}(\zeta)^* T \ul{M}),
\\ (\sigma, \ul{\xi}) \mapsto \Phi_{\ul{u}(\zeta)}^{-1}
\olp_{T_r(\sigma,\zeta),\ul{J},\ul{K}}
e_{\ul{u}}(\ul{\xi}) \end{multline}
where $\Phi_{\ul{u}(\zeta)}^{-1}$ denotes almost complex parallel
transport.  In a suitable Sobolev completion, the map
$\cF_{\ul{u},\zeta}$ is Fredholm and satisfies uniform quadratic and
error estimates, and has a derivative with uniformly bounded right
inverse.  As in \cite{ww:isom} the argument requires splitting off two
more strips of width $\zeta$ near the seam, and using the folding
construction in \cite[Section 3]{ww:isom}.  The Sobolev completions
are described in \cite[Section 3.1]{ww:isom}, the difference here
being the presence of additional $T_r \RR^{d,0}$ in the map.  This
change does not affect any of the estimates: the variation of the
complex structure on the once-quilted strip does not affect the
complex structure of the corresponding biquilted strips near the seam.
It follows that the additional terms as in \cite{mau:gluing} are
independent of $\zeta$.  To show compactness of the resulting moduli
spaces, the argument of \cite[Section 3.3]{ww:isom} shows that, due to
the monotonicity assumptions, disk, sphere, and figure eight bubbles
cannot occur in the limit $\zeta \to 0$.  Indeed, by energy
quantization such bubbles can only occur at finitely many points, for
any sequence of pseudoholomorphic quilts of index one or zero.  By removal
of singularities one obtains in the limit on the complement of the
bubbling set a configuration with lower energy, hence index.  Such a
configuration cannot occur by the regularity hypotheses.
\end{proof}

The following simple case of Theorem \ref{maincompose} gives
independence of the Fukaya category from all choices:

\begin{corollary} {\rm (Independence of the Fukaya category from choices
up to homotopy equivalence)} Let $M$ be a symplectic background.  The
  Fukaya category $\Fuk(M)$ and generalized Fukaya category $\GFuk(M)$
  are independent of all choices used to construct it, up to \ainfty
  homotopy equivalence.
\end{corollary} 

\begin{proof} Let $\Fuk(M)^0, \Fuk(M)^1$ denote the Fukaya categories
defined using two different sets of perturbation data.  The empty
generalized correspondence gives functors $\Phi(\emptyset)^{01}:
\Fuk(M)^0 \to \Fuk(M)^1$ and vice versa, for different choices of
data.  For the same choice of data the empty correspondence gives the
identity functor by Proposition \ref{emptyset}.  Since the composition
of two empty correspondences is empty, we obtain from Theorem
\ref{maincompose} an \ainfty homotopy between $\Phi(\emptyset)^{01}
\circ \Phi(\emptyset)^{10}$ and the identity, and similarly for
$\Phi(\emptyset)^{10} \circ \Phi(\emptyset)^{01}$.
\end{proof}

\section{Conventions on \ainfty categories}
\label{app:ainfty}

The machinery of \ainfty (homotopy associative) algebras was
introduced by Stasheff \cite{st:ho} as a way of recognizing chains on
loop spaces.  Later Fukaya \cite{fuk:garc} introduced \ainfty
categories as a way of understanding product structures in Lagrangian
Floer cohomology.  In this appendix we describe our conventions for
\ainfty categories, which attempt to follow those of Seidel
\cite{se:bo}.  Other references for this material are Fukaya
\cite{fu:fl1}, Lef\'evre-Hasegawa \cite{le:ai}, and Lyubashenko
\cite{ly:ca}.  Kontsevich-Soibelman \cite{ks:ainfty} introduce a more
conceptual framework in which \ainfty algebras are non-commutative
formal pointed differential-graded manifolds; in particular this
approach gives a conceptual framework for the signs in the definitions
below.  Let $N \in 2\N \cup \{ \infty \}$ be an even positive integer,
  or infinity.  In the case $N = \infty$, we adopt the convention
  $\Z_N = \Z$.

\begin{definition} \label{etc} {\rm (\ainfty categories, functors
    etc.)}  
\begin{enumerate} 
\item {\rm (\ainfty categories)} A {\em $\Z_N$-graded \ainfty category
    $\CC$} consists of the following data:
\begin{enumerate}
\item A {\em class of objects} $\Obj(\CC)$; 
\item for each pair $C_1,C_2 \in \Obj(\CC)$, a $\Z_N$-graded abelian
  { \em group of morphisms} 
$\Hom_\CC(C_1,C_2) = \bigoplus_{i \in
    \Z_N} \Hom_\CC^i(C_1,C_2) ;$
\item
 for each $d \ge 0$ and $(d+1)$-tuple $C_0,\ldots,C_d \in \Obj(\CC)$,
 a multilinear {\em composition map}
$$ \mu^d_\CC:= \mu^d_{\CC,C_0,\ldots, C_d}:  \ \Hom_\CC(C_0,C_1) \otimes \ldots \otimes
\Hom_\CC(C_{d-1},C_d) \to \Hom_\CC(C_0,C_d)[2-d] $$
\end{enumerate}
satisfying the {\em \ainfty-associativity equations}
\begin{multline} \label{ainftyassoc} 
0 = \sum_{n+m \leq d} (-1)^{ n + \sum_{i=1}^n |a_i|}
\mu_\CC^{d-m+1}(a_1,\ldots,a_n, \\ 
\mu_\CC^m(a_{n+1},\ldots,a_{n+m}),
a_{n+m+1},\ldots,a_d)
\end{multline}
for any tuple of homogeneous elements $a_1,\ldots,a_d$.  The signs are
the {\em shifted Koszul signs}, that is, the Koszul signs for the
shifted grading in which the structure maps have degree one as in
Kontsevich-Soibelman \cite{ks:ainfty}.  The element
$\mu_{\CC,C_0}^0(1) \in \Hom(C_0,C_0)[2]$ is the {\em curvature} of
the object $C_0$.  An \ainfty category is {\em flat} if
$\mu_{\CC,C_0}^0(1)$ vanishes for every object $C_0$.  We remark that
we do not assume that our \ainfty categories have units.
%
%
%
%
%
\item {\rm (\ainfty functor)}  
 \label{ainftyfunctor}
 Let $\CC_0,\CC_1$ be flat \ainfty categories.  An {\em \ainfty
   functor} $\F$ from $\CC_0$ to $\CC_1$ consists of the following
 data:
\begin{enumerate}
\item a map $\F: \Obj(\CC_0) \to \Obj(\CC_1)$; and 
\item for any $d \ge 1$ and $d +1$-tuple $C_0,\ldots,C_d \in
\Obj(\CC_0)$, a map
$$\F^d: \ \Hom(C_0,C_1) \times \ldots \Hom(C_{d-1},C_d) \to
\Hom(\F(C_0),\F(C_d))[1-d] $$
\end{enumerate}
such that the following holds:
\begin{multline} \label{faxiom}
 \sum_{i + j \leq d} (-1)^{i + \sum_{j=1}^i |a_j|} \F^{d - j +
     1}(a_1,\ldots,a_i, \mu_{\CC_0}^j(a_{i+1
   },\ldots,a_{i+j}),a_{i+j+1},\ldots,a_d) = \\ \sum_{m \ge 1} 
\sum_{i_1 + \ldots + i_m = d}
   \mu_{\CC_1}^m(\F^{i_1}(a_1,\ldots, a_{i_1}), \ldots,
   \F^{i_m}(a_{i_1 + \ldots + i_{m-1} + 1},\ldots,a_d)).
\end{multline}
\item {\rm (Composition of \ainfty functors)} The {\em composition} of
  \ainfty functors $\F_1,\F_2$ is defined by composition of maps on
  the level of objects, and
\begin{multline} \label{composefunc}
 (\F_1 \circ \F_2)^d(a_1,\ldots,a_d)
 = \sum_{m \ge 1}  \sum_{i_1 + \ldots + i_m =d}
  \F_1^{m}( \F_2^{i_1}(a_1,\ldots,a_{i_1}),  \\ \ldots,
  \F_2^{i_m}(a_{i_1 + \ldots + i_{m-1} +1 },\ldots,a_d)) \end{multline}
on the level of morphisms.  
\item \label{cohfunc} {\rm (Cohomology functor)} Any \ainfty functor
  $\F: \cC_1 \to \cC_2$ between flat \ainfty categories $\cC_1,\cC_2$
  defines an ordinary functor $H(\F): H(\cC_1) \to H(\cC_2) $ acting
  in the same way as $\F$ on objects and on morphisms of fixed degree
  by $H(\F)([a]) = [\F(a)]$.
\item {\rm (\ainfty natural transformations)}  
  Let $\F_1,\F_2: \CC_0 \to \CC_1$ be \ainfty functors between flat
  \ainfty categories.  A {\em pre-natural transformation} $\TT$ from
  $\F_1$ to $\F_2$ consists of the following data:
For each $d \ge 0$ and $d+1$-tuple of objects $C_0,\ldots,C_d \in
\Obj(\CC_0)$ a multilinear map
\begin{equation} \label{TTd} \TT^d(C_0,\ldots,C_d): \ \Hom(C_0,C_1) \times \ldots \times
\Hom(C_{d-1},C_d) \to \Hom(\F_1(C_0),\F_2(C_d))[|T| - d] .\end{equation}
Let $\Hom(\F_1,\F_2)$ denote the space of pre-natural transformations
from $\F_1$ to $\F_2$.  Define a differential on $\Hom(\F_1,\F_2)$ by
\begin{multline} \label{mu1}
 (\mu^1_{\Hom(\F_1,\F_2)} \TT)^d (a_1,\ldots,a_d) = \sum_{k,m \ge 1}
  \sum_{i_1 + \ldots + i_m = d} (-1)^\dagger \mu^m_{\cC_2}(
  \F_1^{i_1}(a_1,\ldots,a_{i_1}), \F_1^{i_2}(a_{i_1 + 1},\ldots),
  \ldots, \\ \TT^{i_k}(a_{i_1 + \ldots + i_{k-1} + 1},\ldots, a_{i_1 +
    \ldots + i_k}), \F_2^{i_{k+1}}( a_{i_1 + \ldots + i_k + 1},\ldots,
  ) ,\ldots, \F_2^{i_m}(a_{i_1 + \ldots + i_{m-1} + 1},\ldots, a_d)) \\ - \sum_{i,e}
  (-1)^{i + \sum_{j=1}^i |a_j| + |\TT| - 1} \TT^{d - e +
    1}(a_1,\ldots,a_i, \mu^e_{\cC_1}(a_{i+1},\ldots,
  a_{i+e}),a_{i+e+1},\ldots,a_d) \end{multline}
where 
$ \dagger = (|\TT|-1)( |a_1| + \ldots + |a_{i_1 + \ldots +
  i_{k-1}}| - i_1 - \ldots - i_{k-1}) .$
A {\em natural transformation} is a closed pre-natural transformation.
\item {\rm (Composition of natural transformations)} Given two
  pre-natural transformations
$\TT_1:\F_0 \to \F_1, \quad  \TT_2:\F_1 \to
  \F_2 $
  as above define $\mu^2(\TT_1,\TT_2):\ \F_0\to\F_2$ by
\begin{multline}  \label{T2T1}
 (\mu^2(\TT_1,\TT_2))^d(a_1,\ldots,a_d) =
\sum_{m,k<l} \sum_{i_1 + \ldots + i_m = d} 
(-1)^\ddag \mu^m_{\cC_2}( \F_0^{i_1}(a_1,\ldots,a_{i_1}),
\ldots, \F_0^{i_{k-1}}(\ldots),
 \\
\TT_1^{i_k}(a_{i_1 + \ldots + i_{k-1} + 1},\ldots, a_{i_1 + \ldots + i_k}),
\F_1^{i_{k+1}}(\ldots),\ldots, \F_1^{i_{l-1}}(\ldots), 
\\
\TT_2^{i_l} (a_{i_1 + \ldots + i_{l-1} + 1},\ldots, a_{i_1 + \ldots + i_{l}}),
\F_2^{i_{l+1}}(\ldots),\ldots, \F_2^{i_m}(a_{i_1 + \ldots + i_{m-1} + 1},\ldots,a_d)) 
\end{multline}
where 
$$ \ddag = \sum_{i = 1}^{i_1 + \ldots + i_{k-1}} ( |\TT_1| - 1) ( |a_i| - 1) + 
\sum_{i = 1}^{i_1 + \ldots + i_{l-1}} ( |\TT_2| - 1) ( |a_i| - 1)  .$$
With $\CC_0,\CC_1$ flat \ainfty categories let $\Fun(\CC_0,\CC_1)$
denote the space of \ainfty functors from $\CC_0$ to $\CC_1$, with
morphisms given by pre-natural transformations.  The higher
compositions give $\Fun(\CC_0,\CC_1)$ the structure of an \ainfty
category \cite[10.17]{fu:fl1}, \cite[8.1]{le:ai}, \cite[Section
1d]{se:bo}.
\item {\rm (Cohomology natural transformations)} Any \ainfty natural
  transformation $\TT: \F_1 \to \F_2$ induces a natural transformation
  of the corresponding homological functors $H(\F_1) \to H(\F_2)$.
\item {\rm (\ainfty homotopies)} Suppose that $\F_1,\F_2: \CC_0 \to
  \CC_1$ are functors that act the same way on objects.  A {\em
    homotopy} from $\F_1$ to $\F_2$ is a pre-natural transformation
  $\TT \in \Hom(\F_1,\F_2)$ of degree $-1$ such that
\begin{equation} \label{homotopy} \F_1 - \F_2
= \mu^1(\TT)  \end{equation}  
where $\mu^1(\TT)$ is defined by \eqref{mu1}.  Note that the
assumption on degree substantially simplifies the signs.  Homotopy of
\ainfty functors is an equivalence relation \cite[p.15]{se:bo}.
\item {\rm (Composition of homotopies)} Given homotopies $\TT_1$ from
  $\F_0$ to $\F_1$, and $\TT_2$ from $\F_1$ to $\F_2$, the sum
\begin{equation} \label{composehom}
\TT_2 \circ \TT_1: = \TT_1 + \TT_2 + \mu^2(\TT_1,\TT_2) \in
\Hom(\F_0,\F_2) \end{equation}
is a homotopy from $\F_0$ to $\F_2$.
\item {\rm (Quasi-isomorphisms)} Two \ainfty functors $\F_1,F_2$ are
  {\em quasiisomorphic} if there exist natural transformations
  $\TT_{12}$ from $\F_1$ to $\F_2$ and $\TT_{21}$ from $\F_2$ to
  $\F_1$ such that $\TT_{12} \circ \TT_{21}$ and $\TT_{21} \circ
  \TT_{12}$ are cohomologous to the identity natural transformation on
  $\F_1$ resp. $\F_2$.
\end{enumerate} 
\end{definition}

\def\cprime{$'$} \def\cprime{$'$} \def\cprime{$'$} \def\cprime{$'$}
  \def\cprime{$'$} \def\cprime{$'$}
  \def\polhk#1{\setbox0=\hbox{#1}{\ooalign{\hidewidth
  \lower1.5ex\hbox{`}\hidewidth\crcr\unhbox0}}} \def\cprime{$'$}
  \def\cprime{$'$}

\end{document}